\documentclass[reqno,11pt]{amsart}
\pdfoutput1

\usepackage{amssymb,color,hyperref,mathrsfs,stmaryrd}
\usepackage{amsmath}
\usepackage{mathtools}

\usepackage[usenames,dvipsnames]{xcolor}

\usepackage[curve,matrix,arrow]{xy}

\numberwithin{equation}{section}
\newtheorem{theorem}[equation]{Theorem}
\newtheorem{lemma}[equation]{Lemma}

\newtheorem{prop}[equation]{Proposition}
\newtheorem{mprop}{Proposition}
\newtheorem{cor}[equation]{Corollary}
\newtheorem{Th}{Theorem}

\theoremstyle{definition}
\newtheorem{mdefinition}{Definition}

\newtheorem{hyp}[equation]{Hypothesis}

\newtheorem{rmk}[equation]{Remark}
\newtheorem{question}[equation]{Question}

\newtheorem{definition}[equation]{Definition}
\newtheorem{example}{Example}
\newtheorem{ex}[equation]{Example}

\newcommand{\F}{\mathcal{F}}

\newcommand{\N}{\mathcal{N}}
\renewcommand{\H}{\mathcal{H}}
\newcommand{\T}{\mathcal{T}}
\newcommand{\ob}{\operatorname{ob}}
\renewcommand{\L}{\mathcal{L}}
\newcommand{\wL}{\widetilde{\mathcal{L}}}

\newcommand{\Hom}{\operatorname{Hom}}
\newcommand{\Aut}{\operatorname{Aut}}
\newcommand{\Inn}{\operatorname{Inn}}
\newcommand{\Syl}{\operatorname{Syl}}

\newcommand{\m}{\mathcal}
\newcommand{\ov}{\overline}

\newcommand{\D}{\mathbf{D}}
\newcommand{\id}{\operatorname{id}}
\newcommand{\One}{\operatorname{\mathbf{1}}}
\newcommand{\W}{\mathbf{W}}
\renewcommand{\epsilon}{\varepsilon}

\def \<{\langle }
\def \>{\rangle }

\renewcommand{\phi}{\varphi}

\title[Subcentric Linking Systems]{Subcentric Linking Systems}

\author[E.~Henke]{Ellen Henke}
\address{Institute of Mathematics, 
University of Aberdeen, Fraser Noble Building, Aberdeen AB24 3UE, U.K.}
\email{ellen.henke@abdn.ac.uk}

\thanks{For part of this research, the author was supported by the Danish National Research
  Foundation through the Centre for Symmetry and
  Deformation (DNRF92).}

\begin{document}

\begin{abstract}
Linking systems are crucial for studying the homotopy theory of fusion systems, but are also of interest from an algebraic point of view. We propose a definition of a linking system associated to a saturated fusion system which is more general than the one currently in the literature and thus allows a more flexible choice of objects of linking systems. More precisely, we define subcentric subgroups of fusion systems in a way that every quasicentric subgroup of a saturated fusion system is subcentric. Whereas the objects of linking systems in the current definition are always quasicentric, the objects of our linking systems only need to be subcentric. We prove that, associated to each saturated fusion system $\F$, there is a unique linking system whose objects are the subcentric subgroups of $\F$. Furthermore, the nerve of such a subcentric linking system is homotopy equivalent to the nerve of the centric linking system associated to $\F$. We believe that the existence of subcentric linking systems opens a new way for a classification of fusion systems of characteristic $p$-type. The various results we prove about subcentric subgroups give furthermore some evidence that the concept is of interest for studying extensions of linking systems and fusion systems.
\end{abstract}

\maketitle

\section{Introduction}

Centric linking systems associated to fusion systems were introduced by Broto, Levi and Oliver \cite{BLO2} to be able to study $p$-completed classifying spaces of fusion systems. The existence and uniqueness of a centric linking system associated to each saturated fusion system was however a conjecture for many years until it was proved by Chermak \cite{Chermak:2013} using the classification of finite simple groups. Chermak's proof was reformulated by Oliver \cite{Oliver:2013} and, building on this reformulation, a recent result of Glauberman and Lynd \cite{Glauberman/Lynd} removes the dependence of the proof on the classification. It is an advantage in many contexts to work with linking systems rather than with fusion systems, but it often presents a problem that centric linking systems do not form a category in a meaningful way. Different notions of linking systems  were introduced to allow a more flexible choice of objects making it at least in special cases possible to study extensions of linking systems. So Broto, Castellana, Grodal, Levi and Oliver \cite{BCGLO1} introduced \textit{quasicentric linking systems} and, much later, Oliver \cite{O4} introduced a general notion of a \textit{linking system} providing an axiomatic setup for the full subcategories of quasicentric linking systems studied before. \textit{Transporter systems} as defined by Oliver and Ventura \cite{OV1} give an even more general framework. The main purpose of this paper is to suggest a new notion of a linking system, allowing a more flexible choice of objects than in the existing notion. We prove furthermore some results indicating the usefulness of this new definition.

\smallskip

We write our functions usually on the right hand side. Accordingly, we always compose morphisms in categories from the left to the right. 

\smallskip

\textbf{Throughout, $p$ is a prime, $S$ is a finite $p$-group, and $\F$ is a fusion system over $S$.}

\smallskip

We refer the reader to \cite[Part~I]{Aschbacher/Kessar/Oliver:2011} for an introduction to fusion systems. We adapt the notation and terminology from there. In addition, we write $\F^f$ for the set of fully $\F$-normalized subgroups of $S$. Recall also that a  subgroup $Q\leq S$ is called \textit{quasicentric} in $\F$ if, for any fully centralized $\F$-conjugate $P$ of $Q$, $C_\F(P)=\F_{C_S(P)}(C_S(P))$. The set of quasicentric subgroups is denoted by $\F^q$. 

\smallskip

The objects of a linking system associated to $\F$ in the sense of Oliver are always quasicentric subgroups. The objects of linking systems in our new definition only need to satisfy a weaker condition. Namely, they are \textit{subcentric} subgroups as defined next.

\begin{mdefinition}
A subgroup $Q\leq S$ is said to be \textit{subcentric} in $\F$ if, for every fully normalized $\F$-conjugate $P$ of $Q$, $O_p(N_\F(P))$ is centric in $\F$. Write $\F^s$ for the set of subcentric subgroups of $\F$.
\end{mdefinition}

Recall that $\F$ is said to be \textit{constrained} if $\F$ is saturated and $C_S(O_p(\F))\leq O_p(\F)$. As we show in detail in Lemma~\ref{subcentricEquiv}, assuming $\F$ is saturated, a subgroup $Q\leq S$ is subcentric if and only if for some (and thus for every) fully normalized $\F$-conjugate $P$ of $Q$, $N_\F(P)$ is constrained. Similarly, $Q$ is subcentric if and only if for some (and thus for every) fully centralized $\F$-conjugate $P$ of $Q$, $C_\F(P)$ is constrained. It follows that every quasicentric subgroup is subcentric. Thus, provided $\F$ is saturated, we have the following inclusions:

$$\F^{cr}\subseteq \F^c\subseteq \F^q\subseteq \F^s$$

While linking systems are crucial for studying the homotopy theory of fusion systems, there is some evidence that linking systems and transporter systems are also useful from an algebraic point of view. Chermak \cite{Chermak:2013} introduced with \textit{localities} a concept which in a certain sense is equivalent to the concept of a transporter systems, but has a more group-like flavor. Chermak defines a partial group to be a set $\L$ together with a product which is only defined on certain words in $\L$, and with an inversion map $\L\rightarrow \L$ which is an involutory bijection, subject to certain axioms. So the product of a partial group is a map $\Pi\colon\D\rightarrow\L$ where $\D$ is a set of words in $\L$. A locality is a triple $(\L,\Delta,S)$ such that $\L$ is a partial group which is finite as a set, $S$ is a $p$-subgroup of $\L$, and $\Delta$ is a set of subgroups of $S$, again subject to certain axioms. The set $\Delta$ is called the set of objects of the locality $(\L,\Delta,S)$. While Chermak defined localities first in the context of his proof of the existence and uniqueness of centric linking systems, he is currently developing a local theory of localities; see \cite{Chermak:2015}. We refer the reader to Section~\ref{LocalitiesSection0} for a brief introduction to localities and to \cite{Chermak:2013} and \cite{Chermak:2015} for a detailed treatment of the subject.

\smallskip

Let $(\L,\Delta,S)$ be a locality. Given the group-like nature of $\L$, there is a natural notion of conjugation in $\L$, even though conjugation is not always defined, since products in partial groups are only defined on certain words in $\L$. For $P\subseteq\L$, the normalizer $N_\L(P)$ consists of all $f\in\L$ such that the conjugate $P^f$ is defined and equals $P$. It turns out that, for any $P\in\Delta$, the normalizer $N_\L(P)$ is a subgroup of $\L$ and thus forms a finite group. The fusion system $\F_S(\L)$ is the fusion system over $S$ generated by the conjugation maps between the subgroups of $S$.  We say that a locality $(\L,\Delta,S)$ is a locality over $\F$ if $\F=\F_S(\L)$. One can always construct a transporter system $\T(\L,\Delta)$ associated to $\F_S(\L)$ whose set of objects is $\Delta$. Moreover, every transporter system associated to $\F$ is isomorphic to a transporter system which comes in this way from a locality over $\F$. It follows from the construction of $\T(\L,\Delta)$ that $\Aut_{\T(\L,\Delta)}(P)\cong N_\L(P)$ for every $P\in\Delta$. For more details on the connection between transporter systems and localities we refer the reader to Subsection~\ref{TransporterSystemsSubsection}. We proceed now with our second crucial definition.

\begin{mdefinition}
\noindent\begin{itemize}
\item A finite group $G$ is said to be of \textit{characteristic $p$} if $C_G(O_p(G))\leq O_p(G)$.
  \item Define a locality $(\L,\Delta,S)$ to be of \textit{objective characteristic $p$} if, for any $P\in\Delta$, the group $N_\L(P)$ is of characteristic $p$. A locality $(\L,\Delta,S)$ is called a \textit{linking locality}, if $\F_S(\L)^{cr}\subseteq \Delta$ and $(\L,\Delta,S)$ is of objective characteristic $p$.
  \item Let $\T$ be a transporter system associated to $\F$. Then $\T$ is said to be of \textit{objective characteristic $p$} if $\Aut_\T(P)$ is a group of characteristic $p$ for every object $P$ of $\T$. Moreover, $\T$ is called a \textit{linking system}, if $\F^{cr}\subseteq \ob(\T)$ and $\T$ is of objective characteristic $p$.
  \item A \textit{subcentric linking locality} over $\F$ is a linking locality $(\L,\F^s,S)$ over $\F$. Similarly, a \textit{centric linking locality} over $\F$ is a linking locality $(\L,\F^c,S)$ over $\F$, and a \textit{quasicentric linking locality} over $\F$ is a linking locality $(\L,\F^q,S)$ over $\F$. 
 \item A linking system $\T$ associated to $\F$ is called a \textit{subcentric linking system} if $\ob(\T)=\F^s$.
\end{itemize}
\end{mdefinition}
 
In the following proposition, we summarize some basic but important properties of linking systems and linking localities. Moreover, we explain the connection between our notion of a linking system and the one currently in the literature.
By a \textit{model} for the fusion system $\F$ we always mean a finite group $G$ of characteristic $p$ such that $S\in\Syl_p(G)$ and $\F_S(G)=\F$. As shown in \cite{BCGLO1}, there exists a model for $\F$ if and only if $\F$ is constrained. Moreover, if a model exists, then it is unique up to isomorphism. 

\begin{mprop}\label{ex0}
Let $(\L,\Delta,S)$ be a locality over $\F$, and let $\T$ be a transporter system associated to $\F$. Then the following hold.
\begin{itemize}
 \item[(a)] $\T(\L,\Delta)$ is a linking system if and only if $(\L,\Delta,S)$ is a linking locality.
 \item[(b)] If $(\L,\Delta,S)$ is of objective characteristic $p$, then $\Delta\subseteq \F^s$ and, for any $P\in\Delta\cap\F^f$, $N_\L(P)$ is a model for $N_\F(P)$. Similarly, if $\T$ is of objective characteristic $p$ then $\ob(\T)\subseteq \F^s$ and, for any $P\in\ob(\T)\cap\F^f$, the group $\Aut_\T(P)$ is isomorphic to a model for $N_\F(P)$.
 \item[(c)] Assume $\Delta\subseteq \F^q$. Then $C_\L(P)=C_S(P)O_{p^\prime}(C_\L(P))$ for every $P\in\Delta\cap\F^f$. As a consequence, $(\L,\Delta,S)$ is of objective characteristic $p$ if and only if $C_\L(P)$ is a $p$-group for every $P\in\Delta$. If  $\ob(\T)\subseteq \F^q$, then $\T$ is a linking system in the sense defined above if and only if it is a linking system in the sense of Oliver \cite[Definition~3]{O4}. In particular, every linking system in Oliver's definition is a linking system in our definition.
 \item[(d)] If $\Delta\subseteq \F^c$, then $(\L,\Delta,S)$ is of objective characteristic $p$ if and only if $(\L,\Delta,S)$ is a $\Delta$-linking system in the sense of Chermak \cite{Chermak:2013}, i.e. if and only if $C_\L(P)\leq P$ for every $P\in\Delta$. If $\Delta=\F^c$, then $(\L,\Delta,S)$ is a linking locality in our definition if and only if $(\L,\Delta,S)$ is a centric linking system in the sense of Chermak \cite{Chermak:2013}.\\
If $\ob(\T)=\F^c$, then $\T$ is a linking system in the sense defined above if and only if it is a centric linking system in the sense of \cite[Definition~1.7]{BLO2}.
\end{itemize}
\end{mprop}

Assume now that $\F$ is saturated. Suppose we are given a set $\Delta$ of subgroups such that $\F^{cr}\subseteq\Delta\subseteq \F^q$ and such that $\Delta$ is closed under $\F$-conjugation and with respect to overgroups. It follows from the existence and uniqueness of centric linking systems combined with \cite[Theorem~A, Proposition~3.12]{BCGLO1} that there is a linking system with object set $\Delta$ associated to $\F$, and that such a linking system is unique up to isomorphism. Moreover, the nerve of the linking system does not depend on the object set $\Delta$. In particular, quasicentric linking systems exist and are unique up to isomorphism, and the nerve of a quasicentric linking system is homotopy equivalent to the nerve of a centric linking system. Except for the statement about nerves, a formulation of these results and an algebraic proof using the methods in \cite{Chermak:2013} was given by Chermak in unpublished notes before the idea to define subcentric subgroups arose. We similarly give a version for subcentric linking systems. We also include a statement about nerves, which follows from a result of Oliver and Ventura \cite[Proposition~4.7]{OV1} generalizing the arguments in \cite{BCGLO1}. The crucial property here is that the radical objects of a linking system $\T$ (i.e. the objects $P$ of $\T$ with $O_p(\Aut_\T(P))\cong P$) are precisely the elements of $\F^{cr}$.

\begin{Th}\label{MainThm1}
Let $\F$ be saturated.
\begin{itemize}
\item [(a)] Let $\F^{cr}\subseteq \Delta\subseteq \F^s$ such that $\Delta$ is closed under $\F$-conjugation and with respect to overgroups. Then there exists a linking locality over $\F$ with object set $\Delta$, and such a linking locality is unique up to a rigid isomorphism. Similarly, there exists a linking system $\T$ associated to $\F$ whose set of objects is $\Delta$, and such a linking system is unique up to an isomorphism of transporter systems. Moreover, the nerve $|\T|$ is homotopy equivalent to the nerve of a centric linking system associated to $\F$.
\item [(b)] The set $\F^s$ is closed under $\F$-conjugation and with respect to overgroups. In particular, there exists a subcentric linking locality over $\F$ which is unique up to a rigid isomorphism, and there exists a subcentric linking system associated to $\F$ which is unique up to an isomorphism of transporter systems.
\end{itemize}
\end{Th}

Recall here from \cite{Chermak:2013} that a \textit{rigid isomorphism} between localities $(\L,\Delta,S)$ and $(\L^*,\Delta,S)$ with the same set of objects is an isomorphism $\L\rightarrow \L^*$ of partial groups which restricts to the identity on $S$.

\smallskip

As we will explain next, the existence of subcentric linking localities seems to be important because it leads to a useful setup for a classification of fusion systems of characteristic $p$-type. Recall that a finite group $G$ is said to be of \textit{characteristic $p$-type} (or of \textit{local characteristic $p$}), if every $p$-local subgroup (i.e. every normalizer of a non-trivial $p$-subgroup) is of characteristic $p$. Similarly, if $\F$ is saturated, then $\F$ is said to be of characteristic $p$-type if, for every non-trivial fully $\F$-normalized subgroup $P\leq S$, $N_\F(P)$ is constrained. The main examples of groups of characteristic $p$-type are the finite groups of Lie type in defining characteristic $p$. Moreover, if a finite group is of characteristic $p$-type then its fusion system turns out to be  of characteristic $p$-type whereas the converse is not true in general.

\smallskip

Recall that a subgroup is subcentric if and only if it is $\F$-conjugate to a fully $\F$-normalized subgroup whose normalizer is constrained. Thus, $\F$ is of characteristic $p$-type if and only if every non-trivial subgroup of $S$ is subcentric. So supposing that $\F$ is of characteristic $p$-type and $(\L,\Delta,S)$ is a subcentric linking locality over $\F$, every non-trivial subgroup of $S$ is an element of $\Delta$, and the normalizer $N_\L(P)$ of any non-trivial subgroup $P$ of $S$ is a finite group of characteristic $p$. Hence, ``locally'' the partial group $\L$ looks very much like a finite group of characteristic $p$-type. On the other hand, every group of characteristic $p$-type leads in an elementary way to a linking locality of this kind:

\begin{example}\label{CharpType}
Let $G$ be a group of characteristic $p$-type and let $S\in\Syl_p(G)$. Let $\Delta$ be the set of non-trivial subgroups of $S$. Let $\L_\Delta(G)$ be the set of all elements $g\in G$ with $S\cap S^g\neq 1$. Moreover, define a partial product on $\L_\Delta(G)$ by taking the restriction of the (multivariable) product on $G$ to the set $\D$ of all words $(g_1,\dots,g_n)$ such that $g_i\in G$ and there exist elements $P_0,\dots,P_n\in\Delta$ with $P_{i-1}^{g_i}=P_i$ for $i=1,\dots,n$. Define an inversion map on $\L_\Delta(G)$ by taking the restriction of the inversion map on $G$ to the set $\L_\Delta(G)$. Then $(\L_\Delta(G),\Delta,S)$ is a locality by \cite[Example/Lemma~2.10]{Chermak:2013}. 
Moreover, $N_{\L_\Delta(G)}(P)=N_G(P)$ is of characteristic $p$ for all $P\in\Delta$. Hence, $(\L_\Delta(G),\Delta,S)$ is a subcentric linking locality for $\F_S(G)$
\end{example}

Previous treatments of fusion systems of characteristic $p$-type (as for example in \cite{Aschbacher:2010}, \cite{Aschbacher:2013a}, \cite{Aschbacher:2013b} and \cite{Henke:2011}) have used the existence of models for normalizers of fully normalized subgroups. Supposing that $\F$ is a fusion system of characteristic $p$-type, this involves moving from an arbitrary non-trivial subgroup of $S$ to a fully normalized $\F$-conjugate whose normalizer can then be realized by a model. This process of moving between different $\F$-conjugates often complicates the arguments. Such technical difficulties can be avoided when working with a subcentric linking locality $(\L,\Delta,S)$ over $\F$, because then, for any non-trivial subgroup $P$ of $S$, the normalizer $N_\L(P)$ is a finite group of characteristic $p$. We thus believe that subcentric linking localities allow a much more canonical translation of the arguments used to classify groups of characteristic $p$-type. Building on the ongoing program of Meierfrankenfeld, Stellmacher, Stroth to classify groups of local characteristic $p$, one can hope to achieve a classification of fusion systems of characteristic $p$-type once this program is complete. 

\smallskip

It might be possible to give a unifying approach to the classification of fusion systems of characteristic $p$-type and of groups of characteristic $p$-type whilst avoiding to use Theorem~\ref{MainThm1} and the theory of fusion systems to prove classification theorems for groups of characteristic $p$-type. We suggest to proceed as follows: In a first step one proves a classification theorem for a linking locality $(\L,\Delta,S)$ where $\Delta$ is the set of non-trivial subgroups of $S$. Then in a second step one separately deduces from that a corresponding classification theorem for fusion systems of characteristic $p$-type (using the existence of subcentric linking systems), and for groups of characteristic $p$-type (working with the locality $(\L_\Delta(G),\Delta,S)$ introduced in Example~\ref{CharpType}). A similar approach should be possible for groups and fusion systems which are not of characteristic $p$-type, but satisfy a weaker condition like for example being of \textit{parabolic characteristic $p$}. In Remark~\ref{ParabolicChar} we outline a possible approach after constructing linking localities and localities of objective characteristic $p$ coming from arbitrary finite groups in Section~\ref{GroupLocalities}.

\smallskip

We think that the existence of linking localities and linking systems with subcentric objects is also important for another reason. Namely, it seems that the more flexible choice of objects facilitates the study of extensions and of ``maps'' between linking systems and linking localities in the spirit of \cite{BCGLO2}, \cite{O4}, \cite{OV1}, \cite{AOV1}. This might lead to a local theory of subcentric linking localities similar to the local theory of fusion systems \cite{Aschbacher:2011}, which Aschbacher developed in analogy to the local theory of groups. If this is true, then subcentric linking localities could also lead to simplifications in the classification of simple fusion systems of component type. We continue by stating some technical results which could form the basis of a local theory of subcentric linking localities. In particular, in the next two propositions, we state some relations between the subcentric subgroups of $\F$ and the subcentric subgroups of local subsystems and certain normal subsystems. The proof of these propositions can be found in Section~\ref{SubcentricProperties}.

\begin{mprop}\label{SubcentricProperties1}
If $\F$ is saturated, then the following hold:
\begin{itemize}
 \item [(a)] Let $R\unlhd\F$ and $P\leq S$. Then $PR\in\F^s$ if and only if $P\in\F^s$.
 \item [(b)] Let $Z\leq Z(\F)$ and $P\leq S$. Then $P\in\F^s$ if and only if $PZ/Z$ is subcentric in $\F/Z$.
 \item [(c)] If $Q\in\F^f$ and $P\in N_\F(Q)^s$, then $PQ\in\F^s$. More generally, if $Q\leq S$ and $K\leq\Aut(Q)$ are such that $Q$ is fully $K$-normalized, then $PQ\in\F^s$ for every $P\in N^K_\F(Q)^s$.
\item [(d)] For any $Q\in\F^f$, we have $\{P\in\F^s\colon P\leq N_S(Q)\}\subseteq N_\F(Q)^s$. More generally, for every $Q\leq S$ and every $K\unlhd \Aut_\F(Q)$ such that $Q$ is fully $K$-normalized, we have $\{P\in\F^s\colon P\leq N_S^K(Q)\}\subseteq N_\F^K(Q)^s$.
\end{itemize}
\end{mprop}

The property analogous to (b) for quasicentric subgroups was proved by Broto, Castellana, Grodal, Levi and Oliver in \cite[Lemma~6.4(b)]{BCGLO2}. Building on this result, the authors show that a quasicentric linking system for $\F/Z$ ($Z\leq Z(\F)$) can be constructed as a ``quotient'' of a quasicentric linking system associated to $\F$. A similar construction can be carried out in the world of localities. We prove this in Proposition~\ref{LocalityQuotientModCentral} not only for quasicentric linking localities, but also correspondingly for subcentric linking localities and for arbitrary linking localities. Results corresponding to (c) and (d) are also true for centric and quasicentric subgroups; see Lemma~\ref{KNormSubcentric2a}. As we explain in more detail in Section~\ref{pLocalInclusion}, property (c) implies that a subcentric linking locality over $N_\F^K(Q)$ is contained in a subcentric linking locality over $\F$ such that the inclusion map is a homomorphism of partial groups. This leads also to a functor from the subcentric linking system of $N_\F^K(Q)$ to the subcentric linking system of $\F$. Similar results hold for centric and quasicentric linking systems and linking localities. 

\smallskip

We now turn attention to weakly normal subsystems.

\begin{mprop}\label{SubcentricProperties2}
Let $\F$ be saturated and let $\m{E}$ be a weakly normal subsystem of $\F$ over $T$. Then the following hold:
\begin{itemize}
 \item [(a)] The set $\m{E}^s$ is invariant under $\F$-conjugation.
 \item [(b)] For every $P\in\F^s$ with $P\leq T$, $P\in\m{E}^s$.
 \item [(c)] If $\m{E}$ is normal in $\F$ of index prime to $p$, then $\m{E}^s=\F^s$.
 \item [(d)] If $\m{E}$ is normal in $\F$ of $p$-power index, then $\m{E}^s=\{P\in\F^s\colon P\leq T\}$.
 \item [(e)] If $R\unlhd \F$ and $K\unlhd \Aut_\F(R)$ then $N_\F^K(R)^s=\{P\in\F^s\colon P\leq N_S^K(R)\}$. In particular, $C_\F(R)^s=\{P\in\F^s\colon P\leq C_S(R)\}$.
\end{itemize}
\end{mprop}

Corresponding statements to (a) and (b) are also true for centric and quasicentric subgroups. Property (c) is clearly also true if one considers centric subgroups rather than subcentric subgroups, and a statement corresponding to (d) is true for quasicentric subgroups by \cite[Theorem~4.3]{BCGLO2}. It is shown in \cite[Theorem~5.5]{BCGLO2} that, given a subsystem $\m{E}$  of index prime to $p$, a centric linking system associated to $\m{E}$ can be naturally constructed from the centric linking system associated to $\F$. Similarly, it is shown in \cite[Theorem~4.4]{BCGLO2} that a quasicentric linking system of a subsystem of $p$-power index can be obtained from a quasicentric linking system associated to $\F$. Property (e) fails for centric and quasicentric subgroups as it is stated, but if $\Inn(R)\leq K$, it is true that every centric or quasicentric subgroup of $N_\F^K(R)$ which contains $R$ is $\F$-centric or $\F$-quasicentric respectively, and this is enough for many purposes. In \cite[Definition~1.27]{AOV1}, Andersen, Oliver and Ventura define normal linking systems. The results we summarized enable them to associate normal pairs of linking systems to $(\m{E},\F)$, if $\m{E}$ is a weakly normal subsystem of $\F$ of index prime to $p$, or of $p$-power index, or if $\m{E}=N_\F^K(R)$ for some normal subgroup $R\unlhd \F$ and $\Inn(R)\leq K\unlhd\Aut(Q)$; see \cite[Proposition~1.31]{AOV1}. Andersen, Oliver and Ventura \cite{AOV1} define also the reduction of a fusion system $\F$ by starting with  $\F_0:=C_\F(O_p(\F))/Z(O_p(\F))$ and then alternately taking $\F_i=O^p(\F_{i-1})$ and $\F_i=O^{p^\prime}(\F_{i-1})$ for any positive integer $i$ until the process terminates. Note that Proposition~\ref{SubcentricProperties1}(b) together with Proposition~\ref{SubcentricProperties2}(c),(d),(e) gives a very clean connection between the subcentric subgroups of $\F$ and the subcentric subgroups of the reduction of $\F$. Therefore it could be an advantage to work with subcentric linking systems rather than with centric and quasicentric linking systems in this context. 

\smallskip

We now turn attention to normal subsystems which do not fulfill any additional properties, and we prove that its subcentric subgroups are still closely related to subcentric subgroups of the entire fusion system.

\begin{Th}\label{SubcentricEF}
 Let $\F$ be a saturated fusion system on a finite $p$-group $S$, and let $\m{E}$ be a normal subsystem of $\F$. Then for every subcentric subgroup $P$ of $\m{E}$, $PC_S(\m{E})$ is subcentric in $\F$.
\end{Th}

Here $C_S(\m{E})$ is the subgroup introduced by Aschbacher \cite[Chapter~6]{Aschbacher:2011}. It is the largest subgroup $X$ of $S$ with $\m{E}\subseteq C_\F(X)$. If $\m{E}$ is realized by a partial normal subgroup, then we prove that $C_S(\m{E})$ is indeed easy to describe in the locality:

\begin{mprop}\label{CENThm}
 Let $\m{E}$ be a normal subsystem of $\F$ over $T$ and let $(\L,\Delta,S)$ be a linking locality over $\F$. Suppose there exists a partial normal subgroup $\N$ of $\L$ such that $S\cap\N=T$ and $\m{E}=\F_T(\N)$. Then $C_S(\m{E})=C_S(\N)$.
\end{mprop}

Here $\F_T(\N)$ is the smallest fusion system on $T$ containing all conjugation maps by elements of $\N$ between subgroups of $T$. Supposing that $\F$ is saturated and $(\L,\Delta,S)$ is a linking locality over $\F$ with $\F^q\subseteq\Delta\subseteq\F^s$, it was shown by Chermak and the author of this paper \cite{Chermak/Henke:2017}  that indeed every normal subsystem is of the form $\F_{S\cap\N}(\N)$ for a unique partial normal subgroup $\N$ of $\L$, and that this leads to a one-to-one correspondence between the normal subsystems of $\F$ and the partial normal subgroups of $\L$. In particular, if $(\L,\Delta,S)$ is a subcentric linking locality over $\F$, then a normal subsystem $\m{E}$ of $\F$ is realized by a partial normal subgroup $\N$ of $\L$. This situation is explored further in Subsection~\ref{NormalInclusions}. Using Theorem~\ref{SubcentricEF} and Proposition~\ref{CENThm} we show that a subcentric linking locality for $\m{E}$ is contained in $\L$, and that the inclusion map is a homomorphism of partial groups. This leads to a functor from the subcentric linking system of $\m{E}$ to the subcentric linking system of $\F$. This functor maps every object $P\in\m{E}^s$ to $PC_S(\m{E})\in\F^s$.

\smallskip

%The connection between the normal $p$-subgroups of a fusion system and between the normal $p$-subgroups of a corresponding locality of characteristic $p$ is given in the next proposition, which is short and elementary to prove. We stress that $\F$ is not necessarily saturated here.
The following Proposition is needed in the proof of Theorem~\ref{MainThm1}. If $\F$ is saturated and $(\L,\Delta,S)$ is a linking locality over $\F$, then the statement can be considered as a particular case of the correspondence between the normal subsystems of $\F$ and the partial normal subgroups of $\L$.

\begin{mprop}\label{NormLF}
 Let $(\L,\Delta,S)$ be a locality over $\F$ of objective characteristic $p$. Then a subgroup $Q\leq S$ is normal in $\F$ if and only if $\L=N_\L(Q)$. Similarly, $Q\leq Z(\F)$ if and only if $\L=C_\L(Q)$.
\end{mprop}

A word about our proofs: Since there is some hope that the theory of fusion systems can be revisited using linking localities, we seek to keep the proofs of the results on subcentric subgroups of fusion systems as elementary as possible. In particular, we reprove some known results on constrained systems in Section~\ref{ConstrainedSection} without using the theory of components of fusion systems. However, it should be pointed out that we require this theory and Aschbacher's version of the L-balance theorem for fusion systems for the proof of Theorem~\ref{SubcentricEF}.

\subsubsection*{Organization of the paper} In Section~\ref{ConstrainedSection}, we state some background on constrained fusion systems and groups of characteristic $p$. In Section~\ref{SubcentricProperties} we prove important properties of subcentric subgroups. In particular, in Lemma~\ref{subcentricEquiv} we state equivalent conditions for a $p$-subgroup of a saturated fusion system to be subcentric; in Proposition~\ref{subcentricProp} we prove the property stated in Theorem~\ref{MainThm1}(b) that, if $\F$ is saturated, then the set $\F^s$ of subcentric subgroups is closed under taking $\F$-conjugates and overgroups; and in the remainder of Section~\ref{SubcentricProperties} we prove Propositions~\ref{SubcentricProperties1} and \ref{SubcentricProperties2}. Theorem~\ref{SubcentricEF} is proved in Section~\ref{SubcentricPropertiesSection2}; the proof uses the theory of components of fusion systems. 

\smallskip

We then turn attention to localities. In Section~\ref{LocalitiesSection0}, we summarize some background on localities and the connection to transporter systems. In Section~\ref{LocalitiesSection1}, we prove Proposition~\ref{ex0} and Proposition~\ref{NormLF}. Moreover, with Proposition~\ref{GetLocalityObjectiveCharp}, we  give a way of producing localities of objective characteristic $p$ by ``factoring our $p^\prime$-elements''. We also prove some results which are then used in Section~\ref{Construction} for the proof of Theorem~\ref{MainThm1}. A crucial step in the proof of Theorem~\ref{MainThm1} is Theorem~\ref{A1General}, which is also an interesting result on its own. Proposition~\ref{CENThm} is proved in Section~\ref{Centralizers}. 

\smallskip

In Section~\ref{MapsSection} and Section~\ref{GroupLocalities} we give some indication how the results proved before can be used. In Section~\ref{MapsSection}, we consider quotients of linking localities modulo central subgroups, and we outline how local subsystems and normal subsystems of fusion systems lead to inclusion maps between the corresponding subcentric linking localities. We expect this to be of importance for a local theory of subcentric linking localities. (The results on quotients of linking localities modulo central subgroups are already used in \cite{Henke:2017}, which in turn is necessary for the proof of the main theorem in \cite{Chermak/Henke:2017}.) In Section~\ref{GroupLocalities}, we present some ideas for a unifying approach to a classification of groups and fusion systems that are of characteristic $p$-type or satisfy some similar but slightly weaker condition. In particular, we discuss how linking localities and localities of objective characteristic $p$ can be constructed directly from a finite group; partly this uses Proposition~\ref{GetLocalityObjectiveCharp}.

\subsubsection*{Acknowledgement} The idea for this project arose during discussions with Andrew Chermak and Jesper Grodal. My heartfelt thanks go to both of them. It was Jesper Grodal who first conjectured that subcentric linking systems should exist. He also pointed out that the nerve of a subcentric linking system would be homotopy equivalent to the nerve of a centric linking system. It was Andrew Chermak who suggested using the iterative procedure introduced in \cite{Chermak:2013} to construct subcentric linking systems.

\bigskip

\textbf{Throughout, this text, we continue to assume that $\F$ is a fusion system on a finite $p$-group $S$. Given a subsystem $\m{E}$ of $\F$ we write $T=\m{E}\cap S$ to express that $\m{E}$ is a subsystem over $T\leq S$.}

\section{Groups of characteristic $p$ and constrained fusion systems}\label{ConstrainedSection}

\noindent\textbf{Throughout this section, $\F$ is assumed to be saturated.} 

\smallskip

The purpose of this section is to provide some background on constrained fusion systems and groups of characteristic $p$. Recall that $\F$ is called \textit{constrained} if $C_S(O_p(\F))\leq O_p(\F)$, and a finite group $G$ is said to be of \textit{characteristic $p$} if $C_G(O_p(G))\leq O_p(G)$ (or equivalently, $C_G(O_p(G))=Z(O_p(G))$). A finite group $G$ is called a \textit{model} for $\F$ if $S\in\Syl_p(G)$, $\F=\F_S(G)$ and $G$ is of characteristic $p$. The following lemma summarizes the connection between constrained fusion systems and groups of characteristic $p$ which was (except for some detail) established in \cite{BCGLO1}.

\begin{theorem}\label{Model1}
\begin{itemize}
 \item [(a)] $\F$ is constrained if and only if there exists a model for $\F$. In this case, a model is unique up to an isomorphism which is the identity on $S$.
 \item [(b)] If $\F$ is constrained and $G$ is a model for $\F$ then a subgroup of $S$ is normal in $\F$ if and only if it is normal in $G$. If $Q\leq S$ is normal and centric in $\F$, then in addition $C_G(Q)\leq Q$.
\end{itemize}
\end{theorem}

\begin{proof}
If $G$ is a model for $\F$ then clearly every normal $p$-subgroup of $G$ is normal in $\F$, so in particular, $\F$ is constrained. Thus, (a) follows from \cite[Theorem~III.5.10]{Aschbacher/Kessar/Oliver:2011}. Let now $\F$ be constrained and $G$ a model for $\F$. If $Q$ is a normal centric subgroup of $\F$ then it follows again from \cite[Theorem~III.5.10]{Aschbacher/Kessar/Oliver:2011} that $Q\unlhd G$ and $C_G(Q)\leq Q$. In particular, $O_p(\F)\unlhd G$. So if $g\in G$ then $c_g|_{O_p(\F)}\in \Aut_\F(O_p(\F))$ and thus $P^g=P$ for every normal subgroup $P$ of $\F$. This shows that every normal subgroup of $\F$ is normal in $G$ completing the proof.  
\end{proof}

We continue by listing some properties of groups of characteristic $p$.

\begin{lemma}\label{Charp1}
 Let $G$ be a finite group of characteristic $p$. Then the following hold:
\begin{itemize}
 \item [(a)] We have $O_{p^\prime}(G)=1$. 
 \item [(b)] $N_G(P)$ and $C_G(P)$ are of characteristic $p$ for all non-trivial $p$-subgroups $P$ of $G$. 
 \item [(c)] Every subnormal subgroup of $G$ is of characteristic $p$.
\end{itemize}

\end{lemma} 

\begin{proof}
Part (a) holds as $[O_p(G),O_{p^\prime}(G)]\leq O_p(G)\cap O_{p^\prime}(G)=1$ and $C_G(O_p(G))=Z(O_p(G))$ does not contain any non-trivial $p^\prime$-elements. By Part (c) of \cite[Lemma~1.2]{MS:2012b}, $N_G(P)$ is of characteristic $p$ and by part (a) of the same lemma, (c) holds. As $C_G(P)\unlhd N_G(P)$, it follows now that $C_G(P)$ is of characteristic $p$.
\end{proof}

\begin{lemma}\label{CharpCentral}
 Let $G$ be a finite group of characteristic $p$ and $Z\leq Z(G)$. Then $Z\leq O_p(G)$ and $G/Z$ is of characteristic $p$.
\end{lemma}

\begin{proof}
Note that $O_{p^\prime}(Z)\leq O_{p^\prime}(G)=1$ by Lemma~\ref{Charp1}(a). As $Z$ is abelian, it follows that $Z$ is a $p$-group and thus $Z\leq O_p(G)$. Set $C:=C_G(O_p(G)Z/Z)$. It is sufficient to show that $C\leq O_p(G)$. Note that $[O_p(G),C]\leq Z$ and $[Z,C]=1$ as $Z\leq Z(G)$. Hence, $[O_p(G),O^p(C)]=1$. As $C_G(O_p(G))=Z(O_p(G))$ is a $p$-group, it follows $O^p(C)=1$. So $C$ is a $p$-group and thus $C\leq O_p(G)$. 
\end{proof}

\begin{lemma}\label{CharpCentralEquiv}
Let $G$ be a finite group and $Z\leq Z(G)\cap O_p(G)$. Then $G$ is of characteristic $p$ if and only if $G/Z$ is of characteristic $p$.
\end{lemma}

\begin{proof}
If $G$ is of characteristic $p$, then $G/Z$ is of characteristic $p$ by Lemma~\ref{CharpCentral}. Assume now that $G/Z$ is of characteristic $p$. Note that $O_p(G/Z)=O_p(G)/Z$ as $Z\leq O_p(G)$. So we have $C_G(O_p(G)/Z)\leq O_p(G)$. In particular, $C_G(O_p(G))\leq O_p(G)$ and $G$ is of characteristic $p$.
\end{proof}

\begin{lemma}\label{Charp2}
Let $G$ be a finite group with a normal $p$-subgroup $P$ such that, for $S\in\Syl_p(G)$, we have $\F_{C_S(P)}(C_G(P))=\F_{C_S(P)}(C_S(P))$. Then $C_G(P)=C_S(P)O_{p^\prime}(C_G(P))$, and $G$ is of characteristic $p$ if and only if $C_G(P)$ is a $p$-group.
\end{lemma}

\begin{proof}
 By the Theorem of Frobenius \cite[Theorem~1.4]{Linckelmann:2007a}, $C_G(P)=C_S(P)O_{p^\prime}(C_G(P))$. If $G$ is of characteristic $p$, then by Lemma~\ref{Charp1}(a), $O_{p^\prime}(C_G(P))\leq O_{p^\prime}(G)=1$ and thus $C_G(P)=C_S(P)$ is a $p$-group. On the other hand, if $C_G(P)$ is a $p$-group then $C_G(P)\leq O_p(G)$ as $C_G(P)\unlhd G$. Hence, as $P\leq O_p(G)$, $C_G(O_p(G))\leq C_G(P)\leq O_p(G)$ and $G$ is of characteristic $p$.
\end{proof}

\begin{definition}\label{ThetaDef}
We say that a finite group $G$ is \textit{almost of characteristic $p$} if $G/O_{p^\prime}(G)$ is of characteristic $p$.\footnote{In the literature, groups which are almost of characteristic $p$ are usually called constrained, but we find our definition more intuitive in this context.}
\end{definition}

\begin{lemma}\label{Theta}
Let $G$ be a finite group and $P$ be a $p$-subgroup of $G$.
\begin{itemize}
\item [(a)] We have $O_{p^\prime}(N_G(P))=O_{p^\prime}(C_G(P))$.
\item [(b)] For $\ov{G}=G/O_{p^\prime}(G)$, we have  $N_{\ov{G}}(\ov{P})=\ov{N_G(P)}$ and $C_{\ov{G}}(\ov{P})=\ov{C_G(P)}$. \footnote{Following the usual convention, we write here $\ov{H}$ for the image of $H$ in $\ov{G}$, for every subgroup $H$ of $G$.}
\end{itemize}
\end{lemma}

\begin{proof}
 Since $C_G(P)\unlhd N_G(P)$ and $O_{p^\prime}(C_G(P))$ is characteristic in $C_G(P)$, we have $O_{p^\prime}(C_G(P))\leq O_{p^\prime}(N_G(P))$. As $[O_{p^\prime}(N_G(P)),P]\leq O_{p^\prime}(N_G(P))\cap P=1$, (a) follows. By a Frattini argument, $N_{\ov{G}}(\ov{P})=\ov{N_G(P)}$. If $C$ is the preimage of $C_{\ov{G}}(\ov{P})$ in $N_G(P)$ then $[P,C]\leq O_{p^\prime}(G)\cap P=1$ and thus $C=C_G(P)$. Hence, (b) holds.
\end{proof}

\begin{lemma}\label{AlmostCharp1}
 Let $G$ be a finite group which is almost of characteristic $p$ and $P$ a $p$-subgroup of $G$. Set $\ov{G}=G/O_{p^\prime}(G)$.
\begin{itemize}
 \item [(a)] We have $O_{p^\prime}(C_G(P))=O_{p^\prime}(N_G(P))=O_{p^\prime}(G)\cap N_G(P)=O_{p^\prime}(G)\cap C_G(P)$.
 \item [(b)] $N_G(P)$ and $C_G(P)$ are almost characteristic $p$.
% \item [(c)] Every subnormal subgroup of $G$ is almost of characteristic $p$.
\end{itemize}
\end{lemma}

\begin{proof}
By \cite[8.2.12]{Kurzweil/Stellmacher:2004a}, $O_{p^\prime}(N_G(P))=O_{p^\prime}(G)\cap N_G(P)$. Now (a) follows from Lemma~\ref{Theta}(a).  By (a) and Lemma~\ref{Theta}(b), $N_G(P)/O_{p^\prime}(N_G(P))\cong\ov{N_G(P)}=N_{\ov{G}}(\ov{P})$ and $C_G(P)/O_{p^\prime}(C_G(P))\cong \ov{C_G(P)}=C_{\ov{G}}(\ov{P})$. So the assertion follows from Lemma~\ref{Charp1}(b).
\end{proof}

\begin{lemma}\label{AlmostCharpNormCent}
Let $G$ be a finite group and let $P$ be a $p$-subgroup of $G$. Then $N_G(P)$ is of characteristic $p$ if and only if $C_G(P)$ is of characteristic $p$. Similarly, $C_G(P)$ is almost of characteristic $p$ if and only if $N_G(P)$ is almost of characteristic $p$.
\end{lemma}

\begin{proof}
Replacing $G$ by $N_G(P)$, we may assume without loss of generality that $P$ is normal in $G$. By Lemma~\ref{Charp1}(b), $C_G(P)$ is of characteristic $p$ if $G$ is of characteristic $p$ and, by Lemma~\ref{AlmostCharp1}(b), $C_G(P)$ is almost of characteristic $p$ if $G$ is almost of characteristic $p$. Note that $O_p(C_G(P))\unlhd G$ as $C_G(P)\unlhd G$. Hence, if $C_G(P)$ is of characteristic $p$, we have $C_G(O_p(G))\leq C_G(O_p(C_G(P))P)=C_{C_G(P)}(O_p(C_G(P)))\leq O_p(C_G(P))\leq O_p(G)$ and $G=N_G(P)$ is of characteristic $p$. Set $\ov{G}:=G/O_{p^\prime}(G)$. By Lemma~\ref{Theta}, $O_{p^\prime}(G)=O_{p^\prime}(N_G(P))=O_{p^\prime}(C_G(P))$ and $C_{\ov{G}}(\ov{P})=\ov{C_G(P)}$. By what we have just shown, $\ov{G}$ is of characteristic $p$ if $C_{\ov{G}}(\ov{P})$ is of characteristic $p$. So if $C_G(P)$ is almost of characteristic $p$ then $G=N_G(P)$ is almost of characteristic $p$. 
\end{proof}

The remainder of this section is devoted to exploring some connections between $\F$ being constrained and certain subsystems or factor systems of $\F$ being constrained.

\begin{lemma}\label{CentreConstrained}
Let $Z\leq Z(\F)$. Then $\F$ is constrained if and only if $\F/Z$ is constrained. Moreover, if $G$ is a model for $\F$, then $Z\leq Z(G)$ and $G/Z$ is a model for $\F/Z$.
\end{lemma}

\begin{proof}
 Suppose first that $\F$ is constrained and that $G$ is a model for $\F$. Note that, by Theorem~\ref{Model1}(a), a model $G$ always exists if $\F$ is constrained. By Theorem~\ref{Model1}(b), $Z$ is normal in $G$. So every $g\in G$ induces an $\F$-automorphism of $Z$ which then has to be the identity, as $Z\leq Z(\F)$. Hence, $Z\leq Z(G)$ and $G/Z$ is of characteristic $p$ by Lemma~\ref{CharpCentral}. By \cite[Example~II.5.6]{Aschbacher/Kessar/Oliver:2011}, $\F/Z=\F_{S/Z}(G/Z)$ and so $G/Z$ is a model for $\F/Z$. Hence, by Theorem~\ref{Model1}(a), $\F/Z$ is constrained. Assume now that $\F/Z$ is constrained and let $Z\leq Q\leq S$ with $Q/Z=O_p(\F/Z)$. Then $C_S(Q)\leq Q$ as $C_{S/Z}(Q/Z)\leq Q/Z$. So it is sufficient to show that $Q$ is normal in $\F$. Observe that $Q$ is strongly closed in $\F$, since $Q/Z$ is strongly closed in $\F/Z$ and every morphism in $\F$ induces a morphism in $\F/Z$. By \cite[Proposition~I.4.5]{Aschbacher/Kessar/Oliver:2011}, a subgroup of a fusion system is normal if and only if it is strongly closed and contained in every centric radical subgroup. So $Q/Z$ is contained in every element of $(\F/Z)^{cr}$ and it is sufficient to show that $Q$ is contained in every element of $\F^{cr}$. As shown in \cite[Proposition~3.1]{Kessar/Linckelmann:2008}, we have $R/Z\in(\F/Z)^{cr}$ for every $R\in\F^{cr}$. So $Q$ is contained in every element of $\F^{cr}$ as required. 
\end{proof}

We now turn attention to subsystems of $\F$, in particular to $p$-local subsystems and (weakly) normal subsystems.

\begin{lemma}\label{Model2}
Let $\F$ be constrained and $P\in\F^f$. Then $N_\F(P)$ and $C_\F(P)$ are constrained. Moreover, if $G$ is a model for $\F$, then $N_G(P)$ is a model for $N_\F(P)$ and $C_G(P)$ is a model for $C_\F(P)$.
\end{lemma}

\begin{proof}
 Let $\F$ be a constrained fusion system on a finite $p$-group $S$ and $G$ a model for $\F$. Note that $G$ always exists by Theorem~\ref{Model1}(a). By \cite[Proposition~I.5.4]{Aschbacher/Kessar/Oliver:2011}, $N_S(P)\in\Syl_p(N_G(P))$, $C_S(P)\in\Syl_p(C_G(P))$, $N_\F(P)=\F_{N_S(P)}(N_G(P))$ and $C_\F(P)=\F_{C_S(P)}(C_G(P))$. By Lemma~\ref{Charp1}(b), $N_G(P)$ and $C_G(P)$ are of characteristic $p$, so $N_G(P)$ is a model for $N_\F(P)$ and $C_G(P)$ is a model for $C_\F(P)$. In particular, by Theorem~\ref{Model1}(a), $N_\F(P)$ and $C_\F(P)$ are  constrained.
\end{proof}

We continue with a general lemma needed afterwards to prove results about constrained fusion systems. Crucial for us is part (b) of this lemma, which could also be obtained as a consequence of \cite[(7.4)]{Aschbacher:2011} and the fact that, for every $P\in\F$, $P\unlhd \F$ if and only if $\F_P(P)\unlhd \F$. We prefer however to give an elementary direct proof.

\begin{lemma}\label{OpE}
 Let $\m{E}$ be a weakly normal subsystem of $\F$ over $T\unlhd S$. Then the following hold:
\begin{itemize}
 \item [(a)] If $Q\unlhd\m{E}$ is $\Aut_\F(T)$-invariant, then $Q$ is normal in $\F$.   
 \item [(b)] $O_p(\m{E})$ is normal in $\F$.
\end{itemize}
\end{lemma}

\begin{proof}
As $\m{E}$ is weakly normal in $\F$, every element of $\Aut_\F(T)$ induces an automorphism of $\m{E}$. Thus $O_p(\m{E})$ is $\Aut_\F(T)$-invariant and (b) follows from (a). 

\smallskip

For the proof of (b) let $Q\unlhd \m{E}$ be $\Aut_\F(T)$-invariant. Then $Q$ is in particular strongly closed in $\m{E}$. Therefore, it  follows from the Frattini condition as stated in \cite[Definition~I.6.1]{Aschbacher/Kessar/Oliver:2011} that $Q$ is strongly closed in $\F$. Hence, by \cite[Theorem~I.4.5]{Aschbacher/Kessar/Oliver:2011}, it is sufficient to prove that $Q$ is contained in every element of $\F^{cr}$. 

\smallskip

Let $R\in\F^{cr}$ and set $R_0:=R\cap T$. Recall that $T$ is strongly closed and so $R_0$ is $\Aut_\F(R)$ invariant. As $Q$ is normal in $\m{E}$, $\Aut_{Q}(R_0)\unlhd \Aut_\m{E}(R_0)$. Thus, $\Aut_{Q}(R_0)\leq O_p(\Aut_\m{E}(R_0))\leq O_p(\Aut_\F(R_0))$ since $\Aut_{\m{E}}(R_0)\unlhd \Aut_\F(R_0)$. It follows that the restriction of every element of $X:=\<\Aut_{Q}(R)^{\Aut_\F(R)}\>$ to $R_0$ lies in $O_p(\Aut_\F(R_0))$. Hence, $[R_0,O^p(X)]=1$. Since $[R,N_{Q}(R)]\leq [R,N_T(R)]\leq T\cap R=R_0$, we have $[R,X]\leq R_0$. Thus, $O^p(X)=1$ meaning that $X$ is a normal $p$-subgroup of $\Aut_\F(R)$. Consequently, as $R$ is centric radical, $\Aut_{Q}(R)\leq X\leq \Inn(R)$ and $Q\leq R$.
\end{proof}

The following Lemma can be seen as a fusion system version of Lemma~\ref{AlmostCharpNormCent}.

\begin{lemma}\label{NCconstrained}
 Let $Q\in\F^f$. Then $N_\F(Q)$ is constrained if and only if $C_\F(Q)$ is constrained.
\end{lemma}

\begin{proof}
 If $N_\F(Q)$ is constrained, then it follows from Lemma~\ref{Model2} applied to $N_\F(Q)$ in place of $\F$ that $C_\F(Q)$ is constrained. Assume now that $C_\F(Q)$ is constrained. By \cite[1.25]{AOV1}, $C_\F(Q)$ is weakly normal in $N_\F(Q)$. It follows now from Lemma~\ref{OpE}(b) that $R:=QO_p(C_\F(Q))\unlhd N_\F(Q)$. Moreover, $C_{N_S(Q)}(R)=C_{C_S(Q)}(O_p(C_\F(Q)))\leq O_p(C_\F(Q))\leq R$ as $C_\F(Q)$ is constrained. Thus, $N_\F(Q)$ is constrained.
\end{proof}

The reader is referred to \cite[Section~I.7]{Aschbacher/Kessar/Oliver:2011} for definitions and properties of subsystems of index prime to $p$ and of subsystems of $p$-power index.

\begin{lemma}\label{pPrimeIndexConstrained}
Let $\m{E}$ be a normal subsystem of $\F$ of index prime to $p$. Then $\m{E}$ is constrained if and only if $\F$ is constrained.
\end{lemma}

\begin{proof}
Clearly, $O_p(\F)$ is normal in $\m{E}$, so $O_p(\m{E})=O_p(\F)$ by Lemma~\ref{OpE}(b). As $\m{E}\cap S=S$, it follows that $\m{E}$ is a constrained if and only if $C_S(O_p(\F))\leq O_p(\F)$, which is the case if and only if $\F$ is constrained.
\end{proof}

\begin{lemma}\label{WeaklyNormalConstraint}
Let $\F$ be constrained and let $\m{E}$ be a weakly normal subsystem of $\F$. Then $\m{E}$ is constrained.
\end{lemma}

\begin{proof}
By Lemma~\ref{pPrimeIndexConstrained}, $\m{E}$ is constrained if and only if $O^{p^\prime}(\m{E})$ is constrained. By a theorem of Craven \cite{Craven:2011}, $O^{p^\prime}(\m{E})$ is normal in $\F$. So replacing $\m{E}$ by $O^{p^\prime}(\m{E})$, we may assume that $\m{E}$ is normal in $\F$. Let $G$ be a model for $\F$, which exist by Theorem~\ref{Model1}(a). By \cite[Lemma~II.7.4]{Aschbacher/Kessar/Oliver:2011}, there exists a normal subgroup $N$ of $G$ such that $T:=\m{E}\cap S=N\cap S\in\Syl_p(N)$ and $\m{E}=\F_T(N)$. By Lemma~\ref{Charp1}(c), $N$ is of characteristic $p$ and thus $\m{E}$ is constrained by Theorem~\ref{Model1}(a). 
\end{proof}

The above lemma implies that every subnormal subsystem of a constrained fusion system is constrained; this statement could be seen as a fusion system version of Lemma~\ref{Charp1}(c).

\smallskip

The following lemma is a version of \cite[Lemma~1.3]{MS:2012b} for fusion systems, except that we do not require the subsystem $\m{E}$ to be normal in $\F$. A different proof could be given using the theory of components of fusion systems as developed by Aschbacher \cite{Aschbacher:2010}, but we prefer to keep the proof as elementary as possible.

\begin{lemma}\label{pPowerIndexConstrained}
 Let $\m{E}$ be a subsystem of $\F$ of $p$-power index. Then $\m{E}$ is constrained if and only if $\F$ is constrained.
\end{lemma}

\begin{proof}
If $\F$ is constrained, then $\m{E}$ is constrained by Lemma~\ref{WeaklyNormalConstraint}. Hence, for the rest of the proof, $\m{E}$ is assumed to be constrained, and we will show that $\F$ is constrained.

\smallskip

Let $T=\m{E}\cap S$. Let $T=T_0\unlhd T_1\unlhd\dots \unlhd T_n=S$ be a chain of subgroups such that $|T_i/T_{i-1}|=p$ for $i=1,\dots,n$. By \cite[Theorem~I.7.4]{Aschbacher/Kessar/Oliver:2011}, there is a unique subsystem $\F_{T_i}=\<\Inn(T_i),O^p(\Aut_\F(P))\colon P\leq T_i\>$ of $\F$ of $p$-power index over $T_i$ for every $i=1,\dots,n$. In particular, $\F_T=\F_{T_0}=\m{E}$. Again by \cite[Theorem~I.7.4]{Aschbacher/Kessar/Oliver:2011}, $\F_{T_{i-1}}$ is a normal subsystem of $\F_{T_i}$ of $p$-power index for $i=1,\dots,n$. Hence, we can reduce to the case that $|S:T|=p$ and $\m{E}$ is normal in $\F$. 

\smallskip

So assume now $|S:T|=p$ and $\m{E}\unlhd\F$. By Lemma~\ref{OpE}(b), $Q:=O_p(\m{E})$ is normal in $\F$. It is sufficient to show that $P:=QC_S(Q)$ is normal in $\F$. As $\m{E}$ is constrained, $C_T(Q)\leq Q$ and thus $|P:Q|\leq |S:T|=p$. As $Q$ is normal in $\F$, $P$ is weakly closed in $\F$. We prove now that $P$ is strongly closed. Let $X\leq P$ and $\phi\in \Hom_\F(X,S)$. If $X\leq Q$ then $X\phi\leq Q\leq P$. If $X\not\leq Q$ then $P=QX$ as $|P:Q|\leq p$. Since $Q\unlhd\F$, $\phi$ extends in this case to an element $\Hom_\F(P,S)$. As $P$ is weakly closed in $\F$, it follows $X\phi\leq P$. So $P$ is strongly closed. 

\smallskip

By \cite[Proposition~I.4.6]{Aschbacher/Kessar/Oliver:2011}, there exists a series $1=P_0\leq P_1\leq \dots P_n=Q$ of subgroups strongly closed in $\F$ such that $[P_i,Q]\leq P_{i-1}$ for $i=1,\dots,n$. Since $P=QC_S(Q)\leq QC_S(P_i)$, it follows $[P_i,P]\leq P_{i-1}$  for $i=1,\dots,n$. As $|P:Q|\leq p$, we have $[P,P]\leq Q$. Hence, $P\unlhd \F$ by \cite[Proposition~I.4.6]{Aschbacher/Kessar/Oliver:2011}. As $C_S(P)\leq C_S(Q)\leq P$, it follows that $\F$ is constrained.
\end{proof}

The reader is referred to \cite[Section~I.5]{Aschbacher/Kessar/Oliver:2011} for the definitions of $K$-normalizers. We will need the following elementary lemma:

\begin{lemma}\label{KNormElementary}
 Let $Q\leq S$, $K\leq \Aut(Q)$, $P\leq N_S^K(Q)$ and $\alpha\in\Hom_\F(PQ,S)$. Then $P\alpha\leq N_S^{K^\alpha}(Q\alpha)$.  
\end{lemma}

\begin{proof}
Notice that for every $s\in P$ and every $x\in Q$, $(x\alpha)^{s\alpha}=(x^s)\alpha=(x\alpha)\alpha^{-1}c_s\alpha=(x\alpha)(c_s|_Q)^\alpha$. So as $c_s|_Q\in K$, $c_{s\alpha}|_{Q\alpha}=(c_s|_Q)^\alpha\in K^\alpha$ and thus $s\alpha\in N_S^{K^\alpha}(Q\alpha)$.
\end{proof}

\begin{lemma}\label{WeaklyNormalCentInKNorm}
 Let $Q\leq S$ and $K\leq \Aut(Q)$ be such that $N_\F^K(Q)$ is saturated. 
\begin{itemize}
\item [(a)] Let $K_0\unlhd K$ be such that $N_\F^{K_0}(Q)$ is saturated. Then $N_\F^{K_0}(Q)$ is weakly normal in  $N_\F^K(Q)$.
\item [(b)] Suppose $C_\F(Q)$ is saturated. Then $C_\F(Q)$ is weakly normal in $N_\F^K(Q)$. In particular, $C_\F(Q)$ is constrained if $N_\F^K(Q)$ is constrained. 
\end{itemize}
\end{lemma}

\begin{proof}
As $C_\F(Q)=N_\F^{\{\id_Q\}}(Q)$, part (b) follows from part (a) and Lemma~\ref{WeaklyNormalConstraint}. So it remains to prove (a). As $N_\F^{K_0}(Q)$ and $N_\F^K(Q)$ are by assumption both saturated, it remains to prove that $N_\F^{K_0}(Q)$ is $N_\F^K(Q)$-invariant. 

\smallskip

We show first that $N_S^{K_0}(Q)$ is strongly closed in $N_\F^K(Q)$. Let $R\leq N_S^{K_0}(Q)$ and $\alpha\in\Hom_{N_\F^K(Q)}(R,N_S^K(Q))$. Then it follows from the definition of $N_\F^K(Q)$ that $\alpha$ extends to $\hat{\alpha}\in\Hom_\F(RQ,S)$ with $\hat{\alpha}|_Q\in K$. As $K_0\unlhd K$ and $\hat{\alpha}|_Q\in K$, we have $K_0^{\hat{\alpha}}=K_0$. Hence, by Lemma~\ref{KNormElementary}, $R\alpha\leq N_S^{K_0^{\hat{\alpha}}}(Q\hat{\alpha})=N_S^{K_0}(Q)$.  This shows that $N_S^{K_0}(Q)$ is strongly closed in $N_\F^K(Q)$. 

\smallskip

We show now that the strong invariance condition as given in \cite[Proposition~I.6.4]{Aschbacher/Kessar/Oliver:2011} holds. For that let $A\leq B\leq N_S^{K_0}(Q)$, $\phi\in\Hom_{N_\F^{K_0}(Q)}(A,B)$ and $\psi\in\Hom_{N_\F^K(Q)}(B,N_S^{K_0}(Q))$. We need to prove that $\phi^\psi:=(\psi|_A)^{-1}\phi\psi\in\Hom_{N_\F^{K_0}(Q)}(A\psi,B\psi)$. By the definitions of $N_\F^{K_0}(Q)$ and $N_\F^K(Q)$, $\phi$ extends to $\hat{\phi}\in\Hom_\F(AQ,BQ)$ with $\hat{\phi}|_Q\in K_0$, and $\psi$ extends to $\hat{\psi}\in\Hom_\F(BQ,S)$ with $\hat{\psi}|_Q\in K$. Then $\hat{\phi}^{\hat{\psi}}:=(\hat{\psi}|_{AQ})^{-1}\hat{\phi}\hat{\psi}\in\Hom_\F((A\psi)Q,(B\psi)Q)$ extends $\phi^\psi$. Moreover, $(\hat{\phi}^{\hat{\psi}})|_Q=(\hat{\psi}|_Q)^{-1}(\hat{\phi}|_Q)(\hat{\psi}|_Q)\in K_0$ as $\hat{\phi}|_Q\in K_0$, $\hat{\psi}|_Q\in K$ and $K_0\unlhd K$. This shows that $\phi^\psi$ is a morphism in $N_\F^{K_0}(Q)$ as required.
\end{proof}

\section{Properties of subcentric subgroups}\label{SubcentricProperties}

\noindent\textbf{Throughout this section, $\F$ is assumed to be saturated.}

\smallskip

In this section we prove important results about subcentric subgroups. In particular, in Proposition~\ref{subcentricProp} we prove that the set $\F^s$ of subcentric subgroups is closed under taking $\F$-conjugates and overgroups; and at the end of the section we prove Propositions~\ref{SubcentricProperties1} and \ref{SubcentricProperties2}. We start with the following crucial characterization of subcentric subgroups.

\begin{lemma}\label{subcentricEquiv}
For every $Q\in \F$, the following conditions are equivalent:
\begin{itemize}
\item[(a1)] The subgroup $Q$ is subcentric in $\F$.
\item[(a2)] For some fully normalized $\F$-conjugate $P$ of $Q$, $O_p(N_\F(P))$ is centric in $\F$.
\item[(b1)] For every fully normalized $\F$-conjugate $P$ of $Q$, $N_\F(P)$ is constrained.
\item[(b2)] For some fully normalized $\F$-conjugate $P$ of $Q$, $N_\F(P)$ is constrained.
\item[(c1)] For every fully centralized $\F$-conjugate $P$ of $Q$, $C_\F(P)$ is constrained.
\item[(c2)] For some fully centralized $\F$-conjugate $P$ of $Q$, $C_\F(P)$ is constrained.
\end{itemize} 
\end{lemma}

\begin{proof}
If $P,P^*\in Q^\F$ are both fully normalized, then it follows from \cite[Lemma~I.2.6(c)]{Aschbacher/Kessar/Oliver:2011} that there exists an isomorphism $\phi\in\Hom_\F(N_S(P),N_S(P^*))$ such that $P\phi=P^*$. It is straightforward to check that any such $\phi$ induces an isomorphism from $N_\F(P)$ to $N_\F(P^*)$ and thus $N_\F(P)$ is constrained if and only if $N_\F(P^*)$ is constrained. Moreover, $O_p(N_\F(P))\phi=O_p(N_\F(P^*))$. Thus, conditions (b1) and (b2) are equivalent, and conditions (a1) and (a2) are equivalent. 

\smallskip

Similarly, if $P,P^*\in Q^\F$ are both fully centralized in $\F$, then by the extension axiom, there exists $\phi\in\Hom_\F(C_S(P)P,C_S(P^*)P^*)$ with $P\phi=P^*$, and $\phi|_{C_S(P)}$ induces and isomorphism from $C_\F(P)$ to $C_\F(P^*)$. This proves that conditions (c1) and (c2) are equivalent. 

\smallskip

Let now $P\in Q^\F$ be fully normalized. By Lemma~\ref{NCconstrained}, $N_\F(P)$ is constrained if and only if $C_\F(P)$ is constrained. Since every fully normalized subgroup is fully centralized, this shows that (b2) implies (c2) and that (c1) implies (b1). 

\smallskip

Set now $R:=O_p(N_\F(P))$. If $Q$ is subcentric, then $C_{N_S(P)}(R)\leq C_S(R)\leq R$ and so $N_\F(P)$ is constrained. Hence, (a1) implies (b1). Assume now $N_\F(P)$ is constrained. By \cite[Lemma~I.2.6(c)]{Aschbacher/Kessar/Oliver:2011}, there exists $\phi\in \Hom_\F(N_S(R),S)$ such that $R\phi\in\F^f$. As $N_S(P)\leq N_S(R)$ and $P$ is fully normalized, it follows $N_S(P)\phi=N_S(P\phi)$ and $P\phi\in\F^f$. Again, $\phi|_{N_S(P)}$ induces an isomorphism from $N_\F(P)$ to $N_\F(P\phi)$ and thus $R\phi=O_p(N_\F(P\phi))$ and $N_\F(P\phi)$ is constrained. Observe that $P\phi\leq R\phi$ and thus $C_S(R\phi)\leq C_S(P\phi)\leq N_S(P\phi)$. Hence, $C_S(R\phi)=C_{N_S(P\phi)}(R\phi)\leq R\phi$ as $N_\F(P\phi)$ is constrained. So $R\phi$ and thus $R$ is centric as $R\phi$ is fully normalized. Hence, (b2) implies (a2). 
\end{proof}

Looking more generally at $K$-normalizers rather than at centralizers and normalizers of subgroups of $S$, we get the following sufficient condition for a subgroup to be subcentric:

\begin{lemma}\label{KNormSubcentric}
 Let $Q\in\F$ and $K\leq \Aut(Q)$ be such that $Q$ is fully $K$-normalized. If $N^K_\F(Q)$ is constrained, then $Q$ is subcentric in $\F$.
\end{lemma}

\begin{proof}
Since $Q$ is fully $K$-normalized, $Q$ is fully centralized by \cite[Proposition~I.5.2]{Aschbacher/Kessar/Oliver:2011}. Now by \cite[Theorem~I.5.5]{Aschbacher/Kessar/Oliver:2011}, $C_\F(Q)$ and $N_\F^K(Q)$ are saturated. If $N_\F^K(Q)$ is constrained, it follows therefore from Lemma~\ref{WeaklyNormalCentInKNorm} that $C_\F(Q)$ is constrained. So $Q$ is subcentric by Lemma~\ref{subcentricEquiv}.
\end{proof}

\begin{prop}\label{subcentricProp}
The set $\F^s$ of subcentric subgroups of $\F$ is closed under taking $\F$-conjugates and overgroups.
\end{prop}

\begin{proof}
Note first that the set of subcentric subgroups is by definition closed under $\F$-conjugation. 
Let $Q\in\F^s$ and $R$ an overgroup of $Q$. We need to show that $R$ is subcentric. By induction on the length of a subnormal series of $Q$ in $R$, we reduce to the case that $Q\unlhd R$. 

\smallskip

So we assume now $Q\unlhd R$ and need to show that $R$ is subcentric in $\F$. Since every $\F$-conjugate of $Q$ is subcentric, and any $\F$-conjugate of $R$ contains an $\F$-conjugate of $Q$, we can and will furthermore assume from now on that $R\in\F^f$. Replacing $Q$ by a suitable conjugate of $Q$ in $N_\F(R)$, we will also assume that $Q\in N_\F(R)^f$. 

\smallskip

By \cite[Lemma~I.2.6(c)]{Aschbacher/Kessar/Oliver:2011}, there exists $\alpha\in\Hom_\F(N_S(Q),S)$ such that $Q\alpha\in\F^f$. Then by \cite[(2.2)(1),(2)]{Aschbacher:2010}, $(N_S(Q)\cap N_S(R))\alpha=N_S(Q\alpha)\cap N_S(R\alpha)$, $R\alpha\in N_\F(Q\alpha)^f$, and $\alpha$ induces an isomorphism from $\m{N}_1:=N_{N_\F(R)}(Q)$ to $\m{N}_2:=N_{N_\F(Q\alpha)}(R\alpha)$. 

%Using $Q\leq R$ and $Q\alpha\leq R\alpha$, we observe
%\[C_\F(R)=C_{\m{N}_1}(R)\mbox{ and }C_{\m{N}_2}(R\alpha)=C_{N_\F(Q\alpha)}(R\alpha).\]

\smallskip

As $Q$ is subcentric and $Q\alpha\in\F^f$, $N_\F(Q\alpha)$ is constrained by Lemma~\ref{subcentricEquiv}. Since $R\alpha\in N_\F(Q\alpha)^f$, it follows thus from Lemma~\ref{Model2} applied with $N_\F(Q\alpha)$ in place of $\F$ that $C_{N_\F(Q\alpha)}(R\alpha)$ is constrained. Note that $C_{\m{N}_2}(R\alpha)=C_{N_\F(Q\alpha)}(R\alpha)$. 
So since  $\alpha$ induces an isomorphism $\m{N}_1\rightarrow \m{N}_2$, it follows that $C_{\m{N}_1}(R)$ is constrained. Using $R\leq Q$, we observe that $C_\F(R)=C_{\m{N}_1}(R)$. So $C_\F(R)$ is constrained and $R$ is subcentric by Lemma~\ref{subcentricEquiv}.
\end{proof}

\begin{lemma}\label{NormalTimesSubcentric}
 Let $R\unlhd \F$ and $P\in\F$. Then $PR\in\F^s$ if and only if $P\in\F^s$.
\end{lemma}

\begin{proof}
 If $P\in\F^s$ then by Proposition~\ref{subcentricProp}, $PR\in\F^s$. From now on we assume that $PR\in\F^s$ and want to show that $P\in\F^s$. Since $\F^s$ is closed under $\F$-conjugation, we can assume without loss of generality that $PR\in\F^f$. As $PR\in\F^s$ this means that $N_\F(PR)$ is a constrained fusion system. 

\smallskip

We reduce next to the case that $P$ is fully normalized as well. Pick a fully $\F$-normalized $\F$-conjugate $Q$ of $P$ and an isomorphism $\phi\in\Hom_\F(Q,P)$.  As $R$ is normal in $\F$, $\phi$ extends to a morphism $\hat{\phi}\in\Hom_\F(QR,S)$ with $R\hat{\phi}=R$. Hence, as $(QR)\hat{\phi}=PR\in\F^f$, there exists by \cite[Lemma~I.2.6(c)]{Aschbacher/Kessar/Oliver:2011} a morphism $\alpha\in\Hom_\F(N_S(QR),S)$ such that $(QR)\alpha=PR$. Note that $(Q\alpha)R=(QR)\alpha=PR$ and $Q\alpha$ is $\F$-conjugate to $P$. As $N_S(Q)\leq N_S(QR)$ and $Q$ is fully normalized, it follows that $Q\alpha$ is fully normalized. So replacing $P$ by $Q\alpha$, we can and will assume from now on that $P$ is fully normalized in $\F$. 

\smallskip

Then $P$ is also fully normalized in $N_\F(PR)$ and thus $N_{N_\F(PR)}(P)$ is constrained by Lemma~\ref{Model2}. One easily observes that $N_\F(P)=N_{N_\F(PR)}(P)$, as $R$ is normal in $\F$. So $N_\F(P)$ is constrained and $P$ is subcentric by Lemma~\ref{subcentricEquiv}.  
\end{proof}

\begin{lemma}\label{SubcentricModCentral}
 Let $Z\leq Z(\F)$ and $P\leq S$. Then $P\in \F^s$ if and only if $PZ/Z$ is subcentric in $\F/Z$.
\end{lemma}

\begin{proof}
By Lemma~\ref{NormalTimesSubcentric}, we may assume throughout the proof that $Z\leq P$. Then $Z\leq Q$ for every $Q\in P^\F$. Moreover, the $\F/Z$-conjugates of $P/Z$ are precisely the subgroups of the form $Q/Z$ with $Q\in P^\F$.
So we may assume that $P\in\F^f$. Observe that $N_{S/Z}(Q/Z)=N_S(Q)/Z$ for every $Q\in P^\F$. Hence, $P/Z$ is fully normalized in $\F/Z$. One checks easily that $N_\F(P)/Z=N_{\F/Z}(P/Z)$ and $Z\leq Z(N_\F(P))$. Therefore, by Lemma~\ref{CentreConstrained}, $N_\F(P)$ is constrained if and only if $N_{\F/Z}(P/Z)$ is constrained. The assertion follows now from Lemma~\ref{subcentricEquiv}.
\end{proof}

Suppose $\tilde{\F}$ is a fusion system on a $p$-group $\tilde{S}$ and $\alpha\colon S\rightarrow \tilde{S}$ is an isomorphism of groups. We say that $\alpha$ induces an isomorphism of fusion systems from $\F$ to $\tilde{\F}$ if, for all $P,Q\leq S$, the map $\alpha_{P.Q}\colon \Hom_\F(P,Q)\rightarrow \Hom_\F(P\alpha,Q\alpha)$ with $\phi\mapsto \phi^\alpha:=\alpha^{-1}\phi\alpha$ is a bijection. The maps $\alpha_{P,Q}$ ($P,Q\leq S$) together with the map $P\mapsto P\alpha$ from the set of objects of $\F$ to the set of objects of $\tilde{\F}$ give us an invertible functor  from $\F$ to $\tilde{\F}$. Moreover, $\alpha$ together with the maps $\alpha_{P,Q}$ ($P,Q\leq S$) is a morphism in the sense of \cite[Definition~II.2.2]{Aschbacher/Kessar/Oliver:2011}.

\begin{lemma}\label{IsoSubcentric}
 Let $\tilde{\F}$ be a saturated fusion system on a $p$-group $\tilde{S}$ and $\alpha\colon S\rightarrow \tilde{S}$ a group isomorphism which induces an isomorphism of fusion systems $\F\rightarrow \tilde{\F}$. Then $\tilde{\F}^s=\{P\alpha\colon P\in\F^s\}$.
\end{lemma}

\begin{proof}
One observes easily that $\{Q\alpha:Q\in P^\F\}=(P\alpha)^{\tilde{\F}}$. Note that $N_S(Q)\alpha= N_{\tilde{S}}(Q\alpha)$ for every $Q\leq S$. So $Q\in P^\F$ is fully $\F$-normalized if and only if $Q\alpha$ is fully $\tilde{\F}$-normalized. Let now $Q\in P^\F$ be fully $\F$-normalized. Then it is easy to check that $\alpha|_{N_S(Q)}:N_S(Q)\rightarrow N_{\tilde{S}}(Q\alpha)$ induces an isomorphism from $N_\F(Q)$ to $N_{\tilde{\F}}(Q\alpha)$. In particular, $N_\F(Q)$ is constrained if and only if $N_{\tilde{\F}}(Q\alpha)$ is constrained. Hence, by Lemma~\ref{subcentricEquiv} $P\in\F^s$ if and only if $P\alpha\in\tilde{\F}^s$.  
\end{proof}

\begin{lemma}\label{ESubphi}
 Let $\m{E}$ be weakly normal in $\F$, $P\in\m{E}^s$ and $\phi\in\Hom_\F(P,S)$. Then $P\phi\in\m{E}^s$.
\end{lemma}

\begin{proof}
Let $T=\m{E}\cap S$. Note that $P\phi\leq T$ as $T$ is strongly closed. By the Frattini condition \cite[Definition~I.6.1]{Aschbacher/Kessar/Oliver:2011}, there are $\alpha\in\Aut_\F(T)$ and $\phi_0\in\Hom_\m{E}(P,T)$ such that $\phi=\phi_0\alpha$. As $\phi_0$ is a morphism in $\m{E}$, $P\phi_0\in\m{E}^s$. As $\m{E}$ is normal in $\F$, $\alpha$ induces an automorphism of $\m{E}$. Hence, by Lemma~\ref{IsoSubcentric} applied with $\m{E}$ in the role of $\F$ and $\tilde{\F}$, $P\phi=(P\phi_0)\alpha\in\m{E}^s$.
\end{proof}

Before we continue proving properties of subcentric subgroups we need two general lemmas.

\begin{lemma}\label{NormInE}
 Let $\m{E}$ be an $\F$-invariant subsystem of $\F$ over $T\leq S$. Let $P\in\m{E}^f$ and $\alpha\in\Hom_\F(N_T(P),S)$. Then $P\alpha\in\m{E}^f$, $N_T(P)\alpha=N_T(P\alpha)$ and $\alpha$ induces an isomorphism from $N_{\m{E}}(P)$ to $N_\m{E}(P\alpha)$. 
\end{lemma}

\begin{proof}
By the Frattini condition \cite[Definition~I.6.1]{Aschbacher/Kessar/Oliver:2011} there are $\alpha_0\in\Hom_\m{E}(N_T(P),T)$ and $\beta\in\Aut_\F(T)$ such that $\alpha=\alpha_0\beta$. Clearly, $N_T(P)\alpha_0\leq N_T(P\alpha_0)$ because $T$ is strongly closed in $\F$. As $P\in\m{E}^f$, it follows $N_T(P)\alpha_0=N_T(P\alpha_0)$. Since $\beta$ is an automorphism of $T$, $N_T(P\alpha_0)\beta=N_T(P\alpha_0\beta)=N_T(P\alpha)$. Hence, $N_T(P)\alpha=N_T(P)\alpha_0\beta=N_T(P\alpha_0)\beta=N_T(P\alpha)$. Since $\m{E}$ is $\F$-invariant, it is now straightforward to check that $\alpha$ induces an isomorphism from $N_\m{E}(P)$ to $N_\m{E}(P\alpha)$.  
\end{proof}

\begin{lemma}\label{FfEf}
 Let $\m{E}$ be an $\F$-invariant subsystem of $\F$ over $T\leq S$, and $P\leq T$. If $P\in\F^f$ then $P\in\m{E}^f$.
\end{lemma}

\begin{proof}
Suppose $P\in\F^f$ and choose a fully $\m{E}$-normalized $\m{E}$-conjugate $Q$ of $P$. By \cite[Lemma~I.2.6(c)]{Aschbacher/Kessar/Oliver:2011}, there exists $\alpha\in\Hom_\F(N_S(Q),S)$ such that $Q\alpha=P$. Applying Lemma~\ref{NormInE} with $Q$ in place of $P$ yields then $|N_T(Q)|=|N_T(P)|$ and thus $P\in\m{E}^f$. 
\end{proof}

\begin{lemma}\label{Subcentric}
 Let $\m{E}$ be a weakly normal subsystem of $\F$ over $T\leq S$. Then $P\in\m{E}^s$ for any $P\in\F^s$ with $P\leq T$.
\end{lemma}

\begin{proof}
By Lemma~\ref{ESubphi}, we may replace $P$ by any $\F$-conjugate of $P$ and can thus assume that $P\in\F^f$. Then by Lemma~\ref{FfEf}, $P\in\m{E}^f$. So $N_\F(P)$ and $N_\m{E}(P)$ are saturated. It is now easy to check that $N_\m{E}(P)$ is weakly normal in $N_\F(P)$. Since $P\in\F^s$, $N_\F(P)$ is constrained by Lemma~\ref{subcentricEquiv}. Hence, by Lemma~\ref{WeaklyNormalConstraint}, $N_\m{E}(P)$ is constrained and $P\in\m{E}^s$ again by Lemma~\ref{subcentricEquiv}.
\end{proof}

\begin{lemma}\label{pPrimeIndexSubcentric}
 Let $\m{E}$ be a normal subsystem of $\F$ of index prime to $p$. Then $\m{E}^s=\F^s$.
\end{lemma}

\begin{proof}
 By Lemma~\ref{Subcentric}, we only need to prove that $\m{E}^s\subseteq\F^s$. By Lemma~\ref{ESubphi}, it is sufficient to prove $\m{E}^s\cap \F^f\subseteq \F^s$. Let $P\in\m{E}^s\cap\F^f$. By Lemma~\ref{FfEf}, $P\in\m{E}^f$. Thus $N_\F(P)$ and $N_\m{E}(P)$ are saturated subsystems and one sees easily that $N_\m{E}(P)$ is a weakly normal subsystem of $N_\F(P)$. As they are both fusion systems over $N_S(P)$, it follows that $N_\m{E}(P)$ is actually normal in $N_\F(P)$ and a subsystem of index prime to $p$. As $P\in\m{E}^s$, $N_\m{E}(P)$ is constrained by Lemma~\ref{subcentricEquiv}. Hence, by Lemma~\ref{pPrimeIndexConstrained}, $N_\F(P)$ is constrained and $P\in\F^s$ again by Lemma~\ref{subcentricEquiv}. 
\end{proof}

\begin{lemma}\label{pPowerIndexSubcentric}
 Let $\m{E}$ be a normal subsystem of $\F$ of $p$-power index and $T=\m{E}\cap S$. Then $\m{E}^s=\{P\in\F^s\colon P\leq T\}$. 
\end{lemma}

\begin{proof}
 By Lemma~\ref{Subcentric}, it remains only to prove that $\m{E}^s\subseteq \F^s$. By Lemma~\ref{ESubphi}, it is sufficient to prove $\m{E}^s\cap \F^f\subseteq \F^s$. Let $P\in\m{E}^s\cap\F^f$. By Lemma~\ref{FfEf}, $P\in\m{E}^f$. Hence, $N_\F(P)$ and $N_\m{E}(P)$ are saturated. It follows from the definition of the hyperfocal subgroup that $\mathfrak{hyp}(N_\F(P))\leq \mathfrak{hyp}(\F)\leq T$ and thus $\mathfrak{hyp}(N_\F(P))\leq N_T(P)$. For any $R\leq N_T(P)$, a $p^\prime$-element $\alpha\in\Aut_{N_\F(P)}(R)$ extends to a $p^\prime$-element $\hat{\alpha}\in\Aut_\F(PR)$ normalizing $P$. As $\m{E}$ is a subsystem of $\F$ of $p$-power index, $\hat{\alpha}\in O^p(\Aut_\F(PR))\leq \Aut_{\m{E}}(PR)$. Hence, $\alpha$ extends to an element of $\Aut_\m{E}(PR)$ normalizing $R$, which means $\alpha\in\Aut_{N_\m{E}(P)}(R)$. This shows that $N_\m{E}(P)$ is a subsystem of $N_\F(P)$ of $p$-power index. As $P\in\m{E}^s$, $N_\m{E}(P)$ is constrained by Lemma~\ref{subcentricEquiv}. Hence, by Lemma~\ref{pPowerIndexConstrained}, it follows that $N_\F(P)$ is constrained and $P\in\F^s$ by Lemma~\ref{subcentricEquiv}. 
\end{proof}

\begin{lemma}\label{KNormMainLemma}
 Let $Q\leq S$ and $K\leq \Aut(Q)$ be such that $Q$ is fully $K$-normalized. Let $P\leq N_S^K(Q)$ be such that $P$ is fully centralized in $N_\F^K(Q)$. Then $PQ$ is fully centralized in $\F$.
\end{lemma}

\begin{proof}
We note first that $N_\F^K(Q)$ is saturated by \cite[Theorem~I.5.5]{Aschbacher/Kessar/Oliver:2011} as $Q$ is fully $K$-normalized. 

\smallskip

\noindent{\em Step 1:} Let $P_0\in P^{N_\F^K(Q)}$. Then we show that $|C_S(P_0Q)|\leq |C_S(PQ)|$. Observe first that, by the extension axiom, an $\F$-isomorphism from $P_0$ to $P$ extends to a morphism 
\[\alpha\in \Hom_{N_\F^K(Q)}(P_0C_{N_S^K(Q)}(P_0),PC_{N_S^K(Q)}(P)).\] 
So we can fix $\alpha\in \Hom_{N_\F^K(Q)}(P_0C_{N_S^K(Q)}(P_0),PC_{N_S^K(Q)}(P))$ with $P_0\alpha=P$. It follows from the definition of $N_\F^K(Q)$ that $\alpha$ extends to $\hat{\alpha}\in \Hom_\F(P_0C_{N_S^K(Q)}(P_0)Q,PC_{N_S^K(Q)}(P)Q)$ with $Q\hat{\alpha}=Q$. Note $C_S(P_0Q)\leq C_{N_S^K(Q)}(P_0)$. As $P_0\hat{\alpha}=P_0\alpha=P$ we have $(P_0Q)\hat{\alpha}=PQ$. Hence $C_S(P_0Q)\alpha=C_S(P_0Q)\hat{\alpha}\subseteq C_S((P_0Q)\hat{\alpha})=C_S(PQ)$. Therefore, as $\alpha$ is injective, $|C_S(P_0Q)|=|C_S(P_0Q)\alpha|\leq |C_S(PQ)$. This finishes Step~1.

\smallskip

\noindent{\em Step 2:} We are now in a position to complete the proof. Let $\phi\in \Hom_\F(PQ,S)$ such that $(PQ)\phi$ is fully centralized. Our goal will be to show that $|C_S((PQ)\phi)|\leq |C_S(PQ)|$. Note that $\phi^{-1}$ restricts to an $\F$-isomorphism from $Q\phi$ to $Q$. As $Q$ is fully $K$-normalized it follows from \cite[Proposition~I.5.2]{Aschbacher/Kessar/Oliver:2011} that there exists $\chi\in \Aut_\F^K(Q)$ and $\psi\in \Hom_\F((Q\phi)N_S^{K^\phi}(Q\phi),S)$ such that $\psi|_{Q\phi}=(\phi^{-1}|_{Q\phi})\chi$. As $P\leq N_S^K(Q)$, we have $P\phi\leq N_S^{K^\phi}(Q\phi)$ by Lemma~\ref{KNormElementary}. In particular, $\psi$ is defined on $(PQ)\phi=(P\phi)(Q\phi)$ and $\phi\psi$ is defined on $PQ$. Note that $(\phi\psi)|_Q=\chi\in K$ and so $(\phi\psi)|_P$ is a morphism $N_\F^K(Q)$. Therefore, by Step~1, we have $|C_S((P\phi\psi)Q)|\leq |C_S(PQ)|$. Note that $Q\phi\psi=Q$ and thus $(PQ)\phi\psi=(P\phi\psi)Q$. Moreover, $\psi$ is defined on $C_S((PQ)\phi)\leq C_S(Q\phi)\leq N_S^{K^\phi}(Q\phi)$ and  $C_S((PQ)\phi)\psi\leq C_S((PQ)\phi\psi)=C_S((P\phi\psi)Q)$. Putting these properties together, we obtain $|C_S((PQ)\phi)|=|C_S((PQ)\phi)\psi|\leq |C_S((P\phi\psi)Q)|\leq |C_S(PQ)|$. As $(PQ)\phi$ is fully centralized, it follows that $PQ$ is fully centralized as well.
\end{proof}

\begin{lemma}\label{KNormSubcentric2a}
 Let $Q\in\F$ and $K\leq \Aut(Q)$ be such that $Q$ is fully $K$-normalized. Then $PQ\in\F^s$ for every $P\in N^K_\F(Q)^s$.

\smallskip

A similar result holds for centric and quasicentric subgroups: For every $P\in N_\F^K(Q)^c$ we have $PQ\in\F^c$, and for every $P\in N_\F^K(Q)^q$ we have $PQ\in\F^q$.
\end{lemma}

\begin{proof}
Let $P\leq N_S^K(Q)$. We want to show that $PQ\in\F^s$ if $P\in N_\F^K(Q)^s$, that $PQ\in\F^c$ if $P\in N_\F^K(Q)^c$ and that $PQ\in\F^q$ if $P\in N_\F^K(Q)^q$. Since the collections of centric, quasicentric and subcentric subgroups are closed under taking conjugates in the respective fusion system, we can replace $P$ by a suitable $N_\F^K(Q)$-conjugate and will assume without loss of generality that $P$ is fully centralized in $N_\F^K(Q)$. Then $PQ$ is fully centralized in $\F$ by Lemma~\ref{KNormMainLemma}. 

\smallskip

Assume first $P\in N_\F^K(Q)^s$. Then by Lemma~\ref{subcentricEquiv}, $\m{C}:=C_{N_\F^K(Q)}(P)$ is a constrained (saturated) subsystem. Note that $\m{C}=N_\F^{\tilde{K}}(PQ)$ where $\tilde{K}:=\{\alpha\in\Aut(PQ)\colon \alpha|_Q\in K,\;\alpha|_P=\id_P\}$. Moreover, as $PQ$ is fully $\F$-centralized, $C_\F(PQ)$ is saturated. Hence, by Lemma~\ref{WeaklyNormalCentInKNorm}(b), $C_\F(PQ)$ is constrained. Since $PQ$ is fully $\F$-centralized, Lemma~\ref{subcentricEquiv} implies $PQ\in\F^s$ as required.

\smallskip

If $P\in N_\F^K(Q)^c$ then $C_S(PQ)=C_{C_S(Q)}(P)\leq C_{N_S^K(Q)}(P)\leq P\leq PQ$ and so $PQ$ is centric since $PQ$ is fully centralized. Suppose now $P\in N_\F^K(Q)^q$. Then $\m{C}:=C_{N_\F^K(Q)}(P)$ is the fusion system of the $p$-group $C_{N_S^K(Q)}(P)$. As $C_\F(PQ)$ is a subsystem of $\m{C}$, it follows that $\Aut_{C_\F(PQ)}(R)$ is a $p$-group for every $R\leq C_S(PQ)$. So if $R\in C_\F(PQ)^f$ then $\Aut_{C_\F(PQ)}(R)=\Aut_{C_S(PQ)}(R)$. Since $PQ$ is fully centralized, $C_\F(PQ)$ is saturated. So it follows from Alperin's fusion theorem (see \cite[Theorem~I.3.6]{Aschbacher/Kessar/Oliver:2011}) that $C_\F(PQ)=\F_{C_S(PQ)}(C_S(PQ))$. As $PQ$ is fully centralized, this implies $PQ\in\F^q$ which completes the proof.
\end{proof}

\begin{lemma}\label{KNormSubcentricNormal}
 Let $R$ be a subgroup of $S$ normal in $\F$ and $K\unlhd \Aut_\F(R)$. Then $N_\F^K(R)$ is weakly normal in $\F$ and  $N_\F^K(R)^s=\{P\in\F^s\colon P\leq N_S^K(R)\}$. In particular, $C_\F(R)^s=\{P\in\F^s\colon P\leq C_S(R)\}$. 
\end{lemma}

\begin{proof}
By Lemma~\ref{WeaklyNormalCentInKNorm}(a), $N_\F^K(R)$ is weakly normal in $\F=N_\F^{\Aut_\F(R)}(R)$. Hence, by Lemma~\ref{Subcentric}, every $P\in\F^s$ with $P\leq N_S^K(R)$ is a member of $N_\F^K(R)^s$. Let now $P\in N_\F^K(R)^s$. By Lemma~\ref{KNormSubcentric2a}, $PR\in\F^s$. So by Lemma~\ref{NormalTimesSubcentric}, $P\in\F^s$.
\end{proof}

\begin{lemma}\label{NormalizerSubcentric2}
Let $Q\in\F^f$ and $P\in\F^s$ be such that $P\leq N_S(Q)$. Then $P\in N_\F(Q)^s$.
\end{lemma}

\begin{proof}
By Lemma~\ref{subcentricProp}, $PQ\in\F^s$. Moreover, by Lemma~\ref{NormalTimesSubcentric}, $P\in N_\F(Q)^s$ if $PQ\in N_\F(Q)^s$. Hence, replacing $P$ by $PQ$, we may assume $Q\leq P$. Moreover, replacing $P$ by a $N_\F(Q)$-conjugate, we may assume that $P$ is fully centralized in $N_\F(Q)$. Then $P=PQ$ is fully centralized in $\F$ by Lemma~\ref{KNormMainLemma}. So by Lemma~\ref{subcentricEquiv}, $C_\F(P)$ is constrained. As $C_\F(P)=C_{N_\F(Q)}(P)$, Lemma~\ref{subcentricEquiv} applied with $N_\F(Q)$ in place of $\F$ gives $P\in N_\F(Q)^s$.
\end{proof}

\begin{lemma}\label{KNormSubcentric2}
 Let $Q\in\F$ and $K\unlhd\Aut_\F(Q)$ be such that $Q$ is fully $K$-normalized. Then $\{P\in\F^s\colon P\leq N_S^K(Q)\}\subseteq N_\F^K(Q)^s$.
\end{lemma}

\begin{proof}
By \cite[Lemma~I.2.6(c)]{Aschbacher/Kessar/Oliver:2011}, there exists a morphism $\alpha\in\Hom_\F(N_S(Q),S)$ such that $Q\alpha\in\F^f$. By Lemma~\ref{KNormElementary}, $N^K_S(Q)\alpha\leq N^{K^\alpha}_S(Q\alpha)$. As $Q$ is fully $K$-normalized, it follows that $Q\alpha$ is fully $K^\alpha$-normalized and $N_S^K(Q)\alpha=N_S^{K^\alpha}(Q\alpha)$. It is straightforward to check that $\alpha$ induces an isomorphism from $N_\F^K(Q)$ to $N_\F^{K^\alpha}(Q\alpha)$. Thus, by Lemma~\ref{IsoSubcentric}, for any $P\leq N_S^K(Q)$, we have  $P\alpha\in N^{K^\alpha}_\F(Q\alpha)^s$ if and only if $P\in N_\F^K(Q)^s$. Hence, as $\F^s$ is invariant under $\F$-conjugation, replacing $Q$ by $Q\alpha$, we may assume that $Q\in\F^f$. Then $N_\F(Q)$ is saturated and, as $N_\F^K(Q)=N_{N_\F(Q)}^K(Q)$, it follows from Lemma~\ref{KNormSubcentricNormal} that $N_\F^K(Q)^s=\{P\in N_\F(Q)^s\colon P\leq N_S^K(Q)\}$. If $P\in\F^s$ with $P\leq N_S^K(Q)$ then $P\in N_\F(Q)^s$ by Lemma~\ref{NormalizerSubcentric2}, and therefore $P\in N_\F^K(Q)^s$. This proves the assertion.
\end{proof}

\begin{lemma}\label{fsfrc}
 Let $Q\in\F^s\cap\F^f$ be such that $Q=O_p(N_\F(Q))$. Then $Q\in\F^{frc}$.
\end{lemma}

\begin{proof}
 As $Q\in\F^s\cap\F^f$, we have $Q=O_p(N_\F(Q))\in\F^c$ by definition of subcentric subgroups. Moreover Theorem~\ref{subcentricEquiv} yields that $N_\F(Q)$ is constrained. So by Theorem~\ref{Model1}, there exists a model $G$ for $N_\F(Q)$ and $O_p(G)=O_p(N_\F(Q))=Q$. Note $\Aut_\F(Q)=\Aut_{N_\F(Q)}(Q)\cong G/C_G(Q)=G/Z(Q)$. Then $O_p(\Aut_\F(Q))\cong O_p(G/Z(Q))=Q/Z(Q)\cong \Inn(Q)$ and so $Q$ is radical. 
\end{proof}

\begin{proof}[Proof of Proposition~\ref{SubcentricProperties1}]
 This follows from Lemma~\ref{NormalTimesSubcentric}, Lemma~\ref{SubcentricModCentral}, Lemma~\ref{KNormSubcentric2a} and Lemma~\ref{KNormSubcentric2}. Compare also Lemma~\ref{NormalizerSubcentric2}.
\end{proof}

\begin{proof}[Proof of Proposition~\ref{SubcentricProperties2}]
 The proposition follows from Lemma~\ref{ESubphi}, Lemma~\ref{Subcentric}, Lemma~\ref{pPrimeIndexSubcentric}, Lemma~\ref{pPowerIndexSubcentric} and Lemma~\ref{KNormSubcentricNormal}.
\end{proof}

\section{The proof of Theorem~\ref{SubcentricEF}}\label{SubcentricPropertiesSection2}

\noindent\textbf{Throughout this section, $\F$ is assumed to be saturated. Moreover, $\m{E}$ is always a normal subsystem of $\F$ over $T\leq S$.}

\smallskip

In this section we will prove Theorem~\ref{SubcentricEF} using the theory of components of fusion systems which was developed by Aschbacher. So our goal will be to prove that $PC_S(\m{E})$ is subcentric in $\F$, for every subcentric subgroup $P$ of $\m{E}$. 
The subgroup $C_S(\m{E})$ was introduced in \cite[Chapter~6]{Aschbacher:2011}. We will use throughout the following characterization: The subgroup $C_S(\m{E})$ is the largest subgroup $X$ of $C_S(T)$ such that $\m{E}\subseteq C_\F(X)$.

\begin{lemma}\label{ENormNFCSE}
The subsystem $\m{E}$ is also a normal subsystem of $N_\F(C_S(\m{E}))$. 
\end{lemma}

\begin{proof}
Recall $\m{E}\subseteq C_\F(C_S(\m{E}))\subseteq N_\F(C_S(\m{E}))$. Clearly $\m{E}$ is weakly normal in $N_\F(C_S(\m{E}))$. Set $T=\m{E}\cap S$. As $\m{E}\unlhd\F$, every element $\phi\in\Aut_\m{E}(T)$ extends to $\ov{\phi}\in\Aut_{\F}(TC_S(T))$ such that $[C_S(T),\ov{\phi}]\leq T$. Since $C_S(\m{E})\leq C_S(T)$ and $C_S(\m{E})$ is strongly closed in $\F$ by \cite[(6.7)(2)]{Aschbacher:2011}, we have $C_S(\m{E})\ov{\phi}=C_S(\m{E})$. Hence, $\ov{\phi}\in\Aut_{N_\F(C_S(\m{E}))}(TC_S(T))$. This shows the assertion.
\end{proof}

\begin{lemma}\label{PFullNorm}
Let $Q\in\F^f$ such that $Q=(Q\cap T)C_S(\m{E})$. Then $Q\cap T\in\m{E}^f$.
\end{lemma}

\begin{proof}
 Set $P:=Q\cap T$. Let $\alpha_0\in \Hom_\m{E}(P,T)$ be such that $P\alpha_0$ is fully $\m{E}$-normalized. By the characterization of $C_S(\m{E})$ above, $\alpha_0$ extends to $\alpha\in\Hom_\F(Q,S)$ such that $\alpha$ fixes every element of $C_S(\m{E})$. In particular, $C_S(\m{E})\alpha=C_S(\m{E})$ and $Q\alpha=(P\alpha)C_S(\m{E})$. Moreover, $P\alpha=(Q\cap T)\alpha\leq Q\alpha\cap T$ and $(Q\alpha\cap T)\alpha^{-1}\leq Q\cap T$. So $Q\alpha\cap T=(Q\cap T)\alpha=P\alpha$. 

\smallskip

As $Q=PC_S(\m{E})$ and $C_S(\m{E})\leq C_S(T)$, we have $N_T(P)\leq N_T(Q)$. As $P=Q\cap T$, $N_T(Q)\leq N_T(P)$. Hence, $N_T(P)=N_T(Q)$. Similarly, $N_T(Q\alpha)=N_T(P\alpha)$. By \cite[Lemma~I.2.6(c)]{Aschbacher/Kessar/Oliver:2011}, there exists $\beta\in\Hom_\F(N_S(Q\alpha),S)$ such that $Q\alpha\beta=Q$. For such $\beta$, we have $N_T(Q\alpha)\beta\leq N_T(Q)$. Thus $|N_T(P\alpha_0)|=|N_T(P\alpha)|=|N_T(Q\alpha)|\leq |N_T(Q)|=|N_T(P)|$. Hence, $P\in\m{E}^f$ as $P\alpha_0\in\m{E}^f$.
\end{proof}

\begin{lemma}\label{NEPNormalNFQ}
Let $Q\in\F^f$ such that $Q=(Q\cap T)C_S(\m{E})$. Then $N_\m{E}(Q\cap T)$ is weakly normal in $N_\F(Q)$.
\end{lemma}

\begin{proof}
Set $P:=Q\cap T$. By Lemma~\ref{PFullNorm}, $P\in\m{E}^f$. By assumption $Q\in\F^f$, so both $N_\m{E}(P)$ and $N_\F(Q)$ are saturated. Every morphism $\alpha\in\Hom_{N_\m{E}(P)}(A,B)$ ($A,B\leq N_T(P)$) extends to an element of $\Hom_{\m{E}}(AP,BP)$ normalizing $P$, which then by definition of $C_S(\m{E})$ extends to $\ov{\alpha}\in \Hom_\F(APC_S(\m{E}),BPC_S(\m{E}))$ centralizing $C_S(\m{E})$. As $Q=PC_S(\m{E})$, it follows $Q\ov{\alpha}=Q$ and so $\alpha$ is a morphism in $N_\F(Q)$. This shows that $N_\m{E}(P)$ is a subsystem of $N_\F(Q)$.  Hence, it remains to prove only that $N_\m{E}(P)$ is invariant in $N_\F(P)$. 

\smallskip

We prove the strong invariance condition as stated in \cite[Proposition~I.6.4(d)]{Aschbacher/Kessar/Oliver:2011}. Let $A\leq B\leq N_T(P)$, $\phi\in\Hom_{N_\m{E}(P)}(A,B)$ and $\psi\in\Hom_{N_\F(Q)}(B,N_T(P))$. We need to prove that $(\psi|_A)^{-1}\phi\psi\in\Hom_{N_\m{E}(P)}(A\psi,B\psi)$. By definition of the normalizer subsystems, $\phi$ extends to $\ov{\phi}\in\Hom_\m{E}(AP,BP)$ with $P\ov{\phi}=P$, and $\psi$ extends to $\ov{\psi}\in\Hom_\F(BQ,N_T(P)Q)$ with $Q\ov{\psi}=Q$. As $T$ is strongly closed and, by assumption, $P=Q\cap T$, we have $P\ov{\psi}=P$ and thus $\widehat{\psi}:=\ov{\psi}|_{BP}\in\Hom_\F(BP,N_T(P))$. Since the strong invariance condition holds for $(\m{E},\F)$, we have that $(\widehat{\psi}|_{AP})^{-1}\ov{\phi}\widehat{\psi}$ is a morphism in $\m{E}$. Moreover, $P(\widehat{\psi}|_{AP})^{-1}\ov{\phi}\widehat{\psi}=P$ and $(\widehat{\psi}|_{AP})^{-1}\ov{\phi}\widehat{\psi}$ extends $(\psi|_A)^{-1}\phi\psi$. So $(\psi|_A)^{-1}\phi\psi$ is a morphism in $N_\m{E}(P)$ as required.
\end{proof}

\begin{lemma}\label{MainComponents}
 Let $\m{E}$ be a normal subsystem of $\F$ and $\m{C}$ a component of $\F$. Then $\m{C}\subseteq \m{E}$ or $\m{C}\cap S\leq C_S(\m{E})$. 
\end{lemma}

\begin{proof}
Suppose $\m{C}$ is not contained in $\m{E}$. Then $\m{C}$ is in the set $J$ of components of $\F$ which are not components of $\m{E}$. Then $\m{D}:=\prod_{\m{C}'\in J}\m{C}'$ is a well-defined subsystem of $\F$ containing $\m{C}$ by \cite[(9.8)(2)]{Aschbacher:2011}. It is furthermore shown in \cite[(9.13)]{Aschbacher:2011} that $\m{E}\m{D}$ is well-defined and a central product of $\m{E}$ and $\m{D}$. If $\F$ is the central product of two subsystems $\F_1$ and $\F_2$ then, by the construction of central products in \cite[Chapter~2]{Aschbacher:2011}, $\F_2\cap S\leq C_S(\F_1)$. So $\m{C}\cap S\leq \m{D}\cap S\leq C_S(\m{E})$. 
\end{proof}

\begin{proof}[Proof of Theorem~\ref{SubcentricEF}]
Let $\m{E}$ be a normal subsystem of $\F$ over $T\leq S$. Let $P\in\m{E}^s$ and set $Q:=PC_S(\m{E})$. We need to show that $Q\in\F^s$.

\smallskip
\noindent{\em Step 1:} We show that it is enough to prove the assertion in the case that $Q\in\F^f$ and $P=Q\cap T$. For that take $\phi\in\Hom_\F(Q,S)$ such that $Q\phi$ is fully $\F$-normalized. Then by Lemma~\ref{ESubphi}, $P\phi\in\m{E}^s$. Moreover, as $C_S(\m{E})$ is strongly closed by \cite[(6.7)(2)]{Aschbacher:2011}, $C_S(\m{E})\phi=C_S(\m{E})$ and thus 
$Q\phi=(P\phi)C_S(\m{E})$. So replacing $(P,Q)$ by $(P\phi,Q\phi)$, we may assume that $Q$ is fully $\F$-normalized. Note also that $P\leq Q\cap T$, so by Proposition~\ref{subcentricProp}, $Q\cap T$ is subcentric in $\m{E}$. Moreover, $Q=(Q\cap T)C_S(\m{E})$. Hence, replacing $P$ by $Q\cap T$, we may assume that $P=Q\cap T$. 

\smallskip 

From now on we assume that $Q\in\F^f$ and $P=Q\cap T$.

\smallskip

\noindent{\em Step~2:} We show that $E(N_\F(C_S(\m{E})))\subseteq \m{E}$. Let $\m{C}$ be a component of $N_\F(C_S(\m{E}))$. By Lemma~\ref{ENormNFCSE}, $\m{E}$ is normal in $N_\F(C_S(\m{E}))$. Therefore, by Lemma~\ref{MainComponents}, $\m{C}\subseteq \m{E}$ or $C:=\m{C}\cap S\leq C_S(\m{E})$. 

\smallskip

By \cite[(9.9)(1)]{Aschbacher:2011}, a component of a saturated fusion system centralizes every normal subgroup of the same fusion system. Hence, as $C_S(\m{E})$ is normal in $N_\F(C_S(\m{E}))$, we have $\m{C}\subseteq C_\F(C_S(\m{E}))$. Assume $C\leq C_S(\m{E})$. As $\m{C}\subseteq  C_\F(C_S(\m{E}))$ this means that $C$ is abelian, contradicting \cite[(9.1)(2)]{Aschbacher:2011} and the fact that $\m{C}$ is quasisimple. This proves $\m{C}\subseteq \m{E}$ and, as $\m{C}$ was arbitrary, $E(N_\F(C_S(\m{E})))\subseteq \m{E}$. 

\smallskip

\noindent{\em Step~3:} We complete the proof by showing that $Q$ is subcentric in $\F$.
Suppose this is not true. As we assume  that $Q$ is fully normalized, this means by Lemma~\ref{subcentricEquiv} that $N_\F(Q)$ is not constrained. Thus, by \cite[(14.2)]{Aschbacher:2011}, $E(N_\F(Q))\neq 1$. By \cite[(6.7)(2)]{Aschbacher:2011}, $C_S(\m{E})$ is strongly closed in $\F$.  So as $C_S(\m{E})\leq Q$, we have $N_\F(Q)=N_{N_\F(C_S(\m{E}))}(Q)$. Since $Q$ is fully normalized in $\F$ and $C_S(\m{E})\unlhd S$, $N_\F(C_S(\m{E}))$ is saturated and $Q$ is fully normalized in $N_\F(C_S(\m{E}))$. Thus, by Aschbacher's version of the L-Balance Theorem for fusion systems \cite[Theorem~7]{Aschbacher:2011}, $E(N_\F(Q))=E(N_{N_\F(C_S(\m{E}))}(Q))\subseteq E(N_\F(C_S(\m{E})))$. So by Step~2, $E(N_\F(Q))\subseteq \m{E}$. 

\smallskip

Let $\m{D}$ be a component of $N_\F(Q)$ and $D=S\cap \m{D}$. By assumption, $P=Q\cap T$ and thus $Q=(Q\cap T)C_S(\m{E})$. Hence, by Lemma~\ref{PFullNorm}, $P$ is fully $\m{E}$-normalized. Moreover, by Lemma~\ref{NEPNormalNFQ}, $N_{\m{E}}(P)$ is weakly normal in $N_\F(Q)$. It follows from the latter fact and Lemma~\ref{OpE}(b) that $O_p(N_\m{E}(P))$ is normal in $N_\F(Q)$. By \cite[(9.6)]{Aschbacher:2011} the component $\m{D}$ of $N_\F(Q)$ centralizes every normal subgroup of $N_\F(Q)$. So in particular, $\m{D}\subseteq C_{N_\F(Q)}(O_p(N_\m{E}(P)))$ and thus $[D,O_p(N_\m{E}(P))]=1$. As $E(N_\F(Q))\subseteq \m{E}$, we have $D\leq T$. Hence, $D\leq C_T(O_p(N_\m{E}(P)))=Z(O_p(N_\m{E}(P)))$, because $P$ is subcentric and fully normalized in $\m{E}$ and so $O_p(N_{\m{E}}(P))$ is centric in $\m{E}$. Thus, $D$ is abelian, again contradicting \cite[(9.1)(2)]{Aschbacher:2011} and the fact that $\m{D}$ is quasisimple.
\end{proof}

\section{Partial groups, localities and transporter systems}\label{LocalitiesSection0}

In this section we recall some background on localities. At the end we spell out the connection between localities and transporter systems. 

\subsection{Partial groups}
Adapting the notation from \cite{Chermak:2013} and \cite{Chermak:2015}, we denote the set of words in a set $\L$ by $\W(\L)$. Moreover, we write $\emptyset$ for the empty word, and $v_1\circ v_2\circ\dots\circ v_n$ for the concatenation of words $v_1,\dots,v_n\in\W(\L)$. Roughly speaking, a partial group is a set $\L$ together with a product which is only defined on certain words in $\L$, and an inversion map $\L\rightarrow \L$ which is an involutory bijection, subject to certain axioms. We refer the reader to \cite[Definition~2.1]{Chermak:2013} or \cite[Definition~1.1]{Chermak:2015} for the precise definition of a partial group, and to the elementary properties of partial groups stated in \cite[Lemma~2.2]{Chermak:2013} or \cite[Lemma~1.4]{Chermak:2015}.

\bigskip

\textbf{For the remainder of this section let $\L$ be a partial group with product $\Pi\colon \D\rightarrow \L$ defined on the domain $\D\subseteq\W(\L)$.} 

\bigskip

It follows from the axioms of a partial group that $\emptyset\in\D$. We set $\One=\Pi(\emptyset)$. Moreover, given a word $v=(f_1,\dots,f_n)\in\D$, we write sometimes $f_1f_2\dots f_n$ for the product $\Pi(v)$.

\smallskip

A \textit{partial subgroup} of $\L$ is a subset $\H$ of $\L$ such that $f^{-1}\in\H$ for all $f\in\H$ and $\Pi(w)\in\H$ for all $w\in\W(\H)\cap\D$. Note that $\emptyset\in\W(\H)\cap\D$ and thus $\One=\Pi(\emptyset)\in\H$ if $\H$ is a partial subgroup of $\L$. It is easy to see that a partial subgroup of $\L$ is always a partial group itself whose product is the restriction of the product $\Pi$ to $\W(\H)\cap\D$. Observe furthermore that $\L$ forms a group in the usual sense if $\W(\L)=\D$; see \cite[Lemma~1.3]{Chermak:2015}. So it makes sense to call a partial subgroup $\H$ of $\L$ a \textit{subgroup of $\L$} if $\W(\H)\subseteq\D$. In particular, we can talk about \textit{$p$-subgroups of $\L$} meaning subgroups of $\L$ which are finite and whose order is a power of $p$. 

For any $g\in\L$, $\D(g)$ denotes the set of $x\in\L$ with $(g^{-1},x,g)\in\D$. Thus, $\D(g)$ denotes the set of elements $x\in\L$ for which the conjugation $x^g:=\Pi(g^{-1},x,g)$ is defined. %By the axioms of a partial group, $(g^{-1},g)\in\D$ and $\Pi(g^{-1},g)=\One$. So by Remark~\ref{Ones}, $(g^{-1},\One,g)\in\D$ and $\Pi((g^{-1},\One,g))=\One$. Thus, we have always $\One\in\D(g)$ and $\One^g=\One$.

\smallskip

If $g\in\L$ and $X\subseteq \D(g)$ we set $X^g:=\{x^g\colon x\in X\}$. If we write $X^g$ for some $g\in\L$ and some subset $X\subseteq \L$, we will always implicitly mean that $X\subseteq\D(g)$. Similarly, if we write $x^g$ for $x,g\in\L$, we always mean that $x\in\D(g)$. 

\smallskip

If $X$ is a subset of $\L$ then we set
\[N_\L(X)=\{g\in\L\colon X^g=X\}\mbox{ and }C_\L(X)=\{g\in\L\colon x^g=x\mbox{ for all }x\in X\}.\]
Note that $C_\L(X)\subseteq N_\L(X)$. It follows easily from the axioms of a partial group that $\One$ is contained in the centralizer of any subset of $\L$; see \cite[Lemma~2.5(d)]{Chermak:2013}.

\smallskip

If $X$ and $Y$ are subsets of $\L$ then set $N_Y(X)=N_\L(X)\cap Y$ and $C_Y(X)=C_\L(X)\cap Y$. Moreover, define
\[Z(\L):=C_\L(\L).\]
Generalizing the notion of the normalizer, we set $N_\L(X,Y)=\{f\in\L\colon X^f\subseteq Y\}$.

\smallskip

Since there is a natural notion of conjugation, there is also a natural notion of partial normal subgroups of partial groups. Namely, a partial subgroup $\N$ of $\L$ is called a \textit{partial normal subgroup} of $\L$ if $n^f\in\N$ for all $f\in\L$ and all $n\in\N\cap\D(f)$. 

Let $\L'$ be a partial group with domain $\D'$ and product $\Pi'\colon\D'\rightarrow\L'$. Let $\One'=\Pi'(\emptyset)$ be the identity in $\L'$. Let $\beta\colon\L\rightarrow\L'$ be a map and let $\beta^*\colon \W(\L)\rightarrow\W(\L')$ be the induced map. Recall from \cite[Definition~1.11]{Chermak:2015} that $\beta$ is called a homomorphism of partial groups if $\D\beta^*\subseteq\D'$ and $\Pi'(v\beta^*)=(\Pi(v))\beta$ for all $v\in\D$. If $\beta$ is a homomorphism of partial groups, define the kernel of $\beta$ via
\[\ker(\beta)=\{f\in\L\colon \beta(f)=\One'\}.\]
By \cite[Lemma~1.14]{Chermak:2015}, the kernel of a homomorphism of partial groups forms always a partial normal subgroup. 

\smallskip

A homomorphism $\beta\colon\L\rightarrow\L'$ of partial groups is called an \textit{isomorphism} of partial groups if $\D\beta^*=\D'$ and $\beta$ is injective. As every word in $\L'$ of length one is in $\D'$, the condition $\D\beta^*=\D'$ implies that $\beta$ is surjective. Thus, every isomorphism of partial groups is bijective.

\subsection{Localities}

\begin{definition}\label{LocalityDefinition}
Let $S$ be a $p$-subgroup of $\L$ and let $\Delta$ be a non-empty set of subgroups of $S$. 

\smallskip

The set $\Delta$ is said to be \textit{closed under taking $\L$-conjugates and overgroups in $S$} if for any $P\in\Delta$ the following holds: For every $g\in\L$ with $P\subseteq \D(g)$ and $P^g\subseteq S$ we have $P^g\in\Delta$ (so in particular, $P^g$ is a subgroup of $S$), and for every subgroup $Q$ of $S$ containing $P$ we have $Q\in\Delta$.

\smallskip

We say that $(\L,\Delta,S)$ is a \textit{locality} if the partial group $\L$ is finite as a set and the following conditions hold:
\begin{itemize}
\item[(L1)] $S$ is maximal with respect to inclusion among the $p$-subgroups of $\L$.
\item[(L2)] $\D$ is the set of words $(f_1,\dots,f_n)\in\W(\L)$ such that there exist $P_0,\dots,P_n\in\Delta$ with 
\begin{itemize}
\item [(*)] $P_{i-1}\subseteq \D(f_i)$ and $P_{i-1}^{f_i}=P_i$ for all $i=1,\dots,n$.
\end{itemize}
\item[(L3)] The set $\Delta$ is closed under taking $\L$-conjugates and overgroups in $S$.%, i.e. for any $P\in\Delta$ the following hold: We have $P^g\in\Delta$ for every $g\in\L$ with $P\subseteq \D(g)$ and $P^g\subseteq S$\footnote{Note that this implies in particular that $P^g$ is a subgroup of $S$}, and we have  $Q\in\Delta$ for every subgroup $Q$ of $S$ containing $P$.
\end{itemize}
If $(\L,\Delta,S)$ is a locality and $v=(f_1,\dots,f_n)\in\W(\L)$, then we say that $v\in\D$ via $P_0,\dots,P_n$ (or $v\in\D$ via $P_0$), if $P_0,\dots,P_n$ are elements of $\Delta$ such that (*) holds.
\end{definition}

\begin{rmk}
Our definition of a locality differs slightly from the one given by Chermak in \cite{Chermak:2013} and \cite{Chermak:2015}, but can be shown to be equivalent. It can be easily seen that a locality as defined by Chermak is a triple $(\L,\Delta,S)$ such that $\L$ is a finite partial group, $S$ is a $p$-subgroup of $\L$, $\Delta$ is a set of subgroups of $S$, and such that the conditions (L1) and (L2) together with the following condition hold:
\begin{itemize}
 \item[(L3')] For any subgroup $Q$ of $S$, for which there exist $P\in\Delta$ and $g\in\L$ with $P\subseteq \D(g)$ and $P^g\leq Q$, we have $Q\in\Delta$.
\end{itemize}
Clearly (L3) implies (L3'). If $(\L,\Delta,S)$ is a locality in Chermak's definition, then it is shown in \cite[Proposition~2.6(c)]{Chermak:2015} that $P^g$ is a subgroup of $S$ and thus an element of $\Delta$ if $g\in\L$ with $P\subseteq\D(g)$ and $P^g\subseteq S$. Moreover, $P\subseteq \D(\One)$ and $P^{\One}=P$. Therefore, if $(\L,\Delta,S)$ is a locality in Chermak's definition then (L3) holds and $(\L,\Delta,S)$ is indeed also a locality in our definition.
\end{rmk}

If $(\L,\Delta,S)$ is a locality and $g\in\L$, we set 
\[S_g:=\{s\in\D(g)\cap S\colon s^g\in S\}.\]
More generally, if $w=(g_1,\dots,g_n)\in\W(\L)$, we write $S_w$ for the set of $s\in S$ such that there exist elements $s=s_0,s_1,\dots,s_n$ of $S$ such that, for $i=1,\dots,n$, we have $s_{i-1}\in\D(g_i)$ and $s_{i-1}^{g_i}=s_i$.

\begin{lemma}[Important properties of localities]\label{LocalitiesProp}
Let $(\L,\Delta,S)$ be a locality. Then the following hold:
\begin{itemize}
\item [(a)] $N_\L(P)$ is a subgroup of $\L$ for each $P\in\Delta$.
\item [(b)] Let $P\in\Delta$ and $g\in\L$ with $P\subseteq S_g$. Then $Q:=P^g\in\Delta$, $N_\L(P)\subseteq \D(g)$ and 
$$c_g\colon N_\L(P)\rightarrow N_\L(Q)$$
is an isomorphism of groups. 
\item [(c)] Let $w=(g_1,\dots,g_n)\in\D$ via $(X_0,\dots,X_n)$. Then 
$$c_{g_1}\circ \dots \circ c_{g_n}=c_{\Pi(w)}$$
 is a group isomorphism $N_\L(X_0)\rightarrow N_\L(X_n)$.
\item [(d)] For every $g\in\L$, $S_g\in\Delta$. In particular, $S_g$ is a subgroup of $S$. Moreover, $S_g^g=S_{g^{-1}}$.
\item [(e)] For every $g\in\L$, $c_g\colon \D(g)\rightarrow \D(g^{-1})$ is a bijection with inverse map $c_{g^{-1}}$.
\item [(f)] For any $w\in\W(\L)$, $S_w$ is a subgroup of $S_{\Pi(w)}$, and $S_w\in\Delta$ if and only if $w\in\D$.
%\item [(g)]    
\end{itemize}
\end{lemma}

\begin{proof}
Properties (a),(b) and (c) correspond to the statements (a),(b) and (c) in \cite[Lemma~2.3]{Chermak:2015} except for the fact stated in (b) that $Q\in\Delta$, which is however clearly true if one uses our definition of a locality. Property (d) holds by \cite[Proposition~2.6(a),(b)]{Chermak:2015} and property (e) is stated in \cite[Lemma~2.5(c)]{Chermak:2013}. Property (f) is \cite[Corollary~2.7]{Chermak:2015}.
\end{proof}

We will use the properties stated in Lemma~\ref{LocalitiesProp} most of the time without reference. Note that, by parts (b) and (d), $c_g\colon S_g\rightarrow S$ is a homomorphism of groups, which by part (e) is injective. If $(\L,\Delta,S)$ is a locality, then $\F_S(\L)$ is the fusion system generated by the conjugation maps $c_g\colon S_g\rightarrow S$ with $g\in\L$. Equivalently, $\F_S(\L)$ is generated by the conjugation maps between subgroups of $\Delta$, or by the conjugation maps between subgroups of $S$. The elements of $\Delta$ which are fully $\F_S(\L)$-normalized are characterized in the following lemma.

\begin{lemma}\label{LocalitiesPropG}
Let $(\L,\Delta,S)$ be a locality and $P\in\Delta$. Then $P$ is fully-normalized in $\F_S(\L)$ if and only if $N_S(P)\in\Syl_p(N_\L(P))$. If so then 
$N_\F(P)=\F_{N_S(P)}(N_\L(P))$ and $C_\F(P)=\F_{C_S(P)}(C_\L(P))$. Moreover, for any $P\in\Delta$ there exists $f\in\L$ such that $N_S(P)\leq S_f$ and $N_S(P^f)\in\Syl_p(N_\L(P^f))$.
\end{lemma}

\begin{proof}
It is easy to check that $N_\F(P)=\F_{N_S(P)}(N_\L(P))$ and $C_\F(P)=\F_{C_S(P)}(C_\L(P))$. The remaining claims follow from Proposition~\ref{LocalitiesProp}(b) and \cite[Lemma~2.10]{Chermak:2015}.
\end{proof}

\begin{lemma}\label{NormalizerCentralizerPartialSubgroup}
 Let $(\L,\Delta,S)$ be a locality and $R$ a subgroup of $S$. Then $N_\L(R)$ and $C_\L(R)$ are partial subgroups of $\L$. 
\end{lemma}

\begin{proof}
 By \cite[Lemma~2.19(a)]{Chermak:2013}, $N_\L(R)$ is a partial subgroup of $\L$. So it remains to prove that $C_\L(R)$ is a partial subgroup. It follows from Lemma~\ref{LocalitiesProp}(e) that $C_\L(R)$ is closed under inversion. Let now $w=(f_1,\dots,f_n)\in\D\cap\W(C_\L(R))$. Then $w\in\D$ via a sequence $P_0,\dots,P_n$ of elements of $\Delta$. Since $\Delta$ is closed under taking overgroups in $S$, $Q_i:=\<P_i,R\>\in\Delta$ for $i=0,1,\dots,n$. Moreover, for $i=1,\dots,n$, $Q_{i-1}\leq S_{f_i}$, $c_{f_i}|_R=\id_R$, and $Q_{i-1}^{f_i}=Q_i$ as $c_{f_i}\colon S_{f_i}\rightarrow S$ is a homomorphism of groups. So $w\in\D$ via $Q_0,\dots,Q_n$ and it follows from Lemma~\ref{LocalitiesProp}(c) that $R\leq Q_0\leq S_{\Pi(w)}$ and $c_{\Pi(w)}|_R=\id_R$, So $\Pi(w)\in C_\L(R)$. 
\end{proof}

The following remark is used throughout, usually without reference:

\begin{lemma}\label{MorphismAsConjugation}
 Let $P\in\Delta$ and $\phi\in\Hom_{\F_S(\L)}(P,S)$. Then there exists $g\in\L$ with $P\leq S_g$ and $\phi=c_g|_P$.
\end{lemma}

\begin{proof}
By definition of $\F_S(\L)$, $\phi$ is the composition of suitable restrictions of conjugation maps $c_{g_1},c_{g_2},\dots, c_{g_n}$ with $g_1,g_2,\dots,g_n\in\L$. Then $(g_1,\dots,g_n)\in\D$ via $P$, Moreover, setting $g=\Pi(g_1,g_2,\dots,g_n)$,  Lemma~\ref{LocalitiesProp}(c) implies $P\leq S_g$ and $\phi=c_g|_P$. 
\end{proof}

\subsection{Projections of localities}\label{LocalityProjectionSection}

There is a theory of morphisms and factor systems of fusion systems, where factor systems are formed modulo strongly closed subgroups and similarly the kernels of morphisms are strongly closed. We refer the reader to \cite[Section~II.5]{Aschbacher/Kessar/Oliver:2011} for details.

\smallskip

Let $\F$ and $\F'$ be fusion systems over $S$ and $S'$ respectively. Then we say that a group homomorphism $\alpha\colon S\rightarrow S'$ induces a morphism from $\F$ to $\F'$ if for each $\phi\in\Hom_\F(P,Q)$ there exists $\psi\in\Hom_{\F'}(P\alpha,Q\alpha)$ such that $(\alpha|_P)\psi=\phi(\alpha|_Q)$. Such $\psi$ is then uniquely determined, so $\alpha$ induces a map $\alpha_{P,Q}\colon\Hom_\F(P,Q)\rightarrow \Hom_{\F'}(P\alpha,Q\alpha)$. Together with the map $P\mapsto P\alpha$ from the set of objects of $\F$ to the set of objects of $\F'$ this gives a functor from $\F$ to $\F'$. Moreover, $\alpha$ together with the maps $\alpha_{P,Q}$ ($P,Q\leq S$) is a morphism of fusion systems in the sense of \cite[Definition~II.2.2]{Aschbacher/Kessar/Oliver:2011}. We say that $\alpha$ induces an epimorphism from $\F$ to $\F'$ if $(\alpha,\alpha_{P,Q}\colon P,Q\leq S)$ is a surjective morphism of fusion systems. This means that $\alpha$ is surjective as a map $S\rightarrow S'$ and, for every $P,Q\leq S$ with $\ker(\alpha)\leq P\cap Q$ the map $\alpha_{P,Q}$ is surjective, i.e. for each $\psi\in\Hom_{\F'}(P\alpha,Q\alpha)$, there exists $\phi\in\Hom_\F(P,Q)$ with $(\alpha|_P)\psi=\phi(\alpha|_Q)$. 
If $\alpha$ is in addition injective then we say that $\alpha$ induces an isomorphism from $\F$ to $\F'$. Note that this fits with the definition we gave earlier in Section~\ref{SubcentricProperties}.

\smallskip

If $\alpha$ induces an epimorphism from $\F$ to $\F'$ then notice that the induced map
\[S/\ker(\alpha)\rightarrow S'\]
is a fusion preserving isomorphism from $\F/\ker(\alpha)$ to $\F'$.

\smallskip

In the remainder of this subsection we will summarize the theory of projections and quotients of localities, and relate this theory to the theory of morphisms and quotients of fusion systems. 

\bigskip

\textbf{From now on let $(\L,\Delta,S)$ be a locality.}

%and let $\L'$ be a partial group with product defined on the domain $\D'$. If $(\L',\Delta',S')$ forms a locality then a homomorphism $\beta\colon\L\rightarrow\L'$ of partial groups is called a projection of localities if $\Delta'=\{P\beta\colon P\in\Delta\}$ and $\D\beta^*=\D'$. 

\begin{theorem}\label{LocalityProjection}
 Let $\L'$ be a partial group with product defined on the domain $\D'$. Let $\beta\colon \L\rightarrow\L'$ be a homomorphism of partial groups such that $\D\beta^*=\D'$, where $\beta^*\colon\W(\L)\rightarrow\W(\L')$ is the map induced by $\beta$. Set $T=\ker(\beta)\cap S$, $S'=S\beta$ and $\Delta'=\{P\beta\colon P\in\Delta\}$. 
\begin{itemize}
 \item [(a)] $(\L',\Delta',S')$ is a locality.
 \item [(b)] The restriction $\beta|_S\colon S\rightarrow S'$ of $\beta$ to $S$ induces an epimorphism from $\F_S(\L)$ to $\F_{S'}(\L')$ with kernel $T$. In particular, the group isomorphism $S/T\rightarrow S',\;sT\mapsto s\beta$  induces an isomorphism from $\F_S(\L)/T$ to $\F_{S'}(\L')$.
 \item [(c)] Let $P,Q\in\Delta$ with $T\leq P\cap Q$. Then $\beta$ restricts to a surjection $N_\L(P,Q)\rightarrow N_{\L'}(P\beta,Q\beta)$, and to a surjective homomorphism of groups if $P=Q$.
\end{itemize}
\end{theorem}

\begin{proof}
 For properties (a) and (c) see \cite[Theorem~4.4]{Chermak:2013}. Recall that $\F:=\F_S(\L)$ is generated by the conjugation maps between elements of $\Delta$, and similarly $\F':=\F_{S'}(\L')$ is generated by the conjugation maps between the elements of $\Delta'$. Thus, it is sufficient to prove the following two properties for $P,Q\in\Delta$: 
\begin{itemize}
 \item [(1)] For every conjugation map $\phi\in\Hom_\F(P,Q)$ there exists $\psi\in\Hom_{\F'}(P\beta,Q\beta)$ such that $\phi\beta|_Q=\beta|_P\psi$. 
 \item [(2)] If $T\leq P\cap Q$ then, for any conjugation map $\psi\in\Hom_{\F'}(P\beta,Q\beta)$ there exists $\phi\in\Hom_\F(P,Q)$ such that $\phi\beta|_Q=\beta|_P\psi$.
\end{itemize}

\smallskip

Notice that for $f\in N_\L(P,Q)$ and $x\in P$, we have $x^f\beta=(x\beta)^{f\beta}$ 
as $\beta$ is a homomorphism of partial groups. Hence, $x (c_f|_P) \beta=x^f\beta=(x\beta)^{f\beta}=x\beta (c_{f\beta}|_{P\beta})$. This proves 
\[(*)\;\;\;(c_f|_P)\beta|_Q=\beta|_P(c_{f\beta}|_{P\beta})\mbox{ for any }f\in N_\L(P,Q).\] 
Since any conjugation homomorphism $\phi\in\Hom_\F(P,Q)$ is of the form $c_f|_P$ with $f\in N_\L(P,Q)$, and as $c_{f\beta}|_{P\beta}\in\Hom_{\F'}(P\beta,Q\beta)$, this shows (1). Assume now $T\leq P\cap Q$. Every conjugation homomorphism $\psi\in\Hom_{\F'}(P\beta,Q\beta)$ is of the form $c_g|_{P\beta}$ with $g\in N_{\L'}(P\beta,Q\beta)$. By (c), there exists $f\in N_\L(P,Q)$ with $f\beta=g$. So (2) follows also from (*) as $c_f|_P\in\Hom_\F(P,Q)$.
\end{proof}

\begin{definition}\label{ProjectionDef}
Let $(\L',\Delta',S')$ be a locality with the partial product defined on a domain $\D'$. Then a homomorphism $\beta\colon\L\rightarrow\L'$ of partial groups is called a \textit{projection} (of localities) from $(\L,\Delta,S)$ to $(\L',\Delta',S')$ if $\D\beta^*=\D'$ and $\Delta'=\{P\beta\colon P\in\Delta\}$. 
\end{definition}

If $\beta$ is a projection of localities as in the above definition then note that $S\beta=S'$ as $\Delta'=\{P\beta\colon P\in\Delta\}$. 

\bigskip

As we mentioned before, the kernels of homomorphisms of partial groups form partial normal subgroups. Conversely, given a partial normal subgroup $\N$ of $\L$, one can form a partial group $\L/\N$ such that there is a natural homomorphism from $\L$ onto $\L/\N$ with kernel $\N$. We stress however that $\L/\N$ is not defined for an arbitrary partial group $\L$, so we actually need here the assumption that $(\L,\Delta,S)$ is a locality. 

\smallskip

The ``quotient'' $\L/\N$ is more precisely defined as follows: Call a subset of $\L$ of the form $\N f:=\{\Pi(n,f)\colon n\in\N,\;(n,f)\in\D\}$ a right coset of $\N$ in $\L$. A maximal right coset of $\N$ is a right coset which is maximal with respect to inclusion among the right cosets of $\N$. By \cite[Proposition~3.14(d)]{Chermak:2015}, the maximal right cosets of $\N$ form a partition of $\L$. As a set, $\L/\N$ is the set of maximal right cosets of $\N$ in $\L$. So there is a natural map $\rho\colon\L\rightarrow \L/\N$ which sends every element $g\in\L$ to the (unique) maximal right coset of $\N$ containing $g$. Writing $\rho^*$ for the map $\W(\L)\rightarrow\W(\L/\N)$ induced by $\rho$ and setting $\ov{\D}=\D\rho^*$, $\ov{\L}=\L/\N$ forms a partial group with product $\ov{\Pi}\colon\ov{\D}\rightarrow \ov{\L}$ defined by $\ov{\Pi}(v\rho^*)=\Pi(v)\rho$ for all $v\in\D$. By construction, $\rho$ is then a homomorphism of partial groups; see \cite[Lemma~3.16]{Chermak:2015}. The identity element of $\ov{\L}$ is $\N\One=\N$. So $\ker(\rho)=\N$.  
Setting $\ov{S}=S\rho$ and $\ov{\Delta}=\{P\rho\colon P\in\Delta\}$, $(\ov{\L},\ov{\Delta},\ov{S})$ is a locality by Lemma~\ref{LocalityProjection}. We call this locality the quotient locality of $\L$ modulo $\N$, and we call $\rho\colon \L\rightarrow \L/\N$ the canonical projection.

\begin{cor}\label{LocalityQuotient}
 Let $(\L,\Delta,S)$ be a locality over $\F$ and $R\leq S$ such that $R$ forms a partial normal subgroup of $\L$. Then $R$ is strongly closed in $\F$. Furthermore, setting $\ov{\L}=\L/R$, $\ov{S}=S/R$ and $\ov{\Delta}=\{PR/R\colon P\in\Delta\}$, the triple $(\ov{\L},\ov{\Delta},\ov{S})$ is a locality over $\F/R$.
\end{cor}

\begin{proof}
As $R$ is a partial normal subgroup of $\L$, $R$ is strongly closed in $\F=\F_S(\L)$. Let $\beta\colon \L\rightarrow \L/R$ be the natural projection. Then $\ker(\beta)=R=\ker(\beta)\cap S$, $S\beta=S/R$ and the induced map $S/R\rightarrow S\beta,\;sR\mapsto s\beta$ is just the identity on $S/R$. Hence, the claim follows from Theorem~\ref{LocalityProjection}.
\end{proof}

\subsection{Transporter systems coming from localities}\label{TransporterSystemsSubsection}

If $G$ is a group and $\Delta$ a set of subgroups of $G$ then $\T_{\Delta}(G)$ denotes  the \textit{transporter category of $G$} with object set $\Delta$. That is, for $P,Q\in\Delta$, the set of morphism from $P$ to $Q$ is given by  $\Hom_{\T_{\Delta}(G)}(P,Q)=\{(g,P,Q)\colon g\in G\mbox{ with }P^g\leq Q\}$.

\smallskip

Our use of the term ``transporter system associated to a fusion system $\F$'' has been slightly sloppy so far. A transporter system is not just a category $\T$, but it comes always together with ``structural maps'', namely a pair of functors
\[\T_{\ob(\T)}(S)\stackrel{\epsilon}{\longrightarrow} \T\stackrel{\rho}{\longrightarrow} \F \]
subject to certain axioms. In particular, $\epsilon$ is the identity on objects, $\ob(\T)\subseteq\ob(\F)$, and $\rho$ is the inclusion on objects. So a transporter system should be thought of more correctly as a triple $(\T,\epsilon,\rho)$ with $\epsilon$ and $\rho$ as above. Given such a transporter system $(\T,\epsilon,\rho)$, the map $\rho_{P,P}\colon \Aut_\T(P)\rightarrow \Aut_\F(P)$ is a group homomorphism for any $P\in\ob(\T)$. Its kernel is denoted by $E(P)$. We refer the reader to 
\cite[Definition~3.1]{OV1} for the precise definition of a transporter system. Comparing this definition with the definitions of linking systems in \cite[Definition~1.7]{BLO2} and \cite[Definition~3]{O4} one observes:

\begin{rmk}\label{OliverLinkingSystem}
Let $(\T,\epsilon,\rho)$ be a transporter system associated to $\F$ and $E(P)=\ker(\rho_{P,P})$ for any $P\in\ob(\T)$. Then $(\T,\epsilon,\rho)$ is a centric linking system as defined in \cite[Definition~1.7]{BLO2} if and only if $\ob(\T)=\F^c$ and $E(P)=Z(P\epsilon_{P,P})$.  
Moreover, $(\T,\epsilon,\rho)$ is a linking system associated to $\F$ in the sense of Oliver \cite[Definition~3]{O4} if and only if $\F^{cr}\subseteq \ob(\T)$ and $E(P)$ is a $p$-group for every object $P$ of $\T$. 
\end{rmk}

Recall that a transporter system is a linking system in our sense if and only if $\F^{cr}\subseteq\ob(\T)$ and $\Aut_\T(P)$ is of characteristic $p$ for all $P\in\ob(\T)$.

\smallskip

Two transporter systems $(\T,\epsilon,\rho)$ and $(\T',\epsilon',\rho')$ associated to $\F$ are called isomorphic if there exists an isomorphism between them, i.e. an invertible functor $\alpha\colon \T\rightarrow\T'$ such that (in right handed notation) $\epsilon\circ\alpha=\epsilon'$ and $\alpha\circ \rho'=\rho$.

\begin{rmk}\label{IsomorphicTransporterSystems}
 Let $(\T,\epsilon,\rho)$ and $(\T',\epsilon',\rho')$ be transporter systems associated to $\F$ with $\ob(\T)=\ob(\T')$, and let $\alpha\colon\T\rightarrow\T'$ be an isomorphism between them which is the identity on objects. Then for any $P\in\ob(\T)$, the map $\alpha_{P,P}\colon \Aut_\T(P)\rightarrow \Aut_{\T'}(P)$ is an isomorphism of groups. In particular, $(\T,\epsilon,\rho)$ is a linking system if and only if $(\T',\epsilon',\rho')$ is a linking system. 

\smallskip

Moreover, for every $P\in\ob(\T)$, the isomorphism $\alpha_{P,P}$ maps $E(P)=\ker(\rho_{P,P})$ to $E'(P)=\ker(\rho'_{P,P})$ and $P\epsilon_{P,P}$ to $P\epsilon'_{P,P}$. So $(\T,\epsilon,\rho)$ is a centric linking system if and only if $(\T',\epsilon',\rho')$ is a centric linking system. Similarly $(\T,\epsilon,\rho)$ is a linking locality in the sense of Oliver \cite[Definition~3]{O4} if and only if the same holds for $(\T',\epsilon',\rho')$. 
\end{rmk}

\begin{proof}
 The first part is clear. Let $P\in\Delta$ and $g\in\Aut_\T(P)$. As $\alpha\circ\rho'=\rho$, we have $g\in E(P)$ if and only if $g\alpha_{P,P}\rho'_{P,P}=g\rho_{P,P}=\id_P$, i.e. if and only if $g\alpha_{P,P}\in E'(P)$. Hence, $E(P)\alpha_{P,P}=E'(P)$. As $\epsilon\circ\alpha=\epsilon'$, we have $(P\epsilon_{P,P})\alpha_{P,P}=P\epsilon'_{P,P}$. The last part follows now from Remark~\ref{OliverLinkingSystem}. 
\end{proof}

Suppose now we are given a locality $(\L,\Delta,S)$. Then we can construct a transporter system associated to $\F=\F_S(\L)$ as follows: The objects of $\T(\L,\Delta)$ are the elements of $\Delta$, and a morphism between objects $P,Q\in\Delta$ is a triple $(f,P,Q)$ with $f\in\L$ such that $P\subseteq\D(f)$ and $P^f\leq Q$. Composition of morphisms is given by multiplication in the locality $\L$, i.e. $(f,P,Q)\circ (g,Q,R)=(fg,P,R)$ for all morphisms $(f,P,Q)$ and $(g,Q,R)$ in $\T(\L,\Delta)$. 

\smallskip

Note that $\Hom_{\T_\Delta(S)}(P,Q)\subseteq\Hom_{\T(\L,\Delta)}(P,Q)$ for all $P,Q\in\Delta$. Let $\epsilon=\epsilon_{\L,\Delta}\colon \T_\Delta(S)\rightarrow\T(\L,\Delta)$ be the functor which is the identity on objects and the inclusion map on morphism sets. Let $\rho=\rho_{\L,\Delta}\colon \T(\L,\Delta)\rightarrow \F$ be the functor which is the inclusion on objects, and for $P,Q\in\Delta$, $\rho_{P,Q}\colon \Hom_{\T(\L,\Delta)}\rightarrow \Hom_\F(P,Q)$ is defined by $(f,P,Q)\mapsto c_f|_P$. 

\begin{theorem}\label{TLDelta}
 Let $(\L,\Delta,S)$ be a locality over $\F$. Let $\epsilon=\epsilon_{\L,\Delta}$ and $\rho=\rho_{\L,\Delta}$.
\begin{itemize}
\item [(a)] The triple $(\T(\L,\Delta),\epsilon,\rho)$ forms a transporter system associated to $\F$. 
\item [(b)] For any $P\in\Delta$, we have $\Aut_{\T(\L,\Delta)}(P)\cong N_\L(P)$ and $E(P):=\ker(\rho_{P,P})=\{(f,P,P)\colon f\in C_\L(P)\}\cong C_\L(P)$. 
\item [(c)] The locality $(\L,\Delta,S)$ is a linking locality if and only if $(\T(\L,\Delta),\epsilon,\rho)$ is a linking system.
\item [(d)] The transporter system $(\T(\L,\Delta),\epsilon,\rho)$ is a linking system in the sense of Oliver \cite[Definition~3]{O4} if and only if $\F^{cr}\subseteq\Delta$ and $C_\L(P)$ is a $p$-group for every $P\in\Delta$. Moreover, $(\T(\L,\Delta),\epsilon,\rho)$ is a centric linking system if and only if $\Delta=\F^c$ and $C_\L(P)\leq P$ for every $P\in\Delta$.
\end{itemize}
\end{theorem}

\begin{proof}
Property (a) is shown in \cite[Proposition~A.3(a)]{Chermak:2013}. Clearly, for any $P\in\Delta$, the map $N_\L(P)\rightarrow \Aut_{\T(\L,\Delta)}(P)$ with $f\mapsto (f,P,P)$ is an isomorphism of groups. Moreover, any element  $(f,P,P)\in \Aut_{\T(\L,\Delta)}(P)$ lies in $E(P)$ if and only if $c_f|_P=\id_P$, i.e. if and only if $f\in C_\L(P)$. This shows (b). Property (c) follows now from (b), and (d) follows from (b) and Remark~\ref{OliverLinkingSystem}.
\end{proof}

\begin{theorem}\label{GetEveryTransporterSystem}
Let $(\T,\epsilon,\rho)$ be a transporter system associated to $\F$. Then there exists a locality $(\L,\Delta,S)$ over $\F$ with $\Delta=\ob(\T)$ and an isomorphism $\eta\colon \T\rightarrow \T(\L,\Delta)$ between $(\T,\epsilon,\rho)$ and  $(\T(\L,\Delta),\epsilon_{\L,\Delta},\rho_{\L,\Delta})$ which is the identity on objects.
\end{theorem}

\begin{proof}
Chermak \cite[Appendix~A]{Chermak:2013} constructs a locality $(\L,\Delta,S)$ with $\Delta=\ob(\T)$; see in particular \cite[Proposition~A.13]{Chermak:2013}. It is then shown in Lemma~A.14 and Lemma~A.15 of \cite{Chermak:2013} that there exists an invertible functor $\eta\colon\T\rightarrow \T(\L,\Delta)$ with certain properties. These properties imply that $\F=\F_S(\L)$ and $\eta$ is an isomorphism of transporter systems. The argument is exactly the same as the argument in the proof of Theorem~A in \cite{Chermak:2013} that the two left hand squares in the diagram on p.137 commute.
\end{proof}

\section{Localities of objective characteristic $p$}\label{LocalitiesSection1}

In this section we prove Proposition~\ref{ex0} and Proposition~\ref{NormLF}. We also prove some results that will be used in the next section to show Theorem~\ref{MainThm1}. Moreover, with Proposition~\ref{GetLocalityObjectiveCharp} we  give a method to produce (under certain circumstances) localities of objective characteristic $p$ by ``factoring our $p^\prime$-elements''. This will be used in the last section to demonstrate how linking localities can be constructed from a finite group. 

\smallskip

\noindent\textbf{In this section $\F$ is not necessarily assumed to be saturated.}

\begin{lemma}\label{NormModel}
 Let $(\L,\Delta,S)$ be a locality over $\F$ of objective characteristic $p$. If $P\in\Delta\cap\F^f$ then $N_\L(P)$ is a model for $N_\F(P)$. In particular, $\Delta\subseteq \F^s$.
\end{lemma}

\begin{proof}
Let $P\in\Delta\cap\F^f$. By Lemma~\ref{LocalitiesProp}(a) and Lemma~\ref{LocalitiesPropG}, $N_\L(P)$ is a subgroup of $\L$, $N_S(P)\in\Syl_p(N_\L(P))$ and $N_\F(P)=\F_{N_S(P)}(N_\L(P))$. As $N_\L(P)$ is of characteristic $p$, it follows that $N_\L(P)$ is a model for $N_\F(P)$. It remains to show that $\Delta\subseteq\F^s$. If $\F$ were assumed to be saturated, this would follow immediately from Lemma~\ref{Model1}(a) and Lemma~\ref{subcentricEquiv}. However, since we do not assume that $\F$ is saturated, we need to argue more carefully using the definition of subcentric subgroups. 

\smallskip

Set $R:=O_p(N_\F(P))$. As a next step, we show that $R$ is centric in $\F$. Since $N_\L(P)$ is a model for $N_\F(P)$, Lemma~\ref{Model1}(b) gives $R=O_p(N_\F(P))=O_p(N_\L(P))$. By Lemma~\ref{LocalitiesPropG}, there exists $h\in\L$ such that $N_S(R)\leq S_h$, $R^h\in\F^f$ and $N_S(R^h)\in\Syl_p(N_\L(R^h))$.  Note that $N_S(P)\leq N_S(R)\leq S_h$. By Lemma~\ref{LocalitiesProp}(b), $P^h\in\Delta$, $N_\L(P)\subseteq\D(h)$ and $c_h\colon N_\L(P)\rightarrow N_\L(P^h)$ is an isomorphism of groups. In particular, $R^h=O_p(N_\L(P^h))$. Moreover, $N_S(P)\cong N_S(P)^h\leq N_S(P^h)$. As $P$ is fully normalized, it follows $N_S(P)^h=N_S(P^h)$ and $P^h$ is fully normalized. Since $P$ was arbitrary, everything we proved above for $P$ holds also for $P^h$. So $O_p(N_\F(P^h))=O_p(N_\L(P^h))=R^h$ and $N_\L(P^h)$ is a model for $N_\F(P^h)$. So, as $P^h\leq R^h$, we have  $C_S(R^h)=C_{N_S(P^h)}(R^h)\leq R^h$. This implies that $R$ is centric, because $R^h$ is fully normalized. 

\smallskip

As $P$ was arbitrary, this shows that, for every $Q\in\Delta$ and every fully $\F$-normalized $\F$-conjugate $P$ of $Q$, $O_p(N_\F(P))$ is $\F$-centric. So every $Q\in\Delta$ is subcentric by definition. This shows $\Delta\subseteq\F^s$ and completes the proof.  
\end{proof}

If $(\L,\Delta,S)$ is a locality, define $P\in\Delta$ to be \textit{$\L$-radical} if $O_p(N_\L(P))=P$.

\begin{lemma}\label{RadicalLF}
 Let $(\L,\Delta,S)$ be a locality over $\F$ of objective characteristic $p$ and $P\in\Delta$. Then $P$ is $\L$-radical if and only if $P\in\F^{cr}$.
\end{lemma}

\begin{proof}
It follows from Lemma~\ref{LocalitiesProp}(b) that the set of $\L$-radical subgroups is closed under $\F$-conjugation. The set $\F^{cr}$ is closed under $\F$-conjugation as well. Hence, we may assume that $P\in\F^f$. Then by Lemma~\ref{NormModel}, $G:=N_\L(P)$ is a model for $N_\F(P)$. Note $G/C_G(P)\cong \Aut_\F(P)$ and $P/Z(P)\cong \Inn(P)$. Hence, if $C_G(P)=Z(P)$ then $O_p(\Aut_\F(P))=\Inn(P)$ if and only if $P=O_p(G)$. 

\smallskip

If $P\in\F^{cr}$ then $P\in N_\F(P)^c$ and so by Theorem~\ref{Model1}, $C_G(P)=Z(P)$. Hence, by what we just stated, $P=O_p(G)$ and $P$ is $\L$-radical. 

\smallskip

Conversely, assuming that $P$ is $\L$-radical, $C_G(P)=Z(P)$ as $G$ is of characteristic $p$. So, again by what we stated before, $P\in\F^r$. Moreover, $C_S(P)=C_{N_S(P)}(P)\leq C_G(P)\leq P$. So $P\in\F^c$ as $P\in\F^f$. This proves the assertion.
\end{proof}

\begin{proof}[Proof of Proposition~\ref{NormLF}]
Clearly, $Q\unlhd \F$ if $\L=N_\L(Q)$ and $Q\leq Z(\F)$ if $\L=C_\L(Q)$. Moreover, if $Q\leq Z(\F)$ and $\L=N_\L(Q)$ then clearly, $\L=C_\L(Q)$. Hence, it is sufficient to prove that $\L=N_\L(Q)$ if $Q\unlhd \F$. So assume $Q\unlhd \F$ and $\L\neq N_\L(Q)$. Choose $f\in \L\backslash N_\L(Q)$ such that $|S_f|$ is maximal. 

\smallskip

Since $Q\unlhd \F$ it follows $Q\not\leq S_f$. In particular, $S_f<S$ and thus $S_f^f<N_S(S_f^f)$. By Lemma~\ref{LocalitiesPropG}, there exists $h\in\L$ such that $N_S(S_f^f)\leq S_h$ and $N_S(S_f^{fh})\in\Syl_p(N_\L(S_f^{fh}))$. Then $(f,h,h^{-1})\in\D$ via $S_f$. By the maximality of $|S_f|$, $h\in N_\L(Q)$. So if $fh\in N_\L(Q)$ then $f=(fh)h^{-1}\in N_\L(Q)$ as $N_\L(Q)$ is a partial subgroup of $\L$ by Lemma~\ref{NormalizerCentralizerPartialSubgroup}. Hence, $fh\not\in N_\L(Q)$ and by the maximality of $|S_f|$, $S_f=S_{fh}$. So replacing $f$ by $fh$, we may assume that $N_S(S_f^f)\in \Syl_p(N_\L(S_f^f))$. 

\smallskip

As $c_f:S_f\rightarrow S_f^f$ is a morphism in $\F$ and $Q\unlhd \F$, there exists $g\in\L$ such that $S_fQ\leq S_g$, $c_g|_{S_f}=c_f$ and $g\in N_\L(Q)$. Then $(f^{-1},g)\in\D$ via $S_f^f$ and $f^{-1}g\in C_\L(S_f^f)\subseteq N_\L(S_f^f)$. Since $Q\unlhd\F$, we have $\Aut_{QP}(P)\leq O_p(\Aut_\F(P))$ for any $P\leq S$. Hence, $Q\leq P$ for every $P\in\F^{cr}$. In particular, $S_f\not\in\F^{cr}$ as $Q\not\leq S_f$. So $S_f^f\not\in\F^{cr}$ and thus, by Lemma~\ref{RadicalLF}, $S_f^f<R:=O_p(N_\L(S_f^f))$. As $N_S(S_f^f)\in\Syl_p(N_\L(S_f^f))$ we have $R\leq S$. As $f^{-1}g\in N_\L(S_f^f)$ and $R$ is normal in $N_\L(S_f^f)$, it follows $R\leq S_{f^{-1}g}$. So the maximality of $|S_f|=|S_f^f|$ yields $f^{-1}g\in N_\L(Q)$. As $(f^{-1},g,g^{-1})\in\D$ via $S_f^f$, and since $N_\L(Q)$ is a partial subgroup of $\L$, it follows $f^{-1}=(f^{-1}g)g^{-1}\in N_\L(Q)$. This yields a contradiction to $f\not\in N_\L(Q)$.
\end{proof}

\begin{lemma}\label{ex}
Let $(\L,\Delta,S)$ be a locality over $\F$.
\begin{itemize}
 \item [(a)] If $P\in\Delta\cap \F^{fc}$, then the following conditions are equivalent:
\begin{itemize}
\item[(a1)] $N_\L(P)$ is of characteristic $p$ (and thus a model for $N_\F(P)$).
\item[(a2)] $C_\L(Q)\leq Q$ for some $Q\in P^\F$.
\item[(a3)] $C_\L(Q)\leq Q$ for all $Q\in P^\F$.
\end{itemize}
 \item [(b)] If $P\in\Delta\cap \F^{fq}$, then $C_\L(P)=C_S(P)O_{p^\prime}(C_\L(P))$ and the following conditions are equivalent:
\begin{itemize} 
\item[(b1)] $N_\L(P)$ is of characteristic $p$ (and thus a model for $N_\F(P)$).
\item[(b2)] $C_\L(Q)$ is a $p$-group for some $Q\in P^\F$.
\item[(b3)] $C_\L(Q)$ is a $p$-group for all $Q\in P^\F$.
\end{itemize}
\end{itemize}
\end{lemma}

\begin{proof}
 Let $P\in\Delta\cap \F^f$. Then by Lemma~\ref{LocalitiesProp}(a) and Lemma~\ref{LocalitiesPropG}, $G:=N_\L(P)$ is a finite group with $N_S(P)\in\Syl_p(G)$, $N_\F(P)=\F_{N_S(P)}(G)$ and $C_\F(P)=\F_{C_S(P)}(C_G(P))$. In particular, $G$ is a model for $N_\F(P)$ if and only if $G$ is of characteristic $p$.

\smallskip

By Lemma~\ref{MorphismAsConjugation}, every $\F$-morphism between $P$ and an $\F$-conjugate $Q$ of $P$ can be realized as a conjugation map by an element of $f\in\L$ and then, by Lemma~\ref{LocalitiesProp}(b), $c_f\colon N_\L(P)\rightarrow N_\L(Q)$ is an isomorphism of groups. In particular, $C_\L(P)\cong C_\L(Q)$ for any $Q\in P^\F$. %Hence, $C_\L(P)\leq P$ if and only if $C_\L(Q)\leq Q$ for any $Q\in P^\F$, and $C_\L(P)$ is a $p$-subgroup if and only if $C_\L(Q)$ is a $p$-subgroup for any $Q\in P^\F$. 

\smallskip

For the proof of (a) suppose now that $P$ is $\F$-centric. Then $P$ is also centric in $N_\F(P)$. As $P\unlhd G$, $G$ is  of characteristic $p$ if $C_G(P)=C_\L(P)\leq P$. Conversely, if $G$ is a model for $N_\F(P)$, then Theorem~\ref{Model1}(b) yields that $C_\L(P)=C_G(P)\leq P$. So $G$ is of characteristic $p$ and thus a model for $N_\F(P)$ if and only if $C_\L(P)\leq P$. As $C_\L(P)\cong C_\L(Q)$ for every $Q\in P^\F$, we have $C_\L(P)\leq P$ if and only if (a2) holds, and this is the case if and only if (a3) holds. This proves (a).

\smallskip

To prove (b), assume now that $P$ is quasicentric, i.e. $\F_{C_S(P)}(C_S(P))=C_\F(P)=\F_{C_S(P)}(C_G(P))$. So by Lemma~\ref{Charp2}, $C_\L(P)=C_G(P)=C_S(P)O_{p^\prime}(C_\L(P))$ and $G$ is a model for $N_\F(P)$ if and only if $C_\L(P)=C_G(P)$ is a $p$-group. As $C_\L(P)\cong C_\L(Q)$ for every $Q\in P^\F$, this yields (b). 
\end{proof}

\begin{proof}[Proof of Proposition~\ref{ex0}]
Property (a) is stated in its more precise form in Proposition~\ref{TLDelta}(c). The statement in (b) about $(\L,\Delta,S)$  is proved in Lemma~\ref{NormModel}. The statements in (c) and (d) about $(\L,\Delta,S)$ follow from Lemma~\ref{ex}. 

\smallskip

We argue now that the statements in (b),(c),(d) about $(\L,\Delta,S)$ imply the statements about $\T$, where $(\T,\epsilon,\rho)$ is a transporter system associated to $\F$: By Remark~\ref{IsomorphicTransporterSystems}, we can replace $(\T,\Delta,S)$ by any transporter system which is isomorphic via an isomorphism which is the identity on objects. So by Lemma~\ref{GetEveryTransporterSystem}, we can assume $(\T,\epsilon,\rho)=(\T(\L,\Delta),\epsilon_{\L,\Delta},\rho_{\L,\Delta})$. However, then the statements about $\T$ follow from Theorem~\ref{TLDelta}(b),(c),(d). So it remains only to prove the statement in (c) that every linking system in Oliver's definition is a linking system in our definition. This follows however from what we have shown and the following fact: If $(\T,\epsilon,\rho)$ is a linking system associated to $\F$ in the sense of Oliver \cite[Definition~3]{O4}, then the objects of $\T$ are by \cite[Proposition~4(g)]{O4}  quasicentric in $\F$.
\end{proof}

We close this section by giving a method to produce localities of objective characteristic $p$ in certain circumstances. A weaker version of the result we state was proved by Chermak in an earlier draft of \cite{Chermak:2015}; in the meantime Chermak adapted his text to state our version (cf. \cite[Theorem~4.13]{Chermak:2015}). Recall from Definition~\ref{ThetaDef} that a finite group $G$ is almost of characteristic $p$ if $G/O_{p^\prime}(G)$ is of characteristic $p$.

\begin{prop}\label{GetLocalityObjectiveCharp}
 Let $(\L,\Delta,S)$ be a locality such that $N_\L(P)$ is almost of characteristic $p$ for every $P\in\Delta$. Set $\Theta(P):=O_{p^\prime}(N_\L(P))$ for every $P\in\Delta$, and $\Theta:=\bigcup\{\Theta(P)\colon P\in\Delta\}$. 

\smallskip

Then $\Theta$ is a partial normal subgroup of $\L$ with $\Theta\cap S=1$. The canonical projection $\rho\colon \L\rightarrow \L/\Theta$ restricts to an isomorphism $S\rightarrow S\rho$. Upon identifying $S$ with $S\rho$, the following properties hold:
\begin{itemize}
 \item [(a)] $(\L/\Theta,\Delta,S)$ is a locality of objective characteristic $p$.
 \item [(b)] $\F_S(\L/\Theta)=\F_S(\L)$.
 \item [(c)] For every $P\in\Delta$, the restriction $N_\L(P)\rightarrow N_{\L/\Theta}(P)$ of $\rho$ has kernel $\Theta(P)$ and induces an isomorphism $N_\L(P)/\Theta(P)\cong N_{\L/\Theta}(P)$.
\end{itemize}
\end{prop}

\begin{proof}
\noindent{\em Step 1:} We show that $\Theta(Q)=\Theta(P)\cap C_\L(Q)$ for any $P,Q\in\Delta$ with $P\leq Q$. For the proof note that $P$ is subnormal in $Q$, and by induction on the subnormal length, we may assume that $P\unlhd  Q$. Then $Q\leq N_\L(P)$ and $C_\L(Q)=C_{N_\L(P)}(Q)$. Hence, by Lemma~\ref{AlmostCharp1}(a), $\Theta(Q)=O_{p^\prime}(C_\L(Q))=O_{p^\prime}(C_{N_\L(P)}(Q))=O_{p^\prime}(N_\L(P))\cap C_{N_\L(P)}(Q)=\Theta(P)\cap C_\L(Q)$. This completes Step~1.

\smallskip

\noindent{\em Step 2:} We show $x\in \Theta(S_x)$ for any $x\in\Theta$. Let $x\in\Theta$. Then by definition of $\Theta$, the element $x$ lies in $\Theta(P)$ for some $P\in\Delta$. Choose such $P$ maximal with respect to inclusion. We have $P\leq S_x$ and $[N_{S_x}(P),x]\leq \Theta(P)\cap N_S(P)=1$. Hence, using Step~1, $x\in \Theta(P)\cap C_\L(N_{S_x}(P))=\Theta(N_{S_x}(P))$. So the maximality of $P$ yields $P=N_{S_x}(P)$ and thus $P=S_x$. Hence, $x\in\Theta(S_x)$ as required.

\smallskip

\noindent{\em Step 3:} We show that $\Theta$ is a partial normal subgroup of $\L$. Note that $1\in\Theta$ as $1\in\Theta(P)$ for any $P\in\Delta$. Moreover, clearly $\Theta$ is closed under inversion, since $\Theta(P)$ is a group for any $P\in\Delta$. Let now $(x_1,\dots,x_n)\in\D$ with $x_i\in\Theta$ for $i=1.\dots,n$. Then $R:=S_{(x_1,\dots,x_n)}\in\Delta$ by Lemma~\ref{LocalitiesProp}(f). Note that $R\leq S_{x_1}$ and thus $x_1\in \Theta(S_{x_1})\leq\Theta(R)\leq C_\L(R)$ by Step~1 and Step~2. In particular, $R=R^{x_1}\leq S_{x_2}$. Proceeding similarly using Step~1 and Step~2, one shows by induction on $i$ that $R\leq S_{x_i}$ and $x_i\in\Theta(R)\leq C_\L(R)$ for every $i=1,\dots,n$. Hence, $\Pi(x_1,x_2,\dots ,x_n)\in\Theta(R)\subseteq\Theta$. Thus, $\Theta$ is a partial subgroup of $\L$. Let  $x\in\Theta$ and $f\in\L$ with $(f^{-1},x,f)\in\D$. By Lemma~\ref{LocalitiesProp}(f), $X:=S_{(f^{-1},x,f)}\in\Delta$. Moreover, $X^{f^{-1}}\leq S_x$. By Step~2, we have $x\in\Theta(S_x)$, and then by Step~1, $x\in\Theta(X^{f^{-1}})$. It follows now from Lemma~\ref{LocalitiesProp}(b) that $x^f\in\Theta(X^{f^{-1}})^f=\Theta(X)\subseteq\Theta$. Hence, $\Theta$ is a partial normal subgroup of $\L$.

\noindent{\em Step 4:} We are now in a position to complete the proof. Notice first that $\Theta\cap S=1$ as $\Theta(P)\cap S=\Theta(P)\cap N_S(P)=1$ for every $P\in\Delta$. The quotient map $\rho\colon \L\rightarrow \L/\Theta$ is a homomorphism of partial groups with $\ker(\rho)=\Theta$; see Section~\ref{LocalityProjectionSection}. Therefore, $\rho|_S\colon S\rightarrow S\rho$ is a homomorphism of groups with kernel $S\cap\Theta=1$ and thus an isomorphism of groups. Upon identifying $S$ with $S\rho$, it follows now from Theorem~\ref{LocalityProjection}(a),(b) that $(\L/\Theta,\Delta,S)$ is a locality and $\F_S(\L)=\F_S(\L/\Theta)$. So (b) holds. Let $P\in\Delta$. By Theorem~\ref{LocalityProjection}(c), the restriction of $\rho$ to a map $N_\L(P)\rightarrow N_{\L/\Theta}(P)$ is an epimorphism with kernel $N_\L(P)\cap\Theta$. For any $x\in N_\L(P)\cap\Theta$, we have $P\leq S_x$ and then $x\in\Theta(S_x)\leq\Theta(P)$ by Step~1 and Step~2. This shows $N_\L(P)\cap\Theta=\Theta(P)$ and so (c) holds. In particular, our assumption yields that $N_{\L/\Theta}(P)$ is a group of characteristic $p$ and therefore (a) holds.
\end{proof}

\section{Construction of linking localities}\label{Construction}

In this section we prove Theorem~\ref{MainThm1}. Building on the existence and uniqueness of centric linking systems, the key to the proof is Theorem~\ref{A1General}, which gives a way of expanding a linking locality to another linking locality with a larger object set.

\begin{lemma}\label{Norm2}
Suppose $(\L,\Delta,S)$ is a locality of objective characteristic $p$ over $\F$. Let $T\in\F^f$ be such that every proper overgroup of $T$ is in $\Delta$ and $O_p(N_\F(T))\in\Delta$. Then $N_\L(T)$ is a subgroup of $\L$ which is a model for $N_\F(T)$. 
\end{lemma}

\begin{proof}
As every proper overgroup of $T$ is in $\Delta$, 
\[\Delta_T:=\{N_P(T)\colon T\leq P\in\Delta\}=\{P\in\Delta\colon T\leq P\leq N_S(T)\}\subseteq \Delta.\]
Set
\[R:=O_p(N_\F(T)).\]
 
\noindent{\em Step~1:} We show that $(N_\L(T),\Delta_T,N_S(T))$ is a linking locality of objective characteristic $p$.  First of all, by \cite[Lemma~2.19(c)]{Chermak:2013}, $(N_\L(T),\Delta_T,N_S(T))$ is a locality. If $P\in\Delta_T$ then $N_\L(P)$ is a group of characteristic $p$, as $P\in\Delta$ and $(\L,\Delta,S)$ is of objective characteristic $p$. In particular, $N_{N_{\L}(T)}(P)=N_{N_{\L}(P)}(T)$ is a group of characteristic $p$ by Lemma~\ref{Charp1}(b). Hence, $(N_\L(T),\Delta_T,N_S(T))$ is of objective characteristic $p$. 

\smallskip

\noindent{\em Step~2:} We show that $N_\F(T)=\F_{N_S(T)}(N_\L(T))$. Clearly, $\F_{N_S(T)}(N_\L(T))\subseteq N_\F(T)$. Let now $A,B\leq N_S(T)$ and $\phi\in\Hom_{N_\F(T)}(A,B)$. As $R$ is normal in $N_\F(T)$, $\phi$ extends to $\hat{\phi}\in\Hom_{N_\F(T)}(AR,BR)$ with $R\hat{\phi}=R$. By assumption, $R\in\Delta$ and thus $AR\in\Delta$. So by Lemma~\ref{MorphismAsConjugation}, there exists $f\in\L$ with $RA\leq S_f$ and $\hat{\phi}=c_f|_{AR}$. As $\hat{\phi}$ is a morphism in $N_\F(T)$ and $T\leq R\leq RA\leq S_f$, it follows $T^f=T\hat{\phi}=T$ and thus $f\in N_\L(T)$. Hence, $\phi=c_f|_A$ is a morphism in $\F_{N_S(T)}(N_\L(T))$. This completes Step~2.

\smallskip

\noindent{\em Step 3:} We complete the proof. By Step~1 and Step~2, $(N_\L(T),N_S(T),\Delta_T)$ is a locality of objective characteristic $p$ over $N_\F(T)$. Hence, by Proposition~\ref{NormLF}, we have $N_\L(T)=N_{N_\L(T)}(R)$. As $R\in\Delta$ by assumption, $R\in\Delta_T$. Moreover, as $R$ is normal in $N_\F(T)$, $R$ is fully normalized in $N_\F(T)$. So it follows from Lemma~\ref{NormModel} applied with $N_\L(T)$ and $R$ in place of $\L$ and $P$ that $N_\L(T)=N_{N_\L(T)}(R)$ is a model for $N_\F(T)=N_{N_\F(T)}(R)$. In particular, $N_\L(T)$ is a subgroup of $\L$.
\end{proof}

Suppose $(\L^+,\Delta^+,S)$ is a locality with partial product $\Pi^+\colon\D^+\rightarrow\L^+$. Suppose $\Delta$ is a non-empty subset of $\Delta^+$ which is closed  under taking $\L^+$-conjugates and overgroups in $S$. Set 
\[\L^+|_\Delta:=\{f\in\L^+\colon \exists P\in\Delta\mbox{ such that }P\subseteq \D^+(f)\mbox{ and }P^f\leq S\}\]
and write $\D$ for the set of words $w=(f_1,\dots,f_n)$ such that $w\in\D^+$ via $P_0,\dots,P_n$ for some $P_0,\dots,P_n\in\Delta$. Observe that $\D\subseteq\W(\L^+|_\Delta)$. It is easy to check that $\L^+|_\Delta$ forms a partial group with partial product $\Pi^+|_{\D}\colon\D\rightarrow \L^+|_{\Delta}$, and that $(\L^+|_\Delta,\Delta,S)$ forms a locality. We call $\L^+|_\Delta$ the \textit{restriction} of $\L^+$ to $\Delta$.

\smallskip

Note that $\T(\L^+|_\Delta,\Delta)$ is the full subcategory of $\T(\L^+,\Delta^+)$ with object set $\Delta$.

\begin{theorem}\label{A1General}
Suppose $\F$ is saturated. Let $\Delta$ and $\Delta^+$ be collections of subgroups of $S$  which are both closed under $\F$-conjugation and with respect to overgroups. Suppose that $\F^{cr}\subseteq \Delta\subseteq \Delta^+\subseteq \F^s$, and let $(\L,\Delta,S)$ be a linking locality over $\F$.
\begin{itemize}
\item [(a)] There exists a linking locality $(\L^+,\Delta^+,S)$ such that $\L$ is the restriction $\L^+|_\Delta$ of $\L^+$ to $\Delta$ and $\F_S(\L^+)=\F$. The inclusion of nerves $|\T(\L,\Delta)|\subseteq |\T(\L^+,\Delta^+)|$ is a homotopy equivalence.
\item [(b)] If $(\wL^+,\Delta^+,S)$ is another linking locality over $\F$ with object set $\Delta^+$ and $\beta:\L\rightarrow \wL^+|_\Delta$ is a rigid isomorphism, then $\beta$ extends to a rigid isomorphism $\L^+\rightarrow \wL^+$. So in particular, $\L^+$ is unique up to an isomorphism that restricts to the identity on $\L$.
\item [(c)] If $\Delta^+\backslash\Delta$ is a single $\F$-conjugacy class then $N_\L(R)=N_{\L^+}(R)$ for every $R\in\Delta^+\backslash\Delta$ which is fully $\F$-normalized.
\end{itemize}
\end{theorem}

\begin{proof}
We may assume $\Delta\neq\Delta^+$. Choose $T\in\Delta^+\backslash \Delta$ such that $T$ is maximal with respect to inclusion. Since $\Delta^+$ is closed under taking overgroups, it follows that every proper overgroup of $T$ is in $\Delta$. Therefore, as $\Delta$ is closed under $\F$-conjugation, every proper overgroup of an $\F$-conjugate of $T$ is in $\Delta$. Hence, $\Delta\cup T^\F$ is closed under taking overgroups. By construction, this set is closed under taking $\F$-conjugates. Furthermore, $\Delta\cup T^\F\subseteq \Delta^+$, as $\Delta^+$ is closed under taking $\F$-conjugates. Now by induction on $|\Delta^+\backslash \Delta|$, we may assume $\Delta^+=\Delta\cup T^\F$. Replacing $T$ by a suitable $\F$-conjugate, we may assume $T\in\F^f$. 

\smallskip

As $\F^{cr}\subseteq\Delta$ and $T\not\in\Delta$, we have $T\not\in\F^{cr}$. Then by Lemma~\ref{fsfrc}, $T<O_p(N_\F(T))$ and thus $O_p(N_\F(T))\in\Delta$, as every proper overgroup of $T$ is in $\Delta$. Hence, by Lemma~\ref{Norm2}, $M:=N_\L(T)$ is a subgroup of $\L$ which is a model for $N_\F(T)$. Now clearly properties (1)-(4) of \cite[Hypothesis~5.3]{Chermak:2013} hold. By Theorem~\ref{Model1}(b), $O_p(M)=O_p(N_\F(T))\in\Delta$. So setting $\Delta_T:=\{P\in\Delta\colon T\unlhd P\}$, the locality $\L_{\Delta_T}(M)$ introduced in \cite[Example/Lemma~2.10]{Chermak:2013} is just the group $M$ and $\lambda=\id_M$ can be considered as a rigid isomorphism $N_\L(T)\rightarrow \L_{\Delta_T}(M)$. So Hypothesis~5.3 in \cite{Chermak:2013} is fulfilled. Thus, by \cite[Theorem~5.14]{Chermak:2013}, there exists a locality $(\L^+,\Delta^+,S)$ such that $\L$ is the restriction $\L^+|_\Delta$ of $\L^+$ to $\Delta$ and $\F_S(\L^+)=\F$. Furthermore, $\L^+$ can be taken to be the locality $\L^+(\lambda)$ constructed in \cite[Section~5]{Chermak:2013}. So the first part of (a) holds. 

\smallskip

Since our choice of $T$ was arbitrary, the arguments above give that $N_\L(R)$ is a model for $N_\F(R)$ for any $R\in T^\F\cap\F^f$ and thus $N_{\L^+}(T)=N_\L(R)$. This proves (c). 

\smallskip

To prove (b) let $(\wL^+,\Delta^+,S)$ be another linking locality over $\F$ with object set $\Delta^+$ and let $\beta:\L\rightarrow \wL^+|_\Delta$ be a rigid isomorphism. Then $\wL:=\wL^+|_\Delta$ is a linking locality as well and has thus the same properties we proved above for $\L$. In particular, $N_{\wL}(T)$ is a subgroup of $\wL$ which is a model for $N_\F(T)$. Then $\beta_T=\beta|_M\colon M\rightarrow N_{\wL}(T)$ will be an isomorphism of groups which restricts to the identity on $N_S(T)$, as $\beta$ is a rigid isomorphism. As $(\wL^+,\Delta^+,S)$ is a linking locality and $T\in\Delta^+\cap\F^f$, $N_{\wL^+}(T)$ is a model for $N_\F(T)$ by Lemma~\ref{NormModel}. Clearly, $N_{\wL}(T)\subseteq N_{\wL^+}(T)$ and thus $N_{\wL}(T)=N_{\wL^+}(T)$ by Theorem~\ref{Model1}(a). Hence, $\beta_T$ is also a group isomorphism $M\rightarrow N_{\wL^+}(T)$ which restrict to the identity on $N_S(T)$. So by \cite[Theorem~5.15(a)]{Chermak:2013} applied with $\wL^+$ in place of $\L^*$ and $\beta_T$ in place of $\mu$, there exists a rigid isomorphism $\beta^+\colon \L^+\rightarrow\wL^+$ which restricts to the identity on $\L$. This proves (b). 

\smallskip

It remains to prove the statement in part (a) about the nerves of the transporter systems. Note that $\T(\L,\Delta)$ is the full subcategory of $\T(\L^+,\Delta^+)$ with object set $\Delta$. If $(\T,\epsilon,\rho)$ is a transporter system then $P\in\ob(\T)$ is $\T$-radical in the sense defined in \cite[p.~1015]{OV1} if $O_p(\Aut_\T(P))=P\epsilon_{P,P}$. As $\Aut_{\T(\L^+,\Delta^+)}(P)\cong N_{\L^+}(P)$ for every $P\in\Delta^+$, it follows that $P\in\Delta^+$ is $\T(\L^+,\Delta^+)$-radical if and only if $P$ is $\L^+$-radical in the sense defined above. Hence, by Lemma~\ref{RadicalLF}, the $\T(\L^+,\Delta)$-radical elements of $\Delta^+$ are precisely the elements of $\F^{cr}$. As by assumption $\F^{cr}\subseteq \Delta$, it follows $\T(\L^+,\Delta^+)^r\subseteq \T(\L,\Delta)$, where $\T(\L^+,\Delta^+)^r$ denotes the full subcategory of $\T(\L^+,\Delta^+)$ with object set the $\T(\L^+,\Delta^+)$-radical subgroups. Hence, by \cite[Proposition~4.7]{OV1}, the inclusion of nerves $|\T(\L,\Delta)|\subseteq |\T(\L^+,\Delta^+)|$ is a homotopy equivalence.
\end{proof}

Theorem~\ref{MainThm1} is now easy to deduce from Theorem~\ref{A1General} using the existence and uniqueness of centric linking systems which we state here in the formulation in which we will apply it:

\begin{theorem}[Chermak, Oliver, Glauberman--Lynd]\label{ExistenceUniquenessCentric}
 Let $\F$ be a saturated fusion system over $S$. Then there exists a linking locality $(\L,\Delta,S)$ over $\F$ with object set $\Delta=\F^c$, and such a linking locality is unique up to a rigid isomorphism.
\end{theorem}

By Proposition~\ref{ex0}, $(\L,\Delta,S)$ is a linking locality with object set $\Delta=\F^c$ if and only if it is a centric linking system in the sense of Chermak \cite{Chermak:2013}. Hence, Theorem~\ref{ExistenceUniquenessCentric} is a restatement of the main theorem in \cite{Chermak:2013}. The proof in \cite{Chermak:2013} uses the classification of finite simple groups. However, by Theorem~\ref{TLDelta} and Theorem~\ref{GetEveryTransporterSystem}, the statement of Theorem~\ref{ExistenceUniquenessCentric} is equivalent to the existence and uniqueness of centric linking systems which can be proved without the classification of finite simple groups if combining \cite{Oliver:2013} and \cite{Glauberman/Lynd}.

\begin{proof}[Proof of Theorem~\ref{MainThm1}]
Suppose $\F$ is saturated. By Lemma~\ref{subcentricProp}, the set $\F^s$ is closed under taking $\F$-conjugates and overgroups. Hence, it is sufficient to prove (a).  Let $\Delta_0$ be the set of overgroups of the elements of $\F^{cr}$ in $S$. Then $\Delta_0$ is closed under taking $\F$-conjugates, as $\F^{cr}$ is closed under taking $\F$-conjugates.

\smallskip
\noindent{\em Step 1:} We show that, up to a rigid isomorphism, there exists a unique linking locality $(\L_0,\Delta_0,S)$ over $\F$ and the nerve of $\T(\L_0,\Delta)$ is homotopy equivalent to the nerve of a centric linking system. By Theorem~\ref{ExistenceUniquenessCentric}, a linking locality $(\L^*,\F^c,S)$ over $\F$ exists and is unique up to a rigid isomorphism. Then clearly, setting $\L_0:=\L^*|_{\Delta_0}$, the triple $(\L_0,\Delta_0,S)$ is a linking locality. Suppose we are given another linking locality $(\wL_0,\Delta_0,S)$ over $\F$. Then by Theorem~\ref{A1General}, there exists a linking locality $(\wL^*,\F^s,S)$ over $\F$ with $\wL^*|_{\Delta_0}=\wL_0$. Moreover, 
$|\T(\wL^*,\Delta)|\simeq |\T(\wL_0,\Delta_0)|$. Since $(\L^*,\F^c,S)$ is unique up to a rigid isomorphism, there exists then a rigid isomorphism $\lambda:\L^*\rightarrow \wL^*$. Clearly, $\lambda$ restricts to a rigid isomorphism $\L_0\rightarrow\wL_0$. By \cite[Proposition~A.3(b)]{Chermak:2013}, every rigid isomorphism of localities leads to an isomorphism between the corresponding transporter systems. In particular, $|\T(\L^*,\Delta^*)|\simeq |\T(\wL^*,\Delta)|\simeq |\T(\wL_0,\Delta_0)|\simeq |\T(\L_0,\Delta_0)|$.

\smallskip

\noindent{\em Step 2:} We complete the proof by showing that, up to a rigid isomorphism, there exists a unique linking locality $(\L,\Delta,S)$ and $|\T(\L,\Delta)|$ is homotopy equivalent to the nerve of a centric linking system. Note that $\F^{cr}\subseteq \Delta_0\subseteq \Delta\subseteq \F^s$. By Step~1 there is a linking locality $(\L_0,\Delta_0,S)$ which is unique up to rigid isomorphism and $|\T(\L_0,\Delta_0)|$ is homotopy equivalent to the nerve of a centric linking system. By Theorem~\ref{A1General}, there exists a linking locality $(\L,\Delta,S)$ over $\F$ such that $\L|_{\Delta_0}=\L_0$ and $|\T(\L,\Delta)|\simeq|\T(\L_0,\Delta_0)|$ is homotopy equivalent to the nerve of a centric linking system. Moreover, for every linking locality $(\wL,\Delta,S)$, any rigid isomorphism $\L_0\rightarrow \wL|_{\Delta_0}$ extends to a rigid isomorphism $\L\rightarrow \wL$. Let $(\wL,\Delta,S)$ be a linking locality. Note that $(\wL|_{\Delta_0},\Delta_0,S)$ is a linking locality. So by the uniqueness of $\L_0$, there exists a rigid isomorphism $\gamma:\L_0\rightarrow \wL|_{\Delta_0}$. This extends to a rigid isomorphism $\L\rightarrow \wL$ proving that $\L$ is unique up to a rigid isomorphism.
\end{proof}

\section{Centralizers of partial normal subgroups}\label{Centralizers}

In this section we prove Proposition~\ref{CENThm}. Our proof actually gives some further insight; see Proposition~\ref{CNWeaklyClosed} and Remark~\ref{CNWeaklyClosedRemark}.

\begin{lemma}\label{SplittingLemmaCons}
Let $(\L,\Delta,S)$ be a locality and $\N$ a partial normal subgroup of $\L$. Set $T:=S\cap\N$ and let $Q\in\Delta$. Then there exists $x\in\N$ such that $N_S(Q)\leq S_x$ and $N_T(Q^x)\in\Syl_p(N_\N(Q^x))$. 
\end{lemma}

\begin{proof}
By Lemma~\ref{LocalitiesPropG} there exists $g\in\L$ such that $N_S(Q)\leq S_g$ and $N_S(Q^g)\in\Syl_p(N_\L(Q^g))$. As $N_\L(Q^g)$ is a subgroup of $\L$ with normal subgroup $N_\N(Q^g)$, it follows that $N_T(Q^g)=N_S(Q^g)\cap N_\N(Q^g)\in\Syl_p(N_\N(Q^g))$. Chermak \cite[Definition~4.3]{Chermak:2013} defines a reflexive and transitive relation $\uparrow$ on the set $\L\circ\Delta$ of pairs $(f,P)\in\L\times\Delta$ such that $P\leq S_f$. Furthermore, he calls a pair $(f,P)$ maximal in $\L\circ\Delta$ if $(f,P)\uparrow (f',P')$ implies $|P|=|P'|$. The definition of $\uparrow$ yields that $P=S_f$ if $(f,P)$ is maximal under $\uparrow$. As $S$ is finite, we can take $f\in\L$ and $R\in\Delta$ such that $(g,S_g)\uparrow (f,R)$ and $(f,R)$ is maximal under $\uparrow$. Then $R=S_f$, so it follows from  \cite[Proposition~4.5]{Chermak:2013} that $T\leq S_f=R$. Then by \cite[Lemma~4.6]{Chermak:2013}, there exists $x\in\N$ such that $g=xf$, $S_g\leq S_{(x,f)}$ and $N_T(Q^x)\in\Syl_p(N_\N(Q^x))$. Then $N_S(Q)\leq S_g=S_{(x,f)}\leq S_x$ and the assertion holds. 
\end{proof}

\begin{prop}\label{CNWeaklyClosed}
 Suppose $(\L,\Delta,S)$ is a linking locality over $\F$. Let $\N$ be a partial normal subgroup of $\L$ and $T=\N\cap S$. Assume that $R$ is a subgroup of $C_S(T)$ which is weakly closed in $\F$. Then the following are equivalent:
\begin{itemize}
\item [(1)] $N_\N(T)\subseteq C_\L(R\cap T)$.
\item [(2)] $N_\N(T)\subseteq C_\L(R)$. 
\item [(3)] $\N\subseteq C_\L(R)$.
\end{itemize}
\end{prop}

\begin{proof}
Assume first that (1) holds. To prove (2) let $x\in N_\N(T)$. We want to show that $x\in C_\L(R)$. Set $P:=S_x$. As $x\in N_\N(T)$, we have $T\leq S_x$. Then by \cite[Lemma~3.1(b)]{Chermak:2015}, $x\in N_\L(P)$ and thus $x\in H:=N_{N_\N(P)}(T)$. So it is sufficient to show that $H\subseteq C_\L(R)$. Recall that $P\in\Delta$ and thus $N_\N(P)\leq N_\L(P)$ is a subgroup of $\L$. In particular, $H$ is a subgroup of $\L$. By \cite[Lemma~3.1(c)]{Chermak:2015}, $T$ is maximal in the poset of $p$-subgroups of $\N$. Thus, as $T\leq H\subseteq \N$, it follows that $T\in\Syl_p(H)$. As $(\L,\Delta,S)$ is of objective characteristic $p$, it follows from \cite[Lemma~3.5]{Chermak:2015} that $N_\N(T)\subseteq N_\L(TC_S(T))$. Since $R\leq C_S(T)$ and $R$ is weakly closed in $\F$, this implies $H\subseteq N_\N(T)\subseteq N_\L(R)$. In particular, $R\leq P$ and the commutator group $[R,H]$ is defined in the group $N_\L(P)$. By (1), $[R\cap T,H]=1$. Moreover, $[R,H]\leq R\cap \N=R\cap T$. Thus, $[R,O^p(H)]=[R,O^p(H),O^p(H)]=1$ and so $O^p(H)\subseteq C_\L(R)$. As $T\in\Syl_p(H)$ and $R\leq C_S(T)$, it follows $H=TO^p(H)\subseteq C_\L(R)$ as required. So we have shown that (1) implies (2). Clearly, (3) implies (1), so it remains only to prove that (2) implies (3). 

\smallskip

Suppose (2) holds and that $\N\not\subseteq C_\L(R)$. Choose $n\in\N$ such that $n\not\in C_\L(R)$ and $P:=S_n$ is of maximal order subject to this property. We proceed in two steps.

\smallskip

\noindent{\em Step~1:} We show that $N_\N(Q)\subseteq C_\L(R)$ for all $Q\in\Delta$ with $|Q|\geq |P|$ and $N_T(Q)\in\Syl_p(N_\N(Q))$. Assuming this is wrong we choose a counterexample $Q$. Then $|Q|=|P|$ because of the maximality of $P$. Set $G:=N_\L(Q)$ and notice that $N:=N_\N(Q)$ is a normal subgroup of $G$. As $N_T(Q)\in\Syl_p(N)$, we have $O_p(N)\leq N_T(Q)$. As $N\leq N_\N(QO_p(N))$, the maximality of $|Q|=|P|$ yields $O_p(N)\leq Q$. As $Q_0:=Q\cap T=Q\cap\N\unlhd N$, it follows $Q_0=O_p(N)$. Since $(\L,\Delta,S)$ is a linking locality, $G=N_\L(Q)$ is of characteristic $p$. So by Lemma~\ref{Charp1}(c), $N$ is of characteristic $p$ and thus $C_{N}(Q_0)\leq Q_0$. Hence, $[N_{C_S(Q_0)}(Q),N]\leq C_N(Q_0)\leq Q_0$ and $QN_{C_S(Q_0)}(Q)$ is normalized by $N$. The maximality of $|Q|=|P|$ yields now $N_{C_S(Q_0)}(Q)\leq Q$. As $QC_S(Q_0)$ is a $p$-group, this implies $C_S(Q_0)\leq Q$. In particular, $R\leq C_S(T)\leq C_S(Q_0)\leq Q$. As $R$ is weakly closed in $\F$, it follows that $R\unlhd G$. By assumption $[R,T]=1$ and $N_T(Q)\in\Syl_p(N)$ yielding $[R,O^{p^\prime}(N)]=[R,\<N_T(Q)^N\>]=1$. If $T\leq Q$ then, as $T$ is strongly closed, $N\leq N_\N(T)$. So (2) would yield that $N\leq N_\N(T)\subseteq C_\L(R)$, contradicting the choice of $Q$. Thus $T\not\leq Q$ and, as $TQ$ is a $p$-group, we have $N_T(Q)\not\leq Q$. Thus, by the maximality of $|Q|=|P|$, $N_N(N_T(Q))\subseteq N_\N(N_T(Q)Q)\subseteq C_\L(R)$. By a Frattini argument, $N=O^{p^\prime}(N)N_N(N_T(Q))\leq C_G(R)\subseteq C_\L(R)$. This contradicts our assumption and thus completes Step~1.

\smallskip

\noindent{\em Step~2:} We derive the final contradiction. By Lemma~\ref{SplittingLemmaCons}, there exists $x\in\N$ such that $N_S(P^n)\leq S_x$ and $N_T(P^{nx})\in\Syl_p(N_\N(P^{nx}))$. If $T\leq P$ then, as $T$ is strongly closed, $n\in N_\N(T)\subseteq C_\L(R)$ contradicting the choice of $n$. Hence, $T\not\leq P$ and $T\not\leq P^n$. In particular, $P^n<N_S(P^n)$ and the maximality of $|P|=|P^n|$ yields that $x\in C_\L(R)$. By Lemma~\ref{LocalitiesProp}(b), conjugation with $nx$ induces a group isomorphism from $N_\L(P)$ to $N_\L(P^{nx})$ and so $N_T(P)^{nx}$ is a $p$-subgroup of $N_\N(P^{nx})$. As $N_T(P^{nx})\in\Syl_p(N_\N(P^{nx}))$, there exists $y\in N_\N(P^{nx})$ such that $(N_T(P)^{nx})^y\leq N_T(P^{nx})$. Then by Lemma~\ref{LocalitiesProp}(c), $N_T(P)^{nxy}=(N_T(P)^{nx})^y\leq N_T(P^{nx})$. As $T\not\leq P$ and $TP$ is a $p$-group, we have $N_T(P)\not\leq P$. Moreover, $N_T(P)P\leq S_{nxy}$. Hence, the maximality of $|P|$ yields $nxy\in C_\L(R)$. By Step~1, $y\in N_\N(P^{nx})\subseteq C_\L(R)$. By Lemma~\ref{NormalizerCentralizerPartialSubgroup}, $C_\L(R)$ is a partial subgroup of $\L$. As $(n,x,y,y^{-1})\in\D$ via $P$, it follows that $nx=(nx)(yy^{-1})=(nxy)y^{-1}\in C_\L(R)$. Similarly, as $x\in C_\L(R)$ and $(n,x,x^{-1})\in\D$ via $P$, $n=n(xx^{-1})=(nx)x^{-1}\in C_\L(R)$. This contradicts the choice of $n$ and gives thus the final contradiction. 
\end{proof}

\begin{proof}[Proof of Proposition~\ref{CENThm}]
By assumption $T=\N\cap S$ and $\m{E}=\F_T(\N)$. So it follows immediately that $C_S(\N)\subseteq C_S(\m{E})$. Thus, it is sufficient to show that $R:=C_S(\m{E})$ is contained in $C_S(\N)$, or equivalently, $\N\subseteq C_\L(R)$. By \cite[(6.7)(1)]{Aschbacher:2011}, $R$ is strongly closed in $\F$ and thus weakly closed in $\F$. Furthermore, $R\leq C_S(T)$. As $\m{E}\subseteq C_\F(R)$, $c_n|_{R\cap T}$ is the identity for every $n\in N_\N(T)$, i.e. $N_\N(T)\subseteq C_\L(R\cap T)$. Hence, by Proposition~\ref{CNWeaklyClosed}, $R\leq C_S(\N)$.
\end{proof}

\begin{rmk}\label{CNWeaklyClosedRemark}
 Our arguments show actually that in the situation of Proposition~\ref{CENThm}, the subgroup $C_S(\m{E})=C_S(\N)$ is the largest subgroup of $C_S(T)$ weakly closed in $\F$ such that $N_\N(T)\subseteq C_\L(R\cap T)$. Similarly, it is the largest subgroup of $C_S(T)$ strongly closed in $\F$ such that $N_\N(T)\subseteq C_\L(R\cap T)$.
\end{rmk}

\section{Maps between linking systems}\label{MapsSection}

Building on earlier work of Puig and many other authors, Aschbacher developed a local theory of fusion systems similar to the local theory of finite groups. Aschbacher's main goal is to find a new and better proof of the classification of finite simple groups via classification theorems for simple fusion systems. There are however some conceptual difficulties, which one could perhaps overcome by working with linking localities. The theory of linking systems plays already a role in Aschbacher's program in places where one needs to consider extensions of fusion systems. 

\smallskip

It seems therefore important to develop a local theory of linking localities which connects to the local theory of fusion systems. More concretely, one wants to answer the following question at least in special cases: 

\begin{question}\label{question}
Suppose $\F'$ and $\F$ are saturated fusion systems over $S'$ and $S$ respectively and $\alpha$ is a morphism of fusion systems from $\F'$ to  $\F$. Does $\alpha$ correspond to a suitable ``map'' (however defined) from a linking locality over $\F'$ to a linking locality over $\F$, or from a linking systems associated to $\F'$ to a linking system associated to $\F$? 
\end{question}

It doesn't seem that there is an affirmative answer to this question in general, even if we are flexible in the choice of objects of linking systems and linking localities. However, in this section we will see that the answer to the above question is yes in the following important special cases:
\begin{itemize}
 \item[(1)] The morphism $\alpha$ is surjective and the kernel of $\alpha$ is a central subgroup of $\F'$.
 \item[(2)] $\F'$ is the normalizer or the centralizer in $\F$ of a subgroup of $S$, and $\alpha$ is the inclusion map.
 \item[(3)] $\F'$ is a normal subsystem of $\F$, and $\alpha$ is the inclusion map.
\end{itemize}

\subsection{Quotients modulo central subgroups}

In this subsection, we will study quotients of localities modulo central subgroups contained in $S$. This is needed for the theory of central products of localities developed in \cite{Henke:2017}, which in turn is used in \cite{Chermak/Henke:2017} to prove that there is a one-to-one correspondence between normal subsystems of fusion systems and partial normal subgroups of linking localities. If one were to develop a local theory of linking localities, one would moreover need to consider quotients of linking localities modulo central subgroups to define components of linking localities.

\smallskip

We start by summarizing some crucial facts about quotients of fusion systems modulo central subgroups:

\begin{lemma}\label{QuotientsCentralDeltas}
 Let $\F$ be a saturated fusion system on $S$ and $Z\leq Z(\F)$. Then the following hold for every subgroup $P\leq S$:
\begin{itemize}
 \item [(a)] We have $PZ/Z\in(\F/Z)^s$ if and only if $P\in \F^s$. 
 \item [(b)] We have  $PZ/Z\in(\F/Z)^q$ if and only if $P\in\F^q$. 
 \item [(c)] We have $P\leq \F^{cr}$ if and only if $Z\leq P$ and $P/Z\in(\F/Z)^{cr}$.  
\end{itemize}
\end{lemma}

\begin{proof}
 Part (a) was shown in Lemma~\ref{SubcentricModCentral}. If $Z\leq P$ then it is shown in \cite[Lemma~6.4(b)]{BCGLO2} that $P\in\F^q$ if and only if $P/Z\in(\F/Z)^q$. So for (b) it remains to show that $P\in\F^q$ if and only if $PZ\in\F^q$. As $\F^q$ is closed under taking overgroups, $PZ\in\F^q$ if $P\in\F^q$. Assume now $PZ\in\F^q$. Since $\F^q$ is closed under taking $\F$-conjugates, we may assume that $PZ$ is fully centralized. As $PZ\in\F^q$ it follows that $C_\F(PZ)=\F_{C_S(PZ)}(C_S(PZ))$. Since $Z\leq Z(\F)$, we have $C_S(P)=C_S(PZ)$ and $C_\F(P)=C_\F(PZ)=\F_{C_S(P)}(C_S(P))$. Suppose $Q$ is an $\F$-conjugate of $P$. Then an $\F$-isomorphism $P\rightarrow Q$ extends to an $\F$-isomorphism $PZ\rightarrow QZ$ as $Z\leq Z(\F)$. Hence, $QZ$ is $\F$-conjugate to $PZ$. So, as $PZ$ is fully centralized, $|C_S(Q)|=|C_S(QZ)|\leq |C_S(PZ)|=|C_S(P)|$. Therefore, $P$ is fully centralized in $\F$ and thus quasicentric as $C_\F(P)=\F_{C_S(P)}(C_S(P))$. This shows (b).

\smallskip

If $P\in\F^{cr}$ then $Z\leq C_S(P)\leq P$ as $P$ is centric. Moreover, $P/Z\in (\F/Z)^{cr}$ by \cite[Proposition~3.1]{Kessar/Linckelmann:2008}. Assume now $Z\leq P$ and $P\in(\F/Z)^{cr}$. 
We need to show that $P$ is centric radical in $\F$. Note that $P$ is centric radical in $\F$ if and only if some $\F$-conjugate of $P$ is centric radical in $\F$. Moreover, the $\F/Z$-conjugates of $P/Z$ are precisely the subgroups of the form $Q/Z$ with $Q\in P^\F$, and every $\F/Z$-conjugate of $P/Z$ is centric radical in $\F/Z$. Hence, we can replace $P$ by any $\F$-conjugate of $P$ and will assume without loss of generality that $P$ is fully normalized in $\F$. Then $\Aut_S(P)\in\Syl_p(\Aut_\F(P))$. Set 
\[C:=C_{\Aut_\F(P)}(P/Z).\] 
It follows from the definition of $\F/Z$ that  
\[\Aut_{\F/Z}(P/Z)\cong \Aut_\F(P)/C.\]
As $Z$ is central in $\F$, we have $[Z,\Aut_\F(P)]=1$. In particular, $Z$ is $\Aut_\F(P)$-invariant and thus $C$ is normal in $\Aut_\F(P)$. Moreover, $[P,O^p(C)]=[P,O^p(C),O^p(C)]\leq [Z,O^p(C)]=1$. Thus $O^p(C)=1$ and $C$ is a normal $p$-subgroup of $\Aut_\F(P)$. This implies $C\leq O_p(\Aut_\F(P))$ and $O_p(\Aut_\F(P)/C)=O_p(\Aut_\F(P))/C$. Moreover, as $\Aut_S(P)\in\Syl_p(\Aut_\F(P))$, it follows $C\leq \Aut_S(P)$ and thus $C=C_{\Aut_S(P)}(P/Z)=\Aut_{C_S(P/Z)}(P)$. Since $P/Z$ is centric in $\F/Z$, we have $C_S(P/Z)\leq P$. Thus, $C\leq \Inn(P)$ and $\Inn(P/Z)\cong \Inn(P)/C$. Since $P/Z$ is radical in $\F/Z$, we obtain 
\[O_p(\Aut_\F(P))/C=O_p(\Aut_\F(P)/C)\cong O_p(\Aut_{\F/Z}(P/Z))=\Inn(P/Z)\cong \Inn(P)/C.\] 
As $\Inn(P)\leq O_p(\Aut_\F(P))$, this implies $\Inn(P)=O_p(\Aut_\F(P))$ and $P$ is radical in $\F$. Since $C_S(P)\leq C_S(P/Z)\leq P$ and $P$ is fully normalized in $\F$, $P$ is also centric in $\F$. This shows (c).
\end{proof}

If $(\L,\Delta,S)$ is a locality then every central subgroup $Z\leq Z(\L)$ is a partial normal subgroup of $\L$. Suppose now $(\L,\Delta,S)$ is a linking locality. Then $Z(\L)\leq C_\L(S)\leq S$, so every central subgroup of $\L$ is contained in $S$ and then normal in $\F_S(\L)$. On the other hand, if $Z\leq Z(\F_S(\L))$ then $Z\leq Z(\L)$ by Proposition~\ref{NormLF}. We will now consider quotients modulo subgroups $Z\leq Z(\L)\cap S$. 

\begin{prop}\label{LocalityQuotientModCentral}
Let $(\L,\Delta,S)$ be a locality over $\F$ and $Z\leq Z(\L)\cap S$. Set $\ov{\L}=\L/Z$, $\ov{S}=S/Z$ and $\ov{\Delta}=\{\ov{P}\colon P\in\Delta\}$. 
\begin{itemize}
\item [(a)] The triple $(\ov{\L},\ov{\Delta},\ov{S})$ is a locality over $\F/Z$. 
\item [(b)] The locality $(\ov{\L},\ov{\Delta},\ov{S})$ is of objective characteristic $p$ if and only if $(\L,\Delta,S)$ is of objective characteristic $p$. 
\item [(c)] We have $\F^{cr}\subseteq \Delta$ if and only if $(\F/Z)^{cr}\subseteq \ov{\Delta}$, we have $\Delta=\F^s$ if and only $\ov{\Delta}=(\F/Z)^s$, and we have $\Delta=\F^q$ if and only if $\ov{\Delta}=\F^q$. 
\item [(d)] The locality $(\ov{\L},\ov{\Delta},\ov{S})$ is a linking locality if and only if $(\L,\Delta,S)$ is a linking locality. Similarly, $(\ov{\L},\ov{\Delta},\ov{S})$ is a subcentric linking locality if and only if $(\L,\Delta,S)$ is a subcentric linking locality, and $(\ov{\L},\ov{\Delta},\ov{S})$ is a quasicentric linking locality if and only if $(\L,\Delta,S)$ is a quasicentric linking locality.
\end{itemize}
\end{prop}

\begin{proof}
 It follows from Corollary~\ref{LocalityQuotient} that $(\ov{\L},\ov{\Delta},\ov{S})$ is a locality over $\F/Z$. Let $P\in\Delta$. As $Z\leq Z(\L)$, we have $N_\L(P)\subseteq N_\L(PZ)$ and then $N_\L(P)=N_{N_\L(PZ)}(P)$. So by Lemma~\ref{Charp1}(b), $N_\L(P)$ is of characteristic $p$ if $N_\L(PZ)$ is of characteristic $p$. So for the proof of (b), we can and will assume $Z\leq P$, and we need to show that $N_\L(P)$ is of characteristic $p$ if and only if $N_\L(P)/Z$ is of characteristic $p$. The latter statement is however true by Lemma~\ref{CharpCentralEquiv} since $Z\leq Z(\L)\cap P\leq Z(N_\L(P))\cap O_p(N_\L(P))$. This proves (b). As $Z\leq Z(\F)$, (c) follows from Lemma~\ref{QuotientsCentralDeltas}. Parts (b) and (c) imply (d).
\end{proof}

\begin{prop}
Let $(\L,\Delta,S)$ and $(\L',\Delta',S')$ be localities. Let $\beta$ be a projection of localities from $(\L,\Delta,S)$ to $(\L',\Delta',S')$ with $\ker(\beta)\leq Z(\L)$. 

\smallskip

Then $(\L,\Delta,S)$ is of objective characteristic $p$ if and only if $\ker(\beta)\leq S$ and $(\L',\Delta',S')$ is of objective characteristic $p$. 

\smallskip

Similarly, $(\L,\Delta,S)$ is a linking locality if and only if $\ker(\beta)\leq S$ and $(\L',\Delta',S')$ is a linking locality; $(\L,\Delta,S)$ is a quasicentric linking locality if and only if $\ker(\beta)\leq S$ and $(\L',\Delta',S')$ is a quasicentric linking locality; and $(\L,\Delta,S)$ is a subcentric linking locality if and only if $\ker(\beta)\leq S$ and $(\L',\Delta',S')$ is a subcentric linking locality.
\end{prop}

\begin{proof}
Set $Z:=\ker(\beta)$. If $(\L,\Delta,S)$ is of objective characteristic $p$ then $Z\leq C_\L(S)\leq S$. In particular, this is the case if $(\L,\Delta,S)$ is a linking locality. So assume now $Z\leq S$. Set $\ov{\L}=\L/Z$, $\ov{S}=S/Z$ and $\ov{\Delta}=\{\ov{P}\colon P\in\Delta\}$. Recall that $(\ov{\L},\ov{\Delta},\ov{S})$ is a locality. By \cite[Theorem~4.6]{Chermak:2015}, the induced map $\gamma\colon \ov{\L}\rightarrow \L', Zf\mapsto f\beta$ is an isomorphism of partial groups and a projection from $(\ov{\L},\ov{\Delta},\ov{S})$ to $(\L',\Delta',S')$. Hence, $(\ov{\L},\ov{\Delta},\ov{S})$ is of objective characteristic $p$ if and only $(\L',\Delta',S')$ is of objective characteristic $p$. By Theorem~\ref{LocalityProjection}(b), $\beta|_{\ov{S}}\colon\ov{S}\rightarrow S'$ induces an isomorphism from $\F_{\ov{S}}(\ov{\L})$ to $\F_{S'}(\L')$, and induces thus a bijection from $\F_{\ov{S}}(\ov{\L})^{cr}$ to $\F_{S'}(\L')^{cr}$, from $\F_{\ov{S}}(\ov{\L})^q$ to $\F_{S'}(\L')^q$, and from $\F_{\ov{S}}(\ov{\L})^s$ to $\F_{S'}(\L')^s$. Hence, $(\ov{\L},\ov{\Delta},\ov{S})$ is a linking locality if and only if $(\L',\Delta',S')$ is a linking locality; $(\ov{\L},\ov{\Delta},\ov{S})$ is a quasicentric linking locality if and only if $(\L',\Delta',S')$ is a quasicentric linking locality; and $(\ov{\L},\ov{\Delta},\ov{S})$ is a subcentric linking locality if and only if $(\L',\Delta',S')$ is a subcentric linking locality. So the assertion follows from Proposition~\ref{LocalityQuotientModCentral}(b),(d).
\end{proof}

\subsection{Partial subgroups leading to localities}\label{RestrictionSubsection}

Let $(\L,\Delta,S)$ be a locality. In this section, we show how one can, under certain conditions, construct from a partial subgroup of $\L$ a substructure which is a locality. We use this afterwards to give an affirmative answer to Question~\ref{question} in the case that $\F'$ is the normalizer or the centralizer in $\F$ of a subgroup of $S$, or that $\F'$ is a normal subsystem of $\F$, and that $\alpha$ is the inclusion map.

\smallskip

The partial product on $\L$ will be denoted by $\Pi\colon\D\rightarrow \L$. Let $\H$ be a partial subgroup of $\L$ and set $T:=\H\cap S$. 

\smallskip

Note that, for any $f\in\H$, $(S_f\cap T)^f\in S\cap\H=T$. By $\F_T(\H)$ we denote the fusion system on $T$ generated by the conjugation maps $c_f\colon S_f\cap T\rightarrow T,x\mapsto x^f$. 

\smallskip

Let $\Gamma$ be a non-empty set of subgroups of $T$ closed under taking $\H$-conjugates and overgroups in $T$. Fix $Q\leq S$ and assume that the following properties hold:
\begin{itemize}
 \item [(Q1)] For all $P\in\Gamma$, we have $\<P,Q\>\in\Delta$.
 \item [(Q2)] For all $P_1,P_2\in\Gamma$, we have $N_\H(P_1,P_2)\subseteq N_\L(\<P_1,Q\>,\<P_2,Q\>)$.  
\end{itemize}
Note that (Q2) holds in particular if $\H\subseteq N_\L(Q)$.
Set
\[\H|_\Gamma:=\{f\in\H\colon S_f\cap T\in\Gamma\}.\]
Let $\D_0$ be the set of words $(f_1,\dots,f_n)$ in $\H$ such that there exists $P_0,P_1,\dots,P_n\in\Gamma$ with $P_{i-1}^{f_i}=P_i$ for all $i=1,\dots,n$.\footnote{In particular, we mean here $\emptyset\in\D_0$.}

\begin{lemma}\label{D0InD}
\begin{itemize}
\item [(a)] We have $\D_0\subseteq \D$.
\item [(b)] If $(f_1,\dots,f_n)\in\D_0$ and $P_0,P_1,\dots,P_n\in\Gamma$ with $P_{i-1}^{f_i}=P_i$, then $P_0^{\Pi(f_1,\dots,f_n)}=P_n$. In particular, $\Pi(f_1,\dots,f_n)\in\H|_\Gamma$.
\end{itemize}
\end{lemma}

\begin{proof}
Let $(f_1,\dots,f_n)\in\D_0$ and $P_0,P_1,\dots,P_n\in\Gamma$ such that $P_{i-1}^{f_i}=P_i$ for $i=1,\dots,n$.  Then property (Q1) implies $R_0:=\<P_0,Q\>\in\Delta$. In particular, every $\L$-conjugate of $R_0$ in $S$ is an element of $\Delta$ since $\Delta$ is closed under taking $\L$-conjugates in $S$. So defining $R_i:=R_{i-1}^{f_i}$ for $i=1,\dots,n$,  property (Q2) yields that $R_i$ is well-defined, $R_i\leq \<P_i,Q\>$, and $R_i$ is an element of $\Delta$. Hence, $(f_1,\dots,f_n)\in\D$ via $R_0,R_1,\dots,R_n$. This shows $\D_0\subseteq\D$, so (a) holds. Moreover, by Lemma~\ref{LocalitiesProp}(c), $c_{\Pi(f_1,\dots,f_n)}=c_{f_1}\circ\dots\circ c_{f_n}$ as a map $N_\L(R_0)\rightarrow N_\L(R_n)$. In particular, $P_0\subseteq \D(\Pi(f_1,\dots,f_n))$ and $P_0^{\Pi(f_1,\dots,f_n)}=P_n$. So $P_0\leq S_{\Pi(f_1,\dots,f_n)}\cap T$ and $S_{\Pi(f_1,\dots,f_n)}\cap T\in\Gamma$ as $\Gamma$ is by assumption closed under taking overgroups in $T$. Therefore, $\Pi(f_1,\dots,f_n)\in\H|_\Gamma$. This shows (b).
\end{proof}

\begin{lemma}\label{RestrictionPartialGroup}
The set $\H|_\Gamma$ together with the restriction of the inversion on $\L$ to $\H|_\Gamma$ and with the product $\Pi_0:=\Pi|_{\D_0}\colon \D_0\rightarrow \H|_\Gamma$ forms a partial group. The inclusion map $\H|_\Gamma\rightarrow \L$ is a homomorphism of partial groups.
\end{lemma}

\begin{proof}
By Lemma~\ref{D0InD}, we have $\D_0\subseteq\D$ and $\Pi(v)\in\H|_\Gamma$ for every $v\in\D_0$. Hence, $\Pi_0$ is well-defined. Note that, for every element $f\in \H$, Lemma~\ref{LocalitiesProp}(d) implies $(S_f\cap T)^f\leq S_{f^{-1}}\cap (\H\cap S)=S_{f^{-1}}\cap T$. As $\Gamma$ is closed under taking $\H$-conjugates and overgroups in $T$, it follows that $S_{f^{-1}}\cap T\in\Gamma$ for every $f\in\H|_\Gamma$. Hence, $f^{-1}\in\H|_\Gamma$ for every $f\in\H|_\Gamma$ and the restriction of the inversion on $\L$ to $\H|_\Gamma$ gives us an involutory bijection $\H|_\Gamma\rightarrow \H|_\Gamma$.  

\smallskip

It is immediate from the definition of $\D_0$ that $u,v\in\D_0$ for all $u,v\in\W(\H|_\Gamma)$ with $u\circ v\in\D_0$. Moreover, if $u\circ v\circ w\in\D_0$ then it follows from the definition of $\D_0$ and from Lemma~\ref{D0InD}(b) that $u\circ (\Pi(v))\circ w\in\D_0$. If $w\in\D_0$ then Lemma~\ref{LocalitiesProp}(e) yields $w^{-1}\circ w\in\D_0$. As $\L$ satisfies the axioms of a partial group, it is now easy to observe that $\H|_\Gamma$ together with the partial product $\Pi_0$ and the restriction of the inversion map to $\H|_\Gamma$ satisfies the axioms of a partial group. Since $\D_0\subseteq \D$, it follows moreover that the inclusion map $\H|_\Gamma\rightarrow \L$ is a homomorphism of partial groups. 
\end{proof}

From now on we consider the partial group structure on $\H|_\Gamma$ as defined in the previous lemma. We call this partial group the \textit{restriction} of $\H$ to $\Gamma$. 

\begin{lemma}\label{TRemark}
\begin{itemize}
\item [(a)] We have $\D_0(f)\subseteq \D(f)$ for every $f\in\H|_\Gamma$. The conjugate of an element of $\D_0(f)$ by $f$ in the partial group $\H|_\Gamma$ coincides with the corresponding conjugate in $\L$.
\item [(b)] The subgroup $T$ of $\H$ is also a subgroup of the partial group $\H|_\Gamma$. Moreover, for every $f\in\H|_\Gamma$, we have $S_f\cap T=T_f$ where $T_f:=\{x\in T\colon x\in\D_0(f),\;x^f\in T\}$. 
\end{itemize}
\end{lemma}

\begin{proof}
Part (a) is clear as $\D_0\subseteq\D$ and $\Pi_0=\Pi|_{\D_0}$. So it remains to show (b). As $\Gamma$ is non-empty and closed under taking overgroups in $T$, we have $T\in\Gamma$. If $f\in T$ then  $S_f\cap T=S\cap T=T\in\Gamma$ and so $f\in\H|_\Gamma$. This shows $T\subseteq \H|_\Gamma$. Similarly, if $v=(f_1,\dots,f_n)\in\W(T)$ then $T^{f_i}=T$ and thus $v\in \D_0$. Note that $\Pi_0(v)=\Pi(v)\in T$ as $T$ is a subgroup of $\L$. Thus, $T$ is a subgroup of $\H_0$. 

\smallskip

Let now $f\in\H|_\Gamma$. As $\D_0(f)\subseteq\D(f)$, we have $T_f\leq S_f\cap T$. Let now $x\in S_f\cap T$. As $f\in\H|_\Gamma$, we have $P_1:=S_f\cap T\in\Gamma$. Thus also $P_0:=(S_f\cap T)^f\in\Gamma$. Moreover, $P_0^{f^{-1}}=P_1$, $P_1^x=P_1$ and $P_1^f=P_0$. Therefore, by definition of $\D_0$, $(f^{-1},x,f)\in\D_0$ and so $x\in\D_0(f)$. Moreover, $x^f\in S\cap\H=T$ and thus $x\in T_f$. 
\end{proof}

\begin{lemma}\label{RestrictionLocality}
 Suppose $T$ is a (with respect to inclusion) maximal $p$-subgroup of $\H|_\Gamma$ or of $\H$. Then $(\H|_\Gamma,\Gamma,T)$ is a locality. If $\F_T(\H)$ is saturated and $\F_T(\H)^{cr}\subseteq \Gamma$ then $(\H|_\Gamma,\Gamma,T)$ is a locality over $\F_T(\H)$.
\end{lemma}

\begin{proof}
As $\L$ is a finite set, the set $\H|_\Gamma$ is also finite. By Lemma~\ref{TRemark}, $T$ is a $p$-subgroup of $\H|_\Gamma$. Since $\D_0\subseteq\D$ and $\Pi_0=\Pi|_{\D_0}$,  every subgroup of $\H_0$ is also a subgroup $\H$. Hence, if $T$ is a maximal $p$-subgroup of $\H$, then $T$ is also a maximal $p$-subgroup of the partial group $\H|_\Gamma$. So in any case, our assumption gives that $T$ is a maximal $p$-subgroup of $\H|_\Gamma$. It follows from the definition of $\D_0$ and Lemma~\ref{TRemark} that property (L2) in Definition~\ref{LocalityDefinition} holds for the partial group $\H|_\Gamma$ and the set $\Gamma$ in place of the partial group $\L$ and the set $\Delta$. Using Lemma~\ref{TRemark} one sees also that $\Gamma$ is closed under taking $\H|_\Gamma$-conjugates and overgroups in $T$, as $\Gamma$ is closed under taking $\H$-conjugates and overgroups in $T$. Thus, $(\H|_\Gamma,\Gamma,T)$ is a locality.

\smallskip

Suppose now that $\F_T(\H)$ is saturated and that $\F_T(\H)^{cr}\subseteq\Gamma$. Clearly, by Lemma~\ref{TRemark}, $\F_T(\H|_\Gamma)\subseteq \F_T(\H)$. So it remains to prove that $\F_T(\H)\subseteq \F_T(\H|_\Gamma)$. Let $R\in\F_T(\H)^{cr}$ and $\alpha\in\Aut_{\F_T(\H)}(R)$. By Alperin's fusion theorem, it is sufficient to show that $\alpha$ is a morphism in $\F_T(\H|_\Gamma)$. Since $\F_T(\H)$ is generated by the conjugation maps of the form $c_f\colon S_f\cap T\rightarrow T$ with $f\in \H$, it follows that there exist $f_1,\dots,f_n\in\H$ and subgroups $R=P_0,P_1,\dots,P_n=R$ of $T$ such that $P_{i-1}\subseteq\D(f_i)$, $P_{i-1}^{f_i}=P_i$ and $\alpha=(c_{f_1}|_{P_0})\circ\dots\circ(c_{f_n}|_{P_{n-1}})$. As $R\in\F_T(\H)^{cr}\subseteq \Gamma$ and $\Gamma$ is closed under taking $\H$-conjugates in $T$, it follows that $P_i\in\Gamma$ for $i=0,\dots,n$. Hence, $f_i\in\H|_\Gamma$ for $i=1,\dots,n$. By Lemma~\ref{TRemark}, the conjugation map $c_{f_i}\colon P_{i-1}\rightarrow P_i$ is well-defined in $\H|_\Gamma$ and coincides with the corresponding conjugation map in $\L$. As $\F_T(\H|_\Gamma)$ is generated by conjugation maps between the elements of $\Gamma$, it follows that $\alpha$ is a morphism in $\F_T(\H|_\Gamma)$. 
\end{proof}

\begin{rmk}\label{RestrictionTransporterSystems}
Suppose $T$ is (with respect to inclusion) a maximal $p$-subgroup of $\H$. Consider the transporter systems attached to the localities $(\H|_\Gamma,\Gamma,T)$ and $(\L,\Delta,S)$. Then we can naturally define a functor $\T_\Gamma(\H|_\Gamma)\rightarrow \T_\Delta(\L)$ by sending an object $P\in\Gamma$ to $\<P,Q\>\in\Delta$, and a morphism $(f,P_0,P_1)\in\Hom_{\T_\Gamma(\H|_\Gamma)}(P_0,P_1)$ to $(f,\<P_0,Q\>,\<P_1,Q\>)$ for all $P_0,P_1\in\Gamma$.
\end{rmk}

\subsection{Inclusions of linking systems associated to $p$-local subsystems}\label{pLocalInclusion}

In this section we illustrate how Proposition~\ref{SubcentricProperties1}(c) (or the stronger statement Lemma~\ref{KNormSubcentric2a}) can by used. Roughly speaking, the goal of this section is to show that, for every fully $\F$-normalized subgroup $Q$ of $S$, a subcentric linking locality over $N_\F(Q)$ can be found inside of a subcentric linking locality over $\F$. A similar result holds for centralizers; indeed, we will work more generally with $K$-normalizers. Moreover, we can actually replace subcentric linking localities by centric or quasicentric linking localities in this context. 

\smallskip

We start with some very general results, which will allow us later to apply the general results proved in the previous subsection. Throughout this subsection, let $(\L,\Delta,S)$ be a locality over $\F$. The partial product on $\L$ will be denoted by $\Pi\colon\D\rightarrow \L$. For $Q\in\F$ and $K\leq \Aut(Q)$, set 
\[N_\L^K(Q):=\{f\in N_\L(Q)\colon c_f|_Q\in K\}.\]

\begin{lemma}\label{NLKQMaxpGroup}
The subset $N_\L^K(Q)$ is a partial subgroup of $\L$. If $Q$ is fully $K$-normalized, then $N_S^K(Q)$ is (with respect to inclusion) a maximal $p$-subgroup of $N_\L^K(Q)$. 
\end{lemma}

\begin{proof}
It follows from Lemma~\ref{LocalitiesProp}(e) that $N_\L^K(Q)$ is closed under inversion. Let $v=(f_1,\dots,f_n)\in\D \cap \W(N_\L^K(Q))$. Then $v\in\D$ via some $P_0,\dots,P_n\in\Delta$. Replacing $P_i$ by $\<P_{i-1},Q\>$ we may assume that $Q\leq P_i$. Then by Lemma~\ref{LocalitiesProp}(c), $c_{f_1}\circ\dots\circ c_{f_n}=c_{\Pi(f_1,\dots,f_n)}$ as a map from $N_\L(P_0)$ to $N_\L(P_n)$ and thus in particular as a map from $Q$ to $Q$. As $(c_{f_i})|_Q\in K$ for $i=1,\dots,n$, it follows $c_{\Pi(f_1,\dots,f_n)}|_Q\in K$. Therefore $\Pi(f_1,\dots,f_n)\in N_\L^K(Q)$. This shows that $N_\L^K(Q)$ is a partial subgroup of $\L$. 

\smallskip

Suppose now that $Q$ is fully $K$-normalized. Clearly, $N_S^K(Q)$ is a $p$-subgroup of $N_\L^K(Q)$. Let $R$ be a $p$-subgroup of $N_\L^K(Q)$ such that $N_S^K(Q)\leq R$. By \cite[Proposition~2.11(a)]{Chermak:2015}, every subgroup of $\L$ is contained in the normalizer of some element of $\Delta$. So in particular, $R\leq N_\L(P)$ for some $P\in\Delta$. For such $P$, we have $R\subseteq N_\L(\<P,Q\>)$ and $\<P,Q\>\in\Delta$ as $\Delta$ is closed under taking overgroups in $S$. Hence, we can fix $P\in\Delta$ such that $Q\leq P$ and $R\leq N_\L(P)$. 

\smallskip

As $R$ normalizes $P$, $RP$ is a $p$-subgroup of the finite group $N_\L(P)$. Thus, $RP$ is also a $p$-subgroup of the partial group $\L$. By \cite[Proposition~2.11(b)]{Chermak:2015}, every $p$-subgroup of $\L$ is conjugate to a subgroup of $S$. So  there exists $f\in\L$ such that $RP\subseteq\D(f)$ and $(RP)^f\leq S$. Note that $Q\leq P\leq S_f$. Let $\phi=c_f|_Q\in\Hom_\F(Q,S)$. 

\smallskip

Assume first that $R^f\subseteq N_S^{K^\phi}(Q\phi)$. By Lemma~\ref{LocalitiesProp}(e), we have $|R|=|R^f|$. Recall that $N_S^K(Q)\subseteq R$. As $Q$ is fully $K$-normalized, it follows 
\[|R|=|R^f|\leq |N_S^{K^\phi}(Q\phi)|\leq |N_S^K(Q)|\leq |R|.\]
Hence, equality holds and $N_S^K(Q)=R$. This shows that $N_S^K(Q)$ is a maximal $p$-subgroup, provided we can prove that $R^f\subseteq N_S^{K^\phi}(Q\phi)$.

\smallskip

So it remains to show only that $R^f\subseteq N_S^{K^\phi}(Q\phi)$. For the proof let $r\in R$ and $q\in Q$. We show next that $v:=(f^{-1},r^{-1},f,f^{-1},q,f,f^{-1},r,f)\in\D$. As $R\subseteq\D(f)$, we have $(f^{-1},r,f)\in\D$ via some $P_0,P_1,P_2,P_3\in\Delta$. Recall that $\Delta$ is closed under taking overgroups in $S$. Using Lemma~\ref{LocalitiesProp}(d) and $R\subseteq N_\L^K(Q)\subseteq N_\L(Q)$, we conclude that $(f^{-1},r,f)\in\D$ via $\<P_0,Q^f\>,\<P_1,Q\>,\<P_2,Q\>,\<P_3,Q^f\>$. So we may assume from now on that $Q^f\leq P_0$ and $Q\leq P_1\cap P_2$. Then $q\in P_1$ and thus $P_1^q=P_1$. As $(f^{-1},r,f)\in\D$ via $P_0,P_1,P_2,P_3$, we get the following series of conjugations:
\[P_0\xrightarrow{f^{-1}}P_1\xrightarrow{r}P_2\xrightarrow{f}P_3\mbox{ and }P_3\xrightarrow{f^{-1}}P_2\xrightarrow{r^{-1}}P_1\xrightarrow{f}P_0.\]
Hence, we get the following series of conjugations:
\[P_3\xrightarrow{f^{-1}}P_2\xrightarrow{r^{-1}}P_1\xrightarrow{f}P_0\xrightarrow{f^{-1}}P_1\xrightarrow{q}P_1\xrightarrow{f}P_0\xrightarrow{f^{-1}}P_1\xrightarrow{r}P_2\xrightarrow{f}P_3.\]
So $v\in\D$ via $P_3$. By \cite[Lemma~1.6(b)]{Chermak:2015}, $(r^{-1})^f=(r^f)^{-1}$. So by the axioms of a partial group, $((r^{f})^{-1},q^f,r^f)=((r^{-1})^f,q^f,r^f)\in\D$ and  $(q^f)c_{r^f}=(q^f)^{r^f}=\Pi(v)=(q^f)c_{f^{-1}}c_rc_f=(q^f)\phi^{-1}c_r\phi$. As $q$ is arbitrary, this shows $c_{r^f}|_{Q^f}=\phi^{-1}c_r|_Q\phi$. Since $r\in R\subseteq N_\L^K(Q)$, we have $c_r|_Q\in K$. It follows $c_{r^f}|_{Q\phi}=c_{r^f}|_{Q^f}\in K^\phi$ and thus $r^f\in N_\L^{K^\phi}(Q^f)$. Since $R^f\subseteq S$ and $r\in R$ was arbitrary, this shows $R^f\subseteq N_S^{K^\phi}(Q\phi)$, which is the claim we still had to prove.
\end{proof}

We will now consider centric, quasicentric or subcentric linking localities. More precisely, we will work under the following hypothesis.

\begin{hyp}\label{KNormHyp}
Assume $(\L,\Delta,S)$ is a linking locality over $\F$, and assume $\Delta$ is one of the sets $\F^c$, $\F^q$ or $\F^s$; i.e., $(\L,\Delta,S)$ is a centric, quasicentric or subcentric linking locality. 

\smallskip

Fix $Q\leq S$ and $K\leq \Aut(Q)$, and assume that $Q$ is fully $K$-normalized. Let $\Gamma=N_\F^K(Q)^c$ if $\Delta=\F^c$, $\Gamma=N_\F^K(Q)^q$ if $\Delta=\F^q$, and $\Gamma=N_\F^K(Q)^s$ if $\Delta=\F^s$. 
\end{hyp}

Assuming Hypothesis~\ref{KNormHyp}, recall from \cite[Theorem~I.5.5]{Aschbacher/Kessar/Oliver:2011} that $N_\F^K(Q)$ is saturated as $Q$ is fully $K$-normalized. Our goal will be to show that a locality for $N_\F^K(Q)$ with object set $\Gamma$ is contained in $\L$, and that this locality is a linking locality if $K\cap \Aut_\F(Q)$ is subnormal in $\Aut_\F(Q)$. So if $K\cap \Aut_\F(Q)$ is subnormal in $\Aut_\F(Q)$, we will show that a centric linking locality over $\F$ contains a centric linking locality over $N_\F^K(Q)$, a quasicentric linking locality over $\F$ contains a quasicentric linking locality over $N_\F^K(Q)$, and a subcentric linking locality over $\F$ contains a subcentric linking locality over $N_\F^K(Q)$. In particular, this is true if $K=\Aut_\F(Q)$ and $N_\F^K(Q)=N_\F(Q)$, or if $K=\{\id_Q\}$ and $N_\F^K(Q)=C_\F(Q)$.

\begin{lemma}\label{KNormLF}
If Hypothesis~\ref{KNormHyp} holds, then $N_\F^K(Q)=\F_{N_S^K(Q)}(N_\L^K(Q))$.
\end{lemma}

\begin{proof}
Assume Hypothesis~\ref{KNormHyp}. Clearly $\F_{N_S^K(Q)}(N_\L^K(Q))\subseteq N_\F^K(Q)$, so it remains to prove the converse inclusion. Let $P\in N_\F^K(Q)^{cr}$ and $\phi\in\Aut_{N_\F^K(Q)}(P)$. As $N_\F^K(Q)$ is saturated, it is by Alperin's fusion theorem \cite[Theorem~I.3.6]{Aschbacher/Kessar/Oliver:2011} sufficient to prove that $\phi$ is a morphism in $\F_{N_S^K(Q)}(N_\L^K(Q))$. By definition of $N_\F^K(Q)$, $\phi$ extends to $\hat{\phi}\in\Hom_\F(PQ)$ with $\hat{\phi}|_Q\in K$. By Lemma~\ref{KNormSubcentric2a}, $PQ\in\Delta$. So by Lemma~\ref{MorphismAsConjugation}, $\hat{\phi}=c_f|_{PQ}$ for some $f\in\L$ with $PQ\leq S_f$. Then $c_f|_Q=\hat{\phi}|_Q\in K$ and thus $f\in N_\L^K(Q)$. Hence, $\phi=\hat{\phi}|_P=c_f|_P$ is a morphism in $\F_{N_S^K(Q)}(N_\L^K(Q))$. 
\end{proof}

Assume Hypothesis~\ref{KNormHyp}. Recall from Lemma~\ref{NLKQMaxpGroup} that $N_\L^K(Q)$ is a partial subgroup of $\L$. Note that $\Gamma$ is closed under taking $\F$-conjugates and overgroups in $N_S^K(Q)$; if $\Gamma=N_\F^K(Q)^c$ or $\Gamma=N_\F^K(Q)^q$, then this is easy to check from the definitions; if $\Gamma=N_\F^K(Q)^s$, then this follows from  Proposition~\ref{subcentricProp}. By Lemma~\ref{KNormSubcentric2a}, we have 
\[PQ\in\Delta\mbox{ for every }P\in\Gamma.\]
In particular, property (Q1) from Subsection~\ref{RestrictionSubsection} holds. Moreover, clearly $N_\L^K(Q)\subseteq N_\L(Q)$ and so property (Q2) from Subsection~\ref{RestrictionSubsection} holds with $N_\L^K(Q)$ in place of $\H$. Set 
\[\L_0:=N_\L^K(Q)|_\Gamma=\{f\in N_\L^K(Q)\colon S_f\cap N_S^K(Q)\in\Gamma\}.\] 
Let $\D_0$ be the set of words $(f_1,\dots,f_n)\in\W(\L_0)$ for which there exist $P_0,P_1,\dots,P_n\in\Gamma$ with $P_{i-1}^{f_i}=P_i$ for all $i=1,\dots,n$.

\begin{lemma}
Assume Hypothesis~\ref{KNormHyp} and let $\L_0$ be as above. Then $\L_0$ together with the restriction of the inversion on $\L$ to $\L_0$ and with the product $\Pi_0:=\Pi|_{\D_0}\colon \D_0\rightarrow \L_0$ forms a partial group.

\smallskip

Regarding $\L_0$ as a partial group in this way, the inclusion map $\L_0\rightarrow \L$ is a homomorphism of partial groups and the triple $(\L_0,\Gamma,N_S^K(Q))$ is a locality over $N_\F^K(Q)$. If $K\cap \Aut_\F(Q)$ is subnormal in $\Aut_\F(Q)$ then $(\L_0,\Gamma,N_S^K(Q))$ is a linking locality. 
\end{lemma}

\begin{proof}
We have argued above that (Q1) and (Q2) hold with $N_\L^K(Q)$ in place of $\H$. So by Lemma~\ref{RestrictionPartialGroup}, $\L_0=N_\L^K(Q)|_\Gamma$ forms a partial group as described above, and the inclusion map $\L_0\rightarrow\L$ is a homomorphism of partial groups. By Lemma~\ref{NLKQMaxpGroup}, $N_S^K(Q)=N_\L^K(Q)\cap S$ is a maximal $p$-subgroup of $N_\L^K(Q)$, and by Lemma~\ref{KNormLF}, $\F_{N_S^K(Q)}(N_\L^K(Q))=N_\F^K(Q)$ is saturated. Moreover, by our choice of $\Gamma$, we have $N_\F^K(Q)^{cr}\subseteq \Gamma$. Hence, by Lemma~\ref{RestrictionLocality}, $(\L_0,\Gamma,N_S^K(Q))$ is a locality over $N_\F^K(Q)$. 

\smallskip

Assume now that $K\cap \Aut_\F(Q)$ is subnormal in $\Aut_\F(Q)$. Note that $N_\F^{K\cap\Aut_\F(Q)}(Q)=N_\F^K(Q)$ and $N_\L^{K\cap\Aut_\F(Q)}(Q)=N_\L^K(Q)$. So replacing $K$ by $K\cap \Aut_\F(Q)$, we may assume that $K$ is subnormal in $\Aut_\F(Q)$. 

\smallskip

Let $P\in\Gamma$. We need to show that $N_{\L_0}(P)$ is of characteristic $p$. A priori, $N_{\L_0}(P)$ means here the normalizer of $P$ formed in the partial group $\L_0$. However, by Lemma~\ref{TRemark}, this normalizer coincides with the normalizer in $\L_0$ of $P$ formed in the partial group $\L$. Moreover,  $N_{\L_0}(P)=N_{N_\L^K(Q)}(P)$. Recall that $PQ\in\Delta$. So since $(\L,\Delta,S)$ is a linking locality, $G:=N_\L(PQ)$ is a group of characteristic $p$.  Observe that
$N_{\L_0}(P)=N_{N_\L^K(Q)}(P)=N_{N_G(P)}^K(Q)$. As $G$ is of characteristic $p$, it follows from Lemma~\ref{Charp1}(b) that $N_G(P)$ is of characteristic $p$, and thus $N_{N_G(P)}(Q)$ is of characteristic $p$. As $K$ is subnormal in $\Aut_\F(Q)$, $K_0:=\Aut_G(Q)\cap K$ is subnormal in $\Aut_G(Q)\leq \Aut_\F(Q)$. Let $K_0\unlhd K_1\unlhd\dots\unlhd K_n=\Aut_G(Q)$ be a subnormal series. Then $N_{\L_0}(P)=N_{N_G(P)}^K(Q)=N_{N_G(P)}^{K_0}(Q)\unlhd N_{N_G(P)}^{K_1}(Q)\unlhd \dots \unlhd N_{N_G(P)}^{K_n}(Q)=N_{N_G(P)}(Q)$ and $N_{\L_0}(P)$ is a subnormal subgroup of $N_{N_G(P)}(Q)$. Thus, $N_{\L_0}(P)$ is of characteristic $p$ by Lemma~\ref{Charp1}(c).     
\end{proof}

\begin{rmk}
As described more generally in Remark~\ref{RestrictionTransporterSystems}, with the set-up as above, we are naturally given a functor $\T(\L_0,\Delta_0)\rightarrow \T(\L,\Delta)$ which is injective on morphism sets. It sends an object $P\in\Gamma$ to $PQ\in\Delta$, and a morphism $(f,P_1,P_2)$ to $(f,P_1Q,P_2Q)$. So if $K\cap \Aut_\F(Q)$ is subnormal in $\Aut_\F(Q)$ then we get functors from the centric linking system of $N_\F^K(Q)$ to the centric linking system of $\F$, from the quasicentric linking system of $N_\F^K(Q)$ to the quasicentric linking system of $\F$, and from the subcentric linking system of $N_\F^K(Q)$ to the subcentric linking system of $\F$. 

\smallskip

In particular, this holds if $K=\Aut(Q)$ and $N_\F^K(Q)=N_\F(Q)$, or if $K=\{\id_K\}$ and $N_\F^K(Q)=C_\F(Q)$. So if $Q$ is fully normalized, we get a functor from the subcentric linking system associated to $N_\F(Q)$ to the subcentric linking system associated to $\F$; if $Q$ is fully centralized, we get a functor from the subcentric linking system associated to $C_\F(Q)$ to the subcentric linking system associated to $\F$.  Corresponding statements hold with subcentric replaced by centric or quasicentric.

\smallskip

If $N_\F^K(Q)=N_\F(Q)$, $\Gamma=N_\F(Q)^c$ and $\Delta=\F^c$ then note that $\T(\L_0,\Delta_0)$ is a centric linking system associated to $N_\F(Q)$. In \cite[Definition~6.1]{BLO2} a centric linking system associated to $N_\F(Q)$ was directly constructed from a centric linking system associated to $\F$. So our construction of $(\L_0,\Delta_0,S_0)$ can be seen as a locality version of this construction.
\end{rmk}

\subsection{Inclusions of linking systems associated to normal subsystems}\label{NormalInclusions}
Let $(\L,\Delta,S)$ be a subcentric linking locality over $\F$, i.e. $\Delta=\F^s$. Suppose $\m{E}$ is a normal subsystem of $\F$ over $T\leq S$. In this section, we will explain how Proposition~\ref{CENThm} and Theorem~\ref{SubcentricEF} can be used to show that a subcentric linking locality over $\m{E}$ is contained in $\L$ in a way that the inclusion map is a homomorphism of partial groups. We need to assume however that $\m{E}$ is realized by a partial normal subgroup of $\L$, an assumption which Chermak and the author of this paper \cite{Chermak/Henke:2017} show to be redundant. So we assume for the remainder of this section that there exists a partial normal subgroup $\N$ of $\L$ such that $T=\N\cap S$ and $\F_T(\N)=\m{E}$. Set $\Gamma:=\m{E}^s$ and $\N_0:=\N|_\Gamma=\{f\in\N\colon S_f\cap T\in\Delta_0\}$. 

\smallskip

By Proposition~\ref{CENThm}, $Q:=C_S(\m{E})=C_S(\N)$. So by Theorem~\ref{SubcentricEF}, we have $PQ=PC_S(\N)\in\Delta$ for all $P\in\Gamma$. Note moreover that $\N\subseteq N_\L(Q)$. Hence, properties (Q1) and (Q2) from Subsection~\ref{RestrictionSubsection} hold with $\N$ in place of $\H$. Let $\D_0$ be the set of words $(f_1,\dots,f_n)\in\W(\N_0)$ such that there exist $P_0,\dots,P_n\in\Gamma$ with $P_{i-1}^{f_i}=P_i$ for $i=1,\dots,n$.

\begin{lemma}
The set $\N_0$ together with the restriction of the inversion on $\L$ to $\N_0$ and with the product $\Pi_0:=\Pi|_{\D_0}\colon \D_0\rightarrow \N_0$ forms a partial group.

\smallskip

Regarding $\N_0$ as a partial group in this way, the inclusion map $\N_0\rightarrow \L$ is a homomorphism of partial groups and the triple $(\N_0,\Gamma,T)$ is a subcentric linking locality over $\m{E}$. 
\end{lemma}

\begin{proof}
By Lemma~\ref{RestrictionPartialGroup}, $\N_0$ forms a partial group in the way described, and the inclusion map $\N_0\rightarrow\L$ is a homomorphism of partial groups. 
By \cite[Lemma~3.1(c)]{Chermak:2015}, $T$ is a maximal $p$-subgroup of $\N$. So as $\F_T(\N)=\m{E}$ is saturated and $\m{E}^{cr}\subseteq \m{E}^s=\Gamma$, it follows from Lemma~\ref{RestrictionLocality} that $(\N_0,\Gamma,T)$ is a locality over $\m{E}$. Let $P\in\Gamma$. We need to show that $N_{\N_0}(P)$ is of characteristic $p$. A priori, $N_{\N_0}(P)$ denotes here the normalizer in the partial group $\N_0$, but by Lemma~\ref{TRemark} and the definition of the partial group $\N_0$, this normalizer coincides with the normalizer $N_{\N_0}(P)$ in $\L$ and $N_{\N_0}(P)=N_\N(P)$. Recall that $PQ\in\Delta$ for $Q:=C_S(\N)$. So $G:=N_\L(PQ)$ is a finite group of characteristic $p$. Note that $N_\N(P)=N_G(P)\cap \N$ is a normal subgroup of $N_G(P)$. Hence, by Lemma~\ref{Charp1}(b),(c), $N_\N(P)$ is of  characteristic $p$. This implies the assertion. 
\end{proof}

\begin{rmk}
As described more generally in Remark~\ref{RestrictionTransporterSystems} we get a functor $\T(\N_0,\Gamma)\rightarrow \T(\L,\Delta)$ which sends an object $P\in\Gamma$ to $PC_S(\N)\in\Delta$ and a morphism $(f,P_1,P_2)$ to $(f,P_1C_S(\N),P_2C_S(\N))$. So there is a functor from the subcentric linking system of $\m{E}$ to the subcentric linking system of $\F$ which is injective on the morphism sets.
\end{rmk}

\begin{rmk}
 Chermak \cite{Chermak:2017} defined a set $\delta(\F)$ of subgroups of $S$ which is closed under taking $\F$-conjugates and overgroups in $\F$. He calls a linking locality over $\F$ with object set $\delta(\F)$ a regular locality. If $(\L,\Delta,S)$ is a regular locality rather than a subcentric linking locality, and the assumption is otherwise as above, Chermak showed that $(\N,\delta(\m{E}),T)$ is a regular locality. So for regular localities, it is not necessary to consider the restriction $\N|_\Gamma$. This observation allows Chermak to introduce components of regular localities in a relatively smooth way. However, results analogous to the results in Subsection~\ref{pLocalInclusion} do not hold for regular localities. Therefore, if one wants to develop a local theory of localities which is useful for Aschbacher's program, it might still be better to work with subcentric linking localities rather than with regular localities. 
\end{rmk}

\section{Localities of objective characteristic $p$ coming from finite groups}\label{GroupLocalities}

\noindent\textbf{Throughout let $G$ be a finite group and $S\in\Syl_p(G)$.} 

\smallskip

In this section we present some ways of constructing linking localities and localities of objective characteristic $p$ from $G$. At the end, in Remark~\ref{ParabolicChar}, we speculate how these constructions can be used for a unifying approach to a classification of groups and fusion systems of parabolic characteristic $p$. In general, localities can be constructed from $G$ as follows

\begin{definition}
Let $\Gamma$ be a set of subgroups of $S$ closed under taking $G$-conjugates and overgroups in $S$. Let $\L_\Gamma(G)$ be the set of all elements $g\in G$ with $S\cap S^g\in\Gamma$. Moreover, let $\D_\Gamma$ be the set of all words $(g_1,\dots,g_n)$ such that $g_i\in G$ and there exist elements $P_0,\dots,P_n\in\Gamma$ with $P_{i-1}^{g_i}=P_i$ for $i=1,\dots,n$. Define a partial product $\Pi\colon \D\rightarrow \L_\Gamma(G)$ by mapping $(g_1,g_2,\dots,g_n)\in\D$ to the product $g_1g_2\dots g_n$ in $G$. Define an inversion map on $\L_\Gamma(G)$ by taking the restriction of the inversion map on $G$ to the set $\L_\Gamma(G)$. 
\end{definition}

By \cite[Example/Lemma~2.10]{Chermak:2013}, $(\L_\Gamma(G),\Gamma,S)$ is a locality. In this section we are mainly concerned with the case that $\Gamma$ is one of the sets $\Delta$ and $\Delta^*$ defined as follows: 
\[\Delta:=\{P\leq S\colon N_G(P)\mbox{ is of characteristic }p\}\]
and 
\[\Delta^*:=\{P\leq S\colon N_G(P)\mbox{ is almost of characteristic }p\}.\]

By Lemma~\ref{AlmostCharpNormCent}, we have 
\[\Delta=\{P\leq S\colon C_G(P)\mbox{ is of characteristic }p\}\] 
and 
\[\Delta^*:=\{P\leq S\colon C_G(P)\mbox{ is almost of characteristic }p\}.\]

\begin{lemma}
The sets $\Delta$ and $\Delta^*$ are closed under $\F_S(G)$-conjugation and taking overgroups in $S$. In particular, if $\Delta\neq \emptyset$ then $S\in\Delta$.
\end{lemma}

\begin{proof}
Clearly, $N_G(P^g)=N_G(P)^g\cong N_G(P)$ for any $g\in G$ and $P\leq S$. Hence, $\Delta$ and $\Delta^*$ are closed under $\F_S(G)$-conjugation. Let $P\leq R\leq S$. We want to show that $R\in\Delta$ if $P\in\Delta$, and $R\in\Delta^*$ if $P\in\Delta^*$. Since $R$ is a $p$-group, $P$ is subnormal in $R$. Thus, by induction on the subnormal length of $P$ in $R$, we may assume that $P\unlhd R$. Then $R\leq H:=N_G(P)$ and $C_G(R)=C_H(R)$. If $P\in\Delta$ then $H$ is of characteristic $p$, and if $P\in\Delta^*$ then $H$ is almost of characteristic $p$. So if $P\in\Delta$ then, by Lemma~\ref{Charp1}(b), $C_G(R)=C_H(R)$ is of characteristic $p$  and thus $R\in\Delta$ by Lemma~\ref{AlmostCharpNormCent}. Similarly, if $P\in\Delta^*$ then $C_G(R)=C_H(R)$ is almost of characteristic $p$ by Lemma~\ref{AlmostCharp1}(b), and thus $R\in\Delta^*$ by Lemma~\ref{AlmostCharpNormCent}.
\end{proof}

\begin{lemma}
We have $\Delta\subseteq\Delta^*\subseteq\F_S(G)^s$.
\end{lemma}

\begin{proof}
Since every group of characteristic $p$ is almost of characteristic $p$, we have $\Delta\subseteq\Delta^*$. If $P\in\Delta^*\cap \F_S(G)^f$ then $N_{\F_S(G)}(P)=\F_{N_S(P)}(N_G(P))\cong \F_{\ov{N_S(P)}}(\ov{N_G(P)})$ for $\ov{N_G(P)}=N_G(P)/O_{p^\prime}(N_G(P))$. Hence, $N_{\F_S(G)}(P)$ is constrained by Theorem~\ref{Model1}(a) and $P\in\F_S(G)^s$ by Lemma~\ref{subcentricEquiv}. Since $\Delta^*$ and $\F_S(G)^s$ are closed under $\F_S(G)$-conjugation, it follows $\Delta^*\subseteq \F_S(G)^s$. 
\end{proof}

Clearly, the sets $\Delta$ and $\Delta^*$ can be different in general, since there are groups which are almost of characteristic $p$, but not of characteristic $p$. The next example shows that $\Delta^*$ is not equal to $\F_S(G)^s$ in general. 

\begin{ex}\label{exSubcentricGroup}
Let $P$ be a $P$-group, $G=P\times A_5$ and $S\in\Syl_p(G)$. As $P=O_p(G)$, $G$ is not of characteristic $p$. Since $O_{p^\prime}(G)=1$, it follows that $P\not\in\Delta^*$. However, $S$ is normal in $\F_S(G)=\F_S(P\times A_4)$ and thus $\F_S(G)$ is constrained. Hence, as $P$ is normal in $\F_S(G)$, we have $P\in\F_S(G)^s$ by Lemma~\ref{subcentricEquiv}.
\end{ex}

\begin{lemma}\label{Delta*}
 We have $\F_S(G)^q\subseteq \Delta^*$. In particular, $\F_S(G)^c\subseteq\Delta^*$ and $S\in\Delta^*$. 
\end{lemma}

\begin{proof}
Since both $\F_S(G)^q$ and $\Delta^*$ are closed under $\F_S(G)$-conjugation, it is sufficient to prove that $P\in\Delta^*$ for every $P\in\F_S(G)^{fq}$. Let $P\in\F_S(G)^{fq}$. Then $N_S(P)\in\Syl_p(N_G(P))$ and $\F_{C_S(P)}(C_G(P))=C_{\F_S(G)}(P)=\F_{C_S(P)}(C_S(P))$ by \cite[Proposition~I.5.4]{Aschbacher/Kessar/Oliver:2011}. Hence, Lemma~\ref{Charp2} applied with $N_G(P)$ in place of $G$ gives $C_G(P)=C_S(P)O_{p^\prime}(C_G(P))$. So $C_G(P)/O_{p^\prime}(C_G(P))$ is a $p$-group and thus $C_G(P)$ is almost of characteristic $p$. Therefore, $P\in\Delta^*$ by Lemma~\ref{AlmostCharpNormCent}.  
\end{proof}

\begin{ex}
Let $P$ be an abelian $p$-group and $G=P\times S_4$. Then $G$ is of characteristic $p$ and thus $P\in\Delta\subseteq\Delta^*$. However, as $P\leq Z(G)$, $P$ is fully $\F_S(G)$-centralized and $C_{\F_S(G)}(P)=\F_S(G)\neq\F_S(S)$. Therefore, $P\not\in\F_S(G)^q$.
\end{ex}

\begin{lemma}\label{GetGroupLocalities}
\begin{itemize}
\item [(a)] The locality $(\L_{\Delta}(G),\Delta,S)$ is of objective characteristic $p$. 
\item [(b)] Set $\Theta:=\bigcup_{P\in\Delta^*}O_{p^\prime}(N_G(P))$. Then $\Theta$ is a partial normal subgroup of $\L_{\Delta^*}(G)$ with $S\cap\Theta=1$, the canonical map $\rho\colon \L_{\Delta^*}(G)\rightarrow \L_{\Delta^*}(\L)/\Theta$ restricts to an isomorphism $S\rightarrow S\rho$ and, upon identifying $S$ with $S\rho$,  the locality $(\L_{\Delta^*}(G)/\Theta,\Delta^*,S)$ is a linking locality over $\F_S(G)$.
\end{itemize}
\end{lemma}

\begin{proof}
Note that, for every $P\in\Delta$, $N_{\L_\Delta(G)}(P)=N_G(P)$ is of characteristic $p$. So the locality $(\L_{\Delta}(G),\Delta,S)$ is of objective characteristic $p$. Similarly, for every $P\in\Delta^*$, $N_G(P)=N_{\L_{\Delta^*}(G)}(P)$ is almost of characteristic $p$. Thus, it follows from Proposition~\ref{GetLocalityObjectiveCharp} that $\Theta$ is a partial normal subgroup of $\L_{\Delta^*}(G)$ with $S\cap\Theta=1$, the canonical map $\rho\colon \L_{\Delta^*}(G)\rightarrow \L_{\Delta^*}(\L)/\Theta$ restricts to an isomorphism $S\rightarrow S\rho$ and, upon identifying $S$ with $S\rho$,  $(\L_{\Delta^*}(G)/\Theta,\Delta,S)$ is a locality over $\F_S(\L_{\Delta^*}(G))$ of objective characteristic $p$. By Lemma~\ref{Delta*}, $\F_S(G)^{cr}\subseteq\Delta^*$. Hence, by Alperin's fusion theorem, $\F_S(G)=\F_S(\L_{\Delta^*}(G))$. So  $(\L_{\Delta^*}(G)/\Theta,\Delta,S)$ is a linking locality over $\F_S(G)$. 
\end{proof}

\begin{rmk}
 Let $\Gamma$ be one of the sets $\F_S(G)^c$ or $\F_S(G)^q$, or assume more generally that $\Gamma$ is a subset of $\Delta^*$ such that $\F_S(G)^{cr}\subseteq\Gamma$ and $\Gamma$ is closed under taking $\F_S(G)$-conjugates and overgroups in $S$. Set $\Theta:=\bigcup_{P\in\Gamma}O_{p^\prime}(N_G(P))$. Then a statement similar to the one in Lemma~\ref{GetGroupLocalities}(b) holds:

\smallskip

The set $\Theta$ is a partial normal subgroup of $\L_{\Gamma}(G)$ with $S\cap\Theta=1$, the canonical map $\rho\colon \L_{\Gamma}(G)\rightarrow \L_{\Gamma}(\L)/\Theta$ restricts to an isomorphism $S\rightarrow S\rho$ and, upon identifying $S$ with $S\rho$,  the locality $(\L_{\Gamma}(G)/\Theta,\Gamma,S)$ is a linking locality over $\F_S(G)$.
\end{rmk}

To our knowledge there is no good way of constructing the subcentric linking system of $\F_S(G)$ directly from the group $G$.  However, using Theorem~\ref{A1General}, one can extend $\L_{\Delta^*}(G)/\Theta$ to a subcentric linking locality over $\F_S(G)$. It seems moreover that the localities $\L_{\Delta}(G)$ and $\L_{\Delta^*}(G)$ are good enough for most purposes. As we explain in the following remark, these localities could be useful if one wants to deduce statements about groups of parabolic characteristic $p$ from similar statements about localities.

\begin{rmk}\label{ParabolicChar}
Many classification theorems in the program of Meierfrankenfeld, Stellmacher and Stroth are proved not only for groups of local characteristic $p$, but more generally for groups of \textit{parabolic characteristic $p$}. These are finite groups where every $p$-local subgroup containing a Sylow $p$-subgroup is of characteristic $p$. We say similarly that $\F$ is of parabolic characteristic $p$ if $\F$ is saturated and the normalizer of every normal subgroup of $S$ is constrained. If $G$ is of parabolic characteristic $p$, note that $\Delta$ and $\Delta^*$ contain every non-trivial normal subgroup of $S$. Similarly, if $\F$ is of parabolic characteristic $p$, then $\F^s$ contains every non-trivial normal subgroup of $S$.

\smallskip

As pointed out in the introduction, it might be possible to give a unifying approach to the classification of fusion systems of characteristic $p$-type and of groups of characteristic $p$-type whilst avoiding to use Theorem~\ref{MainThm1} and the theory of fusion systems to prove classification theorems for groups of characteristic $p$-type. Similarly, such an approach could presumably be implemented for groups and fusion systems of parabolic characteristic $p$ if one proceeds roughly as follows:  In a first step one proves a classification theorem for a locality $(\L,\Gamma,S)$ of objective characteristic $p$, where $\Gamma$ contains every non-trivial normal subgroup of $S$. Then in a second step one separately deduces from that a corresponding classification theorem for fusion systems of parabolic characteristic $p$ (using the existence of subcentric linking systems), and for groups of parabolic characteristic $p$ (working with the locality $(\L_\Delta(G),\Delta,S)$ with the set $\Delta$ as above). If this approach turns out to be problematic, one could also in the first step prove a classification theorem for a linking locality $(\L,\Gamma,S)$ (over a saturated fusion system) where $\Gamma$ includes every non-trivial normal subgroup of $S$. Then one would work with the locality $(\L_{\Delta^*}(G)/\Theta,\Delta^*,S)$ to deduce the corresponding classification theorem for groups of parabolic characteristic $p$. Working with $(\L_{\Delta^*}(G)/\Theta,\Delta^*,S)$ has here not only the advantage that $(\L_{\Delta^*}(G)/\Theta,\Delta^*,S)$ is a linking locality, but also that its fusion system is isomorphic to $\F_S(G)$.
\end{rmk}

\bibliographystyle{amsplain}
\bibliography{repcoh}

\end{document}